%% file: main.tex
\documentclass[12pt,reqno]{amsproc}
\input{preamble.tex}

\newcommand{\AS}[1]{\textcolor{red}{#1}}

\title{Multilinear Hyperquiver Representations}

\author{Tommi Muller}\address{Mathematical Institute, University of Oxford, OX2 6GG, UK}\email{tommi.muller@maths.ox.ac.uk} 
\author{Vidit Nanda}\address{Mathematical Institute, University of Oxford, OX2 6GG, UK}\email{nanda@maths.ox.ac.uk}
\author{Anna Seigal}\address{Harvard University, School Of Engineering And Applied Sciences, 29 Oxford Street, Cambridge, MA 02138
USA}\email{aseigal@seas.harvard.edu}

\begin{document}
\captionsetup[subfloat]{labelfont=normalfont}
\begin{abstract}

We count singular vector tuples of a system of tensors assigned to the edges of a directed hypergraph. To do so, we study the generalisation of quivers to directed hypergraphs. Assigning vector spaces to the nodes of a hypergraph and multilinear maps to its hyperedges gives a hyperquiver representation. 
Hyperquiver representations generalise quiver representations (where all hyperedges are edges) and tensors (where there is only one multilinear map).
The singular vectors of a hyperquiver representation are a compatible assignment of vectors to the nodes. 
We compute the dimension and degree of the variety of singular vectors of a sufficiently generic hyperquiver representation. Our formula
specialises to 
known results that count the
singular vectors and eigenvectors of a generic tensor.
Lastly, we study a hypergraph generalisation of the inverse tensor eigenvalue problem and solve it algorithmically.

\bigskip


\bigskip


\bigskip
\bigskip
\medskip
\medskip

\begin{center}
\textit{Communicated by JM Landsberg}
\end{center}

\vspace{-2em}

\end{abstract}

\clearpage\maketitle
\thispagestyle{empty}

\section{Introduction}

The theory of quiver representations provides a unifying framework for some fundamental concepts in linear algebra \cite{bernstein1973coxeter}. In this paper, we introduce and study a natural generalisation of quiver representations, designed to analogously serve the needs of multilinear algebra.

\bigskip

{\bf Quiver Representations and Matrix Spectra.} A {\em quiver} $Q$ consists of finite sets $V$ and $E$, whose elements are called vertices and edges respectively, along with two functions $s,t:E \to V$ sending each edge to its source and target vertex. It is customary to write $e:i \to j$ for the edge $e$ with $s(e) = i$ and $t(e) = j$. The definition does not prohibit self-loops $s(e) = t(e)$  nor parallel edges $e_1,e_2:i \to j$. A {\em representation} $(U,\alpha)$ of $Q$ assigns a finite-dimensional vector space $U_i$ to each $i \in V$ and a linear map $\alpha_e:U_i \to U_j$ to each $e:i \to j$ in $E$. Originally introduced by Gabriel to study finite-dimensional algebras~\cite{gabriel}, quiver representations have since become ubiquitous in mathematics. They appear prominently in disparate fields ranging from representation theory and algebraic geometry~\cite{derksen_introduction_2017} to topological data analysis \cite{oudot_persistence_2015}. In most of these appearances, the crucial task is to classify the representations of a given quiver up to isomorphism. This amounts in practice to cataloguing the {\em indecomposable} representations; i.e., those that cannot be expressed as direct sums of smaller nontrivial representations. 

For all but a handful of quivers, the set of indecomposables (up to isomorphism) is complicated, and such a classification is hopeless~\cite{gabriel}.
Nevertheless, one may follow the spirit of~\cite{seigal2022principal} and use quivers to encode compatibility constraints with spectral interpretations. We work with representations that assign vector spaces $U_i = \mathbb{C}^{d_i}$ to each vertex and matrices $A_e: \mathbb{C}^{d_i} \rightarrow \mathbb{C}^{d_j}$ to each edge.
We denote the quiver representation by $(\mathbf{d},A)$, where 
$\mathbf{d} := (d_1,\ldots,d_n)$ is the dimension vector. Let $[\x] \in \P(\mathbb{C}^d)$ denote the projectivisation of a non-zero $\x \in \mathbb{C}^d$. We define the \emph{singular vectors} of a quiver representation $(\mathbf{d},A)$ to consist of tuples $\left( [\x_i] \in \P(\mathbb{C}^{d_i}) \mid i \in V \right)$ for which there exist scalars $\left( \lambda_e \mid e \in E \right)$ so that the compatibility constraint $A_e\x_i = \lambda_e \x_j$ holds across each edge $e:i \to j$. Standard notions from linear algebra arise as special cases of such singular vectors, see also Figure \ref{fig:quiver-examples}:
\begin{enumerate}
    \item[(a)] The eigenvectors of a matrix $A :\C^d \to \C^d$ are the singular vectors of the representation of the Jordan quiver that assigns $\C^d$ to the unique node and $A$ to the unique edge. 
    \item[(b)] The singular vectors of a matrix $A:\C^{d_1} \to \C^{d_2}$ arise from the representation of the directed cycle of length $2$, with $A$ assigned to one edge and $A^\top$ assigned to the other. 
\item[(c)] The generalised eigenvectors of a pair of matrices $A, B: \C^d \to \C^d$ -- i.e., non-zero solutions $\x$ to $A \x = \lambda \cdot B \x$ for some $\lambda \in \C$ -- are the singular vectors of the representation of the Kronecker quiver  with $A$ on one edge and $B$ on the other.
\end{enumerate}
\input{figures/quiver-examples}
For $d = d_1 = d_2$, a generic instance of any of these three quiver representations has $d$ singular vectors.

\bigskip

{\bf Hyperquiver Representations and Tensors.} This century has witnessed progress towards extending the spectral theory of matrices to the multilinear setting of tensors \cite{qi2017tensor}. 
Given a tensor $T \in \CC^{d_1} \otimes \cdots \otimes \CC^{d_m}$, we write $T(\x_1, \ldots, \x_{j-1}\,, \cdot \,, \x_{j+1}, \ldots, \x_{m}) \in \CC^{d_j}$ for the vector with $i$-th coordinate
\begin{equation*}
\label{eqn:tensor_T}
\sum_{i_1 = 1}^{d_1} \ldots \sum_{i_{j-1} = 1}^{d_{j-1}} \sum_{i_{j+1} = 1}^{d_{j+1}} \ldots \sum_{i_{m} = 1}^{d_m} T_{i_1,\ldots,i_{j-1},i,i_{j+1},\ldots, i_m} {(\x_1)}_{i_1} \cdots {(\x_{j-1})}_{i_{j-1}} {(\x_{j+1})}_{i_{j+1}} \cdots {(\x_m)}_{i_m}.
\end{equation*}
Eigenvectors and singular vectors of tensors were introduced in~\cite{lim2005singular, qi2005eigenvalues}. The eigenvectors of $T \in  (\C^d)^{\otimes m}$ are all non-zero $\x \in \C^d$ satisfying 
\begin{equation*}
    \label{eqn:T_notation}
    T(\, \cdot \,, \x,\ldots,\x) = \lambda \cdot \x,
\end{equation*} 
for some scalar $\lambda \in \C$. In the special case of matrices, this reduces to the familiar formula $A \x = \lambda \x$. Similarly, the singular vectors of a tensor $T \in \CC^{d_1} \otimes \cdots \otimes \CC^{d_n}$ are the tuples of vectors $(\x_1, \ldots, \x_n) \in \CC^{d_1} \times \cdots \times \CC^{d_n}$, with each $\x_i \neq 0$, satisfying 
\begin{align*} 
T(\x_1, \ldots, \x_{j-1},\,\cdot\,,\x_{j+1},\ldots,\x_n) = \mu_j\x_j
\end{align*}
for all $j$. This specialises for matrices to the familiar pair of conditions $A \x_2 = \mu_1 \x_1$ and $A^\top \x_1 = \mu_2 \x_2$.

Eigenvectors and singular vectors have been computed for special classes of tensors in \cite{robeva2016orthogonal, robeva2017singular}; they have been used to study  hypergraphs~\cite{benson2019three, qi2017tensor} and to learn parameters in latent variable models~\cite{anandkumar2014tensor, auddy2020perturbation}. They have a variational interpretation as being the critical points of certain optimisation problems \cite{lim2005singular}. It is well known that the singular vectors of a tensor $T \in \RR^{d_1} \otimes \cdots \otimes \RR^{d_n}$ are both the critical points \cite{friedland-ottaviani} of the best rank-1 approximation
\begin{equation*}
    \min_{\x_i \in \RR^{d_i}}\min_{a \in \RR} \; \|T - a\cdot \x_1 \otimes \cdots \otimes \x_n\| \quad \text{s.t.} \quad \|\x_1\| = \cdots \|\x_n\| = 1
\end{equation*}
and the spectral norm functional
\begin{equation*}
    \max_{\x_i \in \RR^{d_i}} \; T(\x_1,\ldots,\x_n) \quad \text{s.t.} \quad \|\x_1\| = \cdots \|\x_n\| = 1
\end{equation*}
The eigenvectors of a symmetric tensor $T \in (\RR^d)^{\otimes m}$ are likewise the critical points of similar optimisation problems. For various definitions of the symmetric \textit{Laplacian tensor} and \textit{degree tensor} $A,B \in (\RR^d)^{\otimes m}$ of an $m$-regular hypergraph $G$, the \textit{eigenvalues} of $G$ are the critical values of the optimisation problem
\begin{equation}{\label{eq:genealised-tensor-functional}}
    \min_{\x \in \RR^{d}} \; \frac{A(\mathbf{x},\ldots,\mathbf{x})}{B(\mathbf{x},\ldots,\mathbf{x})} \quad \text{s.t.} \quad \|\x\| = 1
\end{equation}
These eigenvalues and eigenvectors give bounds for Cheeger-like inequalities for $G$ and can be used to perform hypergraph clustering \cite{hypergraph2,hypergraph1,hypergraph4,hypergraph3}.

In order to create the appropriate generalisation of quiver singular vectors to subsume these notions from the spectral theory of tensors, we generalise from quivers to {\em hyperquivers}. In general, hyperquivers are obtained from quivers  by allowing the source and target maps $s,t:E \to V$ to be multivalued. For our purposes, it suffices to consider hyperquivers where each edge $e \in E$ has a single target vertex. The hyperedge $e$ now has a tuple of sources $(s_1(e), s_2(e), \ldots, s_\mu(e)) \in V^\mu$ for some $e$-dependent integer $\mu$. A {\em representation} $\bR = (\mathbf{d},T)$ of such a hyperquiver assigns to each vertex $i$ a vector space $\mathbb{C}^{d_i}$ and to each edge $e$ a tensor
$$T_e \in \mathbb{C}^{d_{t(e)}} \otimes \mathbb{C}^{d_{s_1(e)}} \otimes \ldots \otimes \mathbb{C}^{d_{s_\mu(e)}}.$$
We identify a vector space $\C^{d}$ with its dual $(\C^{d})^*$, allowing us to view the tensor $T_e$ as a multilinear map
\begin{equation*}
\begin{aligned}{\label{eq:multilinear-map}}
T_e:\prod_{j=1}^\mu \mathbb{C}^{d_{s_j(e)}} &\to \mathbb{C}^{d_{t(e)}} \\
(\x_{s_1(e)},\ldots,\x_{s_\mu(e)}) &\mapsto T_e(\, \cdot \,, \x_{s_1(e)}, \ldots, \x_{s_\mu(e)}).
\end{aligned}    
\end{equation*}
Our hyperquiver representations reduce to usual quiver representations when each edge has $\mu = 1$. 

A {\em singular vector} of a representation $\bR$ is a tuple $\left( [\x_i] \in \P(\mathbb{C}^{d_i}) \mid i \in V \right)$ that satisfies
\begin{equation}{\label{eq:singular-vector-hyperquiver}}
    T_e(\, \cdot \,,\x_{i_1},\ldots,\x_{i_\mu}) = \lambda_e \cdot \x_j,
\end{equation}
for some scalar $\lambda_e$, across every edge $e \in E$ of the form $(i_1,\ldots,i_\mu) \to j$. We work with vectors in a product of projective spaces, since we require the vectors to be non-zero (as for the singular vectors of a matrix) and moreover because the equation \eqref{eq:singular-vector-hyperquiver} still holds after non-zero rescaling of each $\x_i$, albeit for different scalars $\lambda_e$.

Perhaps the simplest nontrivial family of examples is furnished by starting with the quiver with a single vertex and a single hyperedge with $m-1$ source vertices --- we call this the $m$-Jordan hyperquiver. Consider the representation that assigns, to the vertex, the vector space $\C^d$ for some dimension $d \geq 0$, and to the edge, a tensor $T \in (\CC^d)^{\otimes m}$, seen as a multilinear map $T : (\CC^d)^{(m-1)} \to \CC^d$ that contracts vectors along the last $\mu = m-1$ modes of $T$; see Figure \ref{fig:loop-hyperquiver-ex} for the case $m = 3$.
The singular vectors of this representation are all $[\x] \in \P(\C^d)$ satisfying $T(\, \cdot \,, \x, \x, \ldots, \x ) = \lambda \cdot \x$ for some scalar $\lambda \in \C$. The singular vectors of the representation are therefore the eigenvectors of the tensor $T$.
 
\input{figures/tensor-eigenvector}

The compatibility conditions that define singular vectors can be reframed in terms of the vanishing of minors of suitable $d_i \times 2$ matrices. For a sufficiently generic hyperquiver representation $\mathbf{R}$, we will see that its singular vectors form a smooth, multiprojective variety in $\prod_{i \in V} \P(\mathbb{C}^{d_i})$ which we call the {\em singular vector variety} and denote by $\cS(\bR)$. This variety simultaneously forms a multilinear (and projective) generalisation of the linear {\em space of sections} of a quiver representation from \cite{seigal2022principal}, and a multi-tensor generalisation of the set of singular vectors of a single tensor from \cite{friedland-ottaviani}. 
The property that a point lies in $\mathcal{S}(\mathbf{R})$ is equivariant under an orthogonal change of basis on each vector space, as is true for the singular vectors of a matrix, as follows. Let $([x_1],\ldots,[x_n]) \in \prod_{i \in V} \P(\mathbb{C}^{d_i})$ be a singular vector tuple of a hyperquiver representation with tensors $T_e \in \mathbb{C}^{d_{t(e)}} \otimes \mathbb{C}^{d_{s_1(e)}} \otimes \ldots \otimes \mathbb{C}^{d_{s_\mu(e)}}$ and let 
$Q_1, \ldots, Q_n$ be a tuple of complex orthogonal matrices; i.e., $Q_i^\top Q_i = I_{d_i}$.
Then $([Q_1x_1],\ldots,[Q_nx_n])$ is a singular vector tuple of the hyperquiver representation where each $T_e$ has its source components multiplied by $Q_{s_j(e)}^\top$ and target component multiplied by $Q_{t(e)}$. We expect the topology of this variety, particularly its (co)homology groups, to provide rich and interesting isomorphism invariants for hyperquiver representations.

\bigskip

{\bf Main Result.} We derive exact and explicit formulas for the {\em dimension} and {\em degree} of $\cS(\bR)$ when $\bR$ is a sufficiently generic representation of a given hyperquiver. Here is a simplified version, in the special case when all vector spaces are of the same dimension.

\begin{theorem}{\label{thm:main_specialcase}}
Let $\bR = (\mathbf{d},T)$ be a generic representation of a hyperquiver $H = (V,E)$ with $\mathbf{d} = (d,d,\ldots,d)$. Let $N = (d-1) (|V| - |E|)$. If $N < 0$, then $\cS(\bR) = \varnothing$. If $N \geq 0$, let $D$ be the coefficient of $\left( \prod_{i \in V} h_i \right)^{d-1}$ in the polynomial
    $$\left(\sum_{i \in V} h_i\right)^N \cdot \prod_{e \in E} \left( \sum_{k = 1}^{d} h_{t(e)}^{k-1} \cdot h_{s(e)}^{d-k} \right), \quad \text{where} \quad h_{s(e)} := \sum_{j=1}^{\mu(e)} h_{s_j(e)}.$$
Then $\cS(\bR) = \varnothing$ if and only if $D = 0$. Otherwise, $\cS(\bR)$ has dimension $N$ and degree $D$. Moreover, if $\dim \cS(\bR) = 0$, then each singular vector tuple occurs with multiplicity $1$.
\end{theorem}

\begin{example}
Let $\mathbf{R}$ be the hyperquiver representation in Figure \ref{fig:labelled-edges}, with $T \in \C^3 \otimes \C^3 \otimes \C^3$ a generic tensor. We have $N = (3-1)(2-2) = 0$. We seek the coefficient $D$ of the monomial $h_1^2h_2^2$ in the polynomial
$$\left((h_1+h_2)^2 + h_1(h_1+h_2) +  h_1^2\right)^2 = 9h_1^4 + 18h_1^3h_2 + 15 h_1^2h_2^2 + 6h_1h_2^3 + h_2^4.$$
We see that $D = 15$. Hence the singular vector variety $\mathcal{S}(\mathbf{R})$ has dimension $N = 0$ and consists of 15 singular vector tuples, each occurring with multiplicity 1.
\input{figures/labelled-edges}
\end{example}

Our argument follows the work of Friedland and Ottaviani from \cite{friedland-ottaviani} --- we first construct a vector bundle whose generic global sections have the singular vectors of $\bR$ as their zero set, and then apply a variant of Bertini's theorem to count singular vectors by computing the top Chern class of the bundle. The authors of \cite{friedland-ottaviani} compute the number of singular vectors of a single generic tensor --- this corresponds to counting the singular vectors of the hyperquiver representation depicted in Figure \ref{fig:hyperquiver-examples}(b). Here we derive general formulas to describe the algebraic variety of singular vectors for an arbitrary network of (sufficiently generic) tensors.

\bigskip

{\bf Related Work.} Special cases of our degree formula, all in the case $\dim \cS(\bR) = 0$, recover existing results from the literature --- see \cite{cartwright2013number} and~\cite[Corollary 3.2]{fornaess1995complex} for eigenvector counts, \cite{friedland-ottaviani} for singular vector counts, and \cite{tensor-eigenvalue-problem, friedland-ottaviani} for generalised eigenvector counts. In addition, \cite{Pan16,OSV21} study the asymptotic and stabilisation behaviour of singular vector counts. Other recent work that builds upon the approach in \cite{friedland-ottaviani} includes \cite{draisma2018best, sodomaco2022span} which study the span of the singular vector tuples, \cite{turatti2022tensors} which studies tensors determined by their singular vectors, and the current work~\cite{Abo_Portakal_Sodomaco_2023} which uses a related setup to count totally mixed Nash equilibria. The eigenscheme of a matrix \cite{ABO2016121} and ternary tensor \cite{eigenschemes} is a scheme-theoretic version of $\mathcal{S}(\mathbf{R})$ for the Jordan quiver in Figure \ref{fig:Jordan-quiver} and the hyperquiver in Figure \ref{fig:loop-hyperquiver-ex}.

\bigskip

{\bf Outline.} The rest of this paper is organised as follows. In Section \ref{sec:svv} we introduce hyperquiver representations and their singular vector varieties. We state our main result, Theorem \ref{thm:main}, in Section \ref{sec:consequences} and describe a few of its applications. The construction of the vector bundle corresponding to a hyperquiver representation is given in Section \ref{sec:bundles}, and our Bertini-type theorem -- which we hope will be of independent interest -- is proved in Section \ref{sec:bertini}. We show that for generic $\bR$ the hypotheses of the Bertini theorem are satisfied by the associated vector bundle in Section \ref{sec:almost_generated}, and compute its top Chern class in Section \ref{sec:top_chern}. In Section \ref{sec:inverse-problems}, we consider an application of hyperquiver representations in multilinear algebra, where we address the {\em inverse eigenvalue problem} of finding representations which admit a given collection of vectors as their singular vectors. For completeness, we have collected relevant results from intersection theory in Appendix \ref{sec:intersection_theory}.

\section{The singular vector variety}\label{sec:svv}

We establish notation for hyperquiver representations, define their singular vector varieties, and highlight the genericity condition which plays a key role in the sequel. Without loss of generality, we henceforth assume $V = [n]$, where $[n]:= \{1,\ldots,n\}$ for $n \in \mathbb{N}$. 

\begin{definition}\label{defn:hyperquiver}
A \textit{hyperquiver} $H = (V,E)$ consists of a finite set of \textit{vertices} $V$ of size $|V| = n$ and a finite set of \textit{hyperedges} $E$. For each hyperedge $e \in E$ we have
\begin{enumerate} [label=(\roman*)]
	\item a non-negative integer $\mu(e)$ called the {\em index} of $e$ 
	\item  a tuple of vertices $v(e) \in V^{\mu(e)+1}$ called the \textit{vertices} of $e$.
\end{enumerate}
For brevity, we may refer to a hyperedge as an edge and write $\mu$ as a shorthand for $\mu(e)$.
The $j$-th entry of tuple $v(e)$ is denoted $s_j(e) \in V$. The tuple $s(e) := (s_0(e),s_1(e),\ldots,s_{\mu}(e))$ are the {\em sources} of $e$, and the vertex $t(e) := s_0(e)$ is the {\em target} of $e$.
\end{definition}


\begin{remark}
Usual quivers are the special case with $\mu = 1$ for all $e \in E$. Definition~\ref{defn:hyperquiver} does not exclude entries of $s(e)$ being equal to $t(e)$, nor does it exclude multiple hyperedges with the same tuple $v(e)$. 
\end{remark}

We now define representations of hyperquivers. 
The definition works for vector spaces over any field, but we focus on $\CC$.

\begin{definition} \label{defn:hyperquiverRepresentation}
Fix a hyperquiver $H = (V,E)$. 
Let $\mathbf{d} = (d_1,\ldots,d_n)$ be a \emph{dimension vector}.
A \textit{representation} $\bR = (\mathbf{d},T)$ of $H$ assigns
 \begin{enumerate}[label=(\roman*)]
     \item  A vector space $\mathbb{C}^{d_i}$ to each vertex $i \in V$.
     \item A tensor $T_e \in \mathbb{C}^e$ to each hyperedge $e \in E$, where $\mathbb{C}^e := \mathbb{C}^{d_{t(e)}} \otimes \CC^{d_{s_1(e)}} \otimes \cdots \otimes \CC^{d_{s_{\mu}(e)}}$, which is viewed as a multilinear map $\prod_{j=1}^\mu \mathbb{C}^{d_{s_j(e)}} \to \mathbb{C}^{d_{t(e)}}$.
 \end{enumerate}
\end{definition}
We define for brevity
\begin{equation}{\label{eqn:te_xse}}
T_e(\mathbf{x}_{s(e)}) := T_e (\, \cdot \,, \x_{s_1(e)}, \ldots, \x_{s_\mu(e)}).
\end{equation}
We say that two tensors $T_e$ and $T_{e^\prime}$ \emph{agree up to permutation} if the tuples $v(e)$ and $v(e^\prime)$ agree up to a permutation $\sigma$ and
$$(T_e)_{i_0,i_1,\ldots,i_\mu} = (T_{e^\prime})_{i_{\sigma(0)},i_{\sigma(1)},\ldots,i_{\sigma(\mu)}}.$$

\begin{remark}{\label{remark:choice-of-bases}}
In standard quiver representation theory, abstract vector spaces and linear maps are assigned to vertices and edges, respectively. Aligning with other works in the tensor literature, in this paper we fix a choice of basis for each vector space to identify $\mathbb{C}^{d_i}$ with its dual $(\mathbb{C}^{d_i})^*$, allowing us to assign transposes of matrices or permutations of tensors along the edges. We do not equip the vector spaces with their standard Hermitian inner product $\langle \mathbf{u}, \mathbf{v} \rangle_{\mathbb{C}} = \sum_{i=1}^d u_i\overline{v_i}$ for $\mathbf{u}, \mathbf{v} \in \mathbb{C}^{d}$. Instead, contraction of tensors will be performed with the standard real inner product $\mathbf{u}^\top \mathbf{v} = \langle \mathbf{u}, \mathbf{v} \rangle_{\mathbb{R}} = \sum_{i=1}^d u_i v_i$ for $\mathbf{u}, \mathbf{v} \in \mathbb{C}^{d}$. We will choose bases over $\mathbb{C}$ to consist of real vectors that are orthonormal with respect to $\langle \, \cdot \, , \, \cdot \, \rangle_{\mathbb{R}}$, see \cite[Remark 1.1]{eigenschemes}. The fact that this pairing is degenerate over $\mathbb{C}^d$, i.e. there are non-zero $\mathbf{u} \in \mathbb{C}^{d}$ with $\mathbf{u}^\top \mathbf{u} = 0$, will play a role later in the paper.
\end{remark}

Our main result finds the dimension and degree of the singular vector variety for a hyperquiver representation 
 with sufficiently general tensors on the hyperedges. We say that a property $P$ holds for a generic point of an irreducible affine variety $V$ if there exists a Zariski-open set $U$ in $V$ such that $P$ holds for all points in $U$. We call any point of such a $U$ a \emph{generic point} of $V$.
One way that a hyperquiver representation can be sufficiently generic is for the tuple of tensors $(T_e \mid e \in E)$ assigned to its edges to be generic; that is, a generic point of $\prod_{e \in E} \otimes_{i=0}^{\mu(e)} \mathbb{C}^{d_{s_i(e)}}$. This holds, for example, in Figure~\ref{fig:Jordan-quiver} and~\ref{fig:Kronecker-quiver}.
But our notion of genericity allows tensors on different hyperedges to coincide, as in Figure~\ref{fig:Loop-quiver}. Our genericity condition is encoded by a partition of the hyperedges.
\begin{definition}[Genericity of  a hyperquiver representation]
\label{def:generic}
\mbox{}
\begin{enumerate}[label=(\roman*)]
    \item A \emph{partition} of a hyperquiver $H = (V,E)$ is a partition of its hyperedges $E = \coprod_{r=1}^M E_r$ such that for any hyperedges $e, e^\prime, e^{\prime\prime} \in E_r$,
    \mbox{}
    \begin{enumerate}
        \item the indices $\mu(e)$ and $\mu(e^\prime)$ equal the same number $\mu$
        \item the tuples $v(e)$ and $v(e^\prime)$ coincide up to a permutation $\sigma$ of the set $\{0,1,\ldots,\mu\}$, which must be the identity permutation if $e = e^\prime$
        \item if $\sigma$ and $\sigma^\prime$ are the permutations in (b) for $v(e)\rightarrow v(e^\prime)$ and $v(e^{\prime\prime})\rightarrow v(e^\prime)$ respectively, and if $e \neq e^{\prime\prime}$, then $\sigma(0) \neq \sigma^\prime(0)$.
    \end{enumerate}
    \item The \emph{partition} of a representation $\bR = (\mathbf{d},T)$ is the unique partition of $H$ such that for any $e,e^\prime \in E_r$, the tensors on $e$ and $e^\prime$ agree up to a permutation $\sigma$.
    \item The representation $\mathbf{R} = (\mathbf{d},T)$ is \emph{generic} if given hyperedges $e_r \in E_r$ for $r \in [M]$, the tuple of tensors $\left(T_{e_1},T_{e_2},\ldots,T_{e_M}\right)$ is a generic point in $\prod_{r =1}^M \mathbb{C}^{e_r}$.
\end{enumerate}
\end{definition}

\begin{definition}\label{defn:svv}
The set of \textit{singular vector tuples} $\cS(\bR)$ of a representation $\bR$ consists of tuples $\mathbf{\chi} = 
([\mathbf{x}_1], \ldots, [\mathbf{x}_n]) \in 
\prod_{i=1}^n \P(\CC^{d_i})$ such that
\begin{equation}
    \label{eqn:sv_def}
T_e(\mathbf{x}_{s(e)}) = \lambda_e \mathbf{x}_{t(e)},
\end{equation}
for some scalar $\lambda_e \in \mathbb{C}$, for every edge $e \in E$. 
\end{definition}

\begin{remark}
    The scalars $\lambda_e$ in~\eqref{eqn:sv_def} can be thought of as the singular values of the singular vector tuple $(\mathbf{x}_1, \ldots, \mathbf{x}_n)$. However, the non-homogeneity of~\eqref{eqn:sv_def} means that rescaling vectors in the tuple can change the singular values. We say that a singular vector tuple has a \textit{singular value zero} if $\lambda_e = 0$ for some edge $e\in E$. 
\end{remark}

The singular vector tuples $\mathbf{\chi} = ([\mathbf{x}_1],\ldots,[\mathbf{x}_n])$ are the tuples whose $d_{t(e)} \times 2$ matrix
\[
M_e(\mathbf{x}) := \begin{pmatrix}
| & | \\
T_e(\mathbf{x}_{s(e)}) & \mathbf{x}_{t(e)} \\
| & | 
\end{pmatrix}
\]
has rank $\leq 1$ for all $e \in E$. 
The rank of this matrix depends only on the points $[\mathbf{x}_i] 
\in \P(\CC^{d_i})$, and not on the vectors $\mathbf{x}_i \in \CC^{d_i}$. Therefore, the set of singular vector tuples $\cS(\bR)$ is a scheme in the multiprojective space $X = \prod_{i=1}^n \mathbb{P}(\mathbb{C}^{d_i})$, whose defining equations are the $2 \times 2$ minors of all matrices $M_e(\mathbf{x})$ for $e \in E$. When $\bR$ is a generic representation, our main Theorem \ref{thm:main} in the next section shows that $\cS(\bR)$ is in fact a smooth variety. Thus, when $\bR$ is a generic representation, we call $\cS(\bR)$ the \textit{singular vector variety} of $\bR$. When we speak of the degree of the singular vector variety $\mathcal{S}(\mathbf{R})$, we refer to the degree of its image under the Segre embedding $s: X \xhookrightarrow[]{} \mathbb{P}^D$, for $D = \prod_{i=1}^n d_i - 1$.

\begin{example}
Being a singular vector tuple is not invariant under an arbitrary change of basis. For example, the quiver in Figure \ref{fig:quiver-examples}(b) with a generic square matrix $A: \mathbb{C}^d \rightarrow \mathbb{C}^d$ has $d$ singular vector pairs $([\x],[\y])$. However, there exist change of basis matrices $M_1,M_2 \in GL(d,\mathbb{C})$ such that $M_2AM_1^{-1} = I_d$, and the identity matrix $I_d$ has infinitely many singular vector pairs: all pairs $([\mathbf{z}],[\mathbf{z}])$. Therefore, the singular vector variety is not $GL(d_i,\mathbb{C})$-invariant. As discussed in the introduction and in \cite[Remark 1.1]{eigenschemes} and \cite[Theorem 2.20]{qi2017tensor}, the property of being a singular vector tuple is preserved by an complex orthogonal change of basis. Therefore, the singular vector variety is $O(d_i,\mathbb{C})$-invariant, with respect to the standard real inner product between two complex vectors, see Remark \ref{remark:choice-of-bases}.
\end{example}

\section{Main theorem and its consequences}\label{sec:consequences}

In this section, we present our main result in full generality and study its consequences. Recall that a vector $\mathbf{u} \in \mathbb{C}^d$ is \textit{isotropic} if $\mathbf{u}^\top \mathbf{u} = 0$, see Remark \ref{remark:choice-of-bases}.

\begin{theorem}\label{thm:main}
Let $\bR = (\mathbf{d},T)$ be a generic hyperquiver representation and $\mathcal{S}(\mathbf{R})$ be the set of singular vector tuples of $\mathbf{R}$. Let $N = \sum_{i \in V} (d_i - 1) - \sum_{e \in E} (d_{t(e)} - 1)$. If $N < 0$, then $\cS(\bR) = \varnothing$. If $N \geq 0$, let $D$ be the coefficient of the monomial $h_1^{d_1-1}\cdots h_n^{d_n-1}$ in the polynomial
\begin{equation}{\label{eq:main-thm-polynomial}}
    \left(\sum_{i \in V} h_i\right)^N \cdot \prod_{e \in E} \left( \sum_{k = 1}^{d_{t(e)}} 
    h_{t(e)}^{k-1}
    h_{s(e)}^{d_{t(e)}-k} \right), \quad \text{where} \quad h_{s(e)} := \sum_{j=1}^{\mu(e)} h_{s_j(e)}.
\end{equation}
Then $\cS(\bR) = \varnothing$ if and only if $D = 0$. Otherwise, $\cS(\bR)$ is a smooth variety of pure dimension $N$ and has degree $D$. If $\mathbf{R}$ has finitely many singular vector tuples, then each singular vector tuple is of multiplicity 1, is not isotropic, and has no singular value equal to 0.
\end{theorem}

Note that the partition from Definition~\ref{def:generic} does not appear in the statement of Theorem~\ref{thm:main}: the partition provides a genericity condition for the result to hold, but the dimension and degree of the singular vector variety do not depend on the partition. Next we give a sufficient condition for a hyperquiver representation to consist of finitely many points. This condition applies to Figure~\ref{fig:loop-hyperquiver-ex} and Figure~\ref{fig:Jordan-hyperquiver-ex}, but not to Figure~\ref{fig:Kronecker-hyperquiver-ex}.

\begin{corollary}
\label{cor:finiteSVV} 
The hyperquivers that have finitely many singular vector tuples in a generic representation, for any choice of dimension vector, are those whose vertices each have exactly one incoming hyperedge.
\end{corollary}

\begin{proof}
If $\dim \mathcal{S}(\mathbf{R}) = N = \sum_{i \in V} (d_i - 1) - \sum_{e \in E} (d_{t(e)} - 1) = 0$ for all dimensions $d_i$, then $\sum_{i \in V} (d_i - 1) = \sum_{e \in E} (d_{t(e)} - 1)$ as polynomials in the variables $d_i$. Each $d_i$ appears exactly once on the left hand side of the equation. Hence it must also appear exactly once on the right hand side. Therefore $|V| = |E|$ and every $i \in V$ has exactly one $e \in E$ with $i = t(e)$.
\end{proof}

We show how Theorem~\ref{thm:main} specialises to count the eigenvectors and singular vectors of a generic tensor, as well as to count the solutions to the generalised eigenproblem from~\cite{tensor-eigenvalue-problem}. 

\begin{example}[Eigenvectors of a tensor]
\label{ex:eigenvectors}
We continue our discussion from the introduction. The representation of the $m$-Jordan hyperquiver with a generic tensor $T \in (\C^d)^{\otimes m}$ on its hyperedge is generic in the sense of Definition~\ref{def:generic}, since we have only one hyperedge. There are finitely many eigenvectors, by Corollary~\ref{cor:finiteSVV}. The polynomial \eqref{eq:main-thm-polynomial} is $$\sum_{k=1}^d h^{k-1}((m-1)h)^{d-k} = \left(\sum_{k=1}^d (m-1)^{d-k}\right)h^{d-1} = \frac{(m-1)^d -1}{m-2}h^{d-1}.$$
The coefficient of $h^{d-1}$ is $\frac{(m-1)^d -1}{m-2}$. This agrees with the count for the number of eigenvectors of a generic tensor from~\cite[Theorem 1.2]{cartwright2013number} and \cite[Corollary 3.2]{fornaess1995complex}.
\end{example}

We now consider singular vectors of \textit{partially symmetric tensors}. Let $m = \omega_0 + \omega_1 + \cdots + \omega_p$ be a partition where $\omega_0 = 0$ and each $\omega_j$ is a positive integer for $j \in [p]$. Let $\boldsymbol{\omega} = (\omega_1,\ldots,\omega_p)$ and let $S^{\boldsymbol{\omega}}(\CC^{\mathbf{d}})$ denote the subspace of $\bigotimes_{i = 1}^p (\CC^{d_i})^{\otimes \omega_i}$ consisting of tensors $T$ that are partially symmetric with respect to $\boldsymbol{\omega}$, i.e. $T_{i_1,\ldots,i_m}$ is invariant under any permutation of the indices in the $k$-th group of indices $i_{j_k},i_{j_k+1}\ldots,i_{j_k+\omega_k}$ where $j_k = 1 + \sum_{j=0}^{k-1} \omega_j$, for $k \in [p]$. Then when $p = m$ and $\boldsymbol{\omega} = (1,\ldots,1)$, $S^{\boldsymbol{\omega}}(\CC^{\mathbf{d}})$ is the space of all tensors $\CC^{d_1} \otimes \cdots \otimes \CC^{d_m}$, and when $p = 1$ and $\boldsymbol{\omega} = m$, it is the space of symmetric tensors in $(\CC^{d})^{\otimes m}$.

A \textit{symmetric singular vector tuple} of $T \in S^{\boldsymbol{\omega}}(\CC^{\mathbf{d}})$ is a singular vector tuple of the form $([\x_1],\ldots,[\x_1],\ldots,[\x_p],\ldots,[\x_p]) \in \prod_{i=1}^p\P(\CC^{d_i})^{\omega_i}$. Thus when $p = m$, a symmetric singular vector tuple is a singular vector tuple of a tensor, and when $p = 1$, it is an eigenvector of a symmetric tensor. A result of Friedland and Ottaviani~\cite{friedland-ottaviani} is:

\begin{theorem}[{Friedland and Ottaviani \cite[Theorem 12]{friedland-ottaviani}}]{\label{friedland-ottaviani}}
The number of symmetric singular vectors of a generic partially symmetric tensor $T \in S^{\boldsymbol{\omega}}(\CC^{\mathbf{d}})$ is the coefficient of the monomial $h_1^{d_1-1}\cdots h_n^{d_n-1}$ in the polynomial
\begin{equation}{\label{friedland-ottaviani-formula}}
\prod_{i \in [p]} \frac{\widehat{h_i}^{d_i} - h_i^{d_i}}{\widehat{h_i} - h_i}, \qquad \text{where} \quad \widehat{h_i} := (\omega_i - 1)h_i \; - \sum_{j \in [p]\setminus\{i\}} \omega_j h_j, \; i \in [p].
\end{equation}
Each singular vector tuple is of multiplicity 1, is not isotropic, and does not have singular value 0. \end{theorem}

The count in the above result interpolates between the number of singular vectors of a generic tensor when $p = m$, and the number of eigenvectors of a generic symmetric tensor when $p = 1$, which is the same count in Example \ref{ex:eigenvectors} for a generic non-symmetric tensor. We explain how Theorem \ref{friedland-ottaviani} follows from Theorem~\ref{thm:main}. We consider the $p = m$ case and then partially symmetric tensors in general.

\begin{example}[Singular vectors of a tensor]
\label{ex:singularvector}
Consider the hyperquiver with $n$ vertices $V = [n]$ and $n$ hyperedges. For every vertex $i \in V$, there is a hyperedge $e_i$ with $s(e_i) = (1, \ldots, i-1, i+1, \ldots, n)$ and target $t(e) = i$. Consider the representation that assigns the vector space $\mathbb{C}^{d_i}$ to each vertex and the same generic tensor $T \in \C^{d_1} \otimes \cdots \otimes \C^{d_n}$ to each hyperedge. On each edge $e_i$, the tensor $T$ is seen as a multilinear map
\begin{align*}
T: \C^{d_1} \times \ldots \times \C^{d_{i-1}} \times \C^{d_{i+1}} \times \cdots \times \C^{d_n} &\to \C^{d_i} \\
(\x_1, \ldots, \x_{i-1},\x_{i+1},\ldots,\x_n) &\mapsto T(\x_1, \ldots, \x_{i-1},\,\cdot\,,\x_{i+1},\ldots,\x_n).
\end{align*}
This representation is generic in the sense of Definition~\ref{def:generic}, where the partition of the edge set $E$ has size $M = 1$ and the permutation $\sigma$ sending $v(e_i)$ to $v(e_j)$ is the one that swaps $i$ and $j$ and keeps all other entries fixed. Figure \ref{fig:Jordan-hyperquiver-ex} illustrates this representation for $n = 3$. The singular vector variety consists of all non-zero vectors $\mathbf{x}_i \in \mathbb{C}^{d_i}$ such that $T(\mathbf{x}_{s(e)}) = \lambda_e \mathbf{x}_{t(e)}$ for some $\lambda_e \in \mathbb{C}$ and all $e \in E$, where $T(\mathbf{x}_{s(e)})$ is defined in \eqref{eqn:te_xse}. That is, the singular vector variety consists of all singular vector tuples of $T$. Corollary \ref{cor:finiteSVV} shows that there are finitely many singular vector tuples. The polynomial \eqref{eq:main-thm-polynomial} specialises to
$$\prod_{i\in [n]} \left( \sum_{k = 1}^{d_i} 
h_i^{k-1}
\widehat{h_i}^{d_i-k} \right), \quad \text{where} \quad \widehat{h_i} := \sum_{j \in [n]\setminus\{i\}} h_j, \; i \in [n].$$
This is equivalent to \eqref{friedland-ottaviani-formula} and \cite[Theorem 1]{friedland-ottaviani}, via the identity $\frac{x^n - y^n}{x-y} = \sum_{k=1}^n x^{k-1}y^{n-k}$.
\end{example}

\begin{example}[Symmetric singular vectors of a partially symmetric tensor]
Suppose $T \in \CC^{3} \otimes \CC^{3} \otimes \CC^{3}$ is a generic tensor that is partially symmetric in its first two entries: $T_{ijk} = T_{jik}$. Then a symmetric singular vector tuple $([\x],[\x],[\mathbf{y}]) \in \P(\CC^3) \times \P(\CC^3) \times \P(\CC^3)$ is defined by the equations
\begin{align*}
    [T(\x,\x,\,\cdot\,)] &= [\mathbf{y}] \\
    [T(\x,\,\cdot\,,\mathbf{y})] &= [\mathbf{x}] \\
    [T(\,\cdot\,,\x,\mathbf{y})] &= [\mathbf{x}]
\end{align*}
in $\P(\CC^3)$. We notice that one of these equations is redundant since $T(\x,\,\cdot\,,\mathbf{y}) = T(\,\cdot\,,\x,\mathbf{y})$ by partial symmetry. Thus, we can take only the first two equations and construct the corresponding hyperquiver representation as shown in Figure \ref{fig:partially-symmetric}. The singular vector tuples of this representation are the same as the symmetric singular vector tuples of $T$. The polynomial in \eqref{eq:main-thm-polynomial} is
$$((h_1+h_2)^2 + (h_1h_2)h_1 + h_1^2)((2h_1)^2 + 2h_1h_2 + h_2^2) = 12 h_1^4 + 18 h_1^3 h_2 + 13 h_1^2 h_2^2 + 5 h_1 h_2^3 + h_2^4$$
Thus, the number of symmetric singular vector tuples of $T$ is 13, which is the same answer given by Theorem \ref{friedland-ottaviani}, and this construction can be extended to any partially symmetric tensor. In general, one can prove a version of Theorem \ref{thm:main} for a representation assigning partially symmetric tensors on every hyperedge by appropriately modifying the conditions in Definition \ref{def:generic} to define a \textit{generic} partially symmetric representation. This is needed to avoid duplicate defining equations for singular vector tuples, as seen above. For example, the representation in Figure \ref{fig:labelled-edges} is not generic for the tensor $T$.
\end{example}

\input{figures/partially-symmetric}

\begin{example}[The generalised tensor eigenvalue problem]
\label{ex:generalised_eigenproblem}
Consider a generic representation of the \textit{Kronecker hyperquiver} with a generic pair of tensors $A, B \in \C^{d_2}\otimes(\C^{d_1})^{\otimes (m-1)}$, see Figure \ref{fig:Kronecker-hyperquiver-ex} with $m = 3$ and $d = d_1 = d_2$. The edge set $E$ has a partition with $M = 2$. We remark that Corollary \ref{cor:finiteSVV} implies that there will not be finitely many singular vector tuples for \textit{all} representations of this hyperquiver. There will be a non-zero finite number of singular vectors if and only if $d := d_1 = d_2$ since this is when $N = 0$ in Theorem \ref{thm:main}.
The singular vector tuples are the non-zero pairs $\mathbf{x}, \mathbf{y} \in \mathbb{C}^d$ such that $A(\,\cdot\,,\mathbf{x},\ldots,\mathbf{x}) = \lambda^\prime \mathbf{y}$ and $B(\,\cdot\,,\mathbf{x},\ldots,\mathbf{x}) = \lambda^{\prime\prime} \mathbf{y}$, for some $\lambda^\prime, \lambda^{\prime\prime} \in \mathbb{C}$. This reduces to the single equation
\begin{equation}{\label{eq:generalized-eigenvalue-problem}}
A(\,\cdot\,,\mathbf{x},\ldots,\mathbf{x}) = \lambda B(\,\cdot\,,\mathbf{x},\ldots,\mathbf{x}) 
\end{equation}
for some $\lambda \in \mathbb{C}$. This is a tensor-analogue of the generalised eigenvalue problem for two matrices. When $A$ and $B$ are real with $A$ symmetric, $B$ positive definite, and $m$ even, the generalised eigenvectors $\x$ are the critical points of the spherical optimisation problem \eqref{eq:genealised-tensor-functional} \cite[Thoerem 4.58]{qi2017tensor}. It was shown in \cite[Corollary 16]{friedland-ottaviani} and \cite[Theorem 2.1]{tensor-eigenvalue-problem} that there are $d(m-1)^{d-1}$ generalised tensor eigenvalue pairs $\mathbf{x}$ and $\mathbf{y}$ for the tensors $A$ and $B$. Our general formula in Theorem \ref{thm:main} also recovers this number, as follows. The polynomial \eqref{eq:main-thm-polynomial} is
\begin{equation}{\label{kronecker-polynomial}}
\left(\sum_{k=1}^d h_2^{k-1}((m-1)h_1)^{d-k}\right)\left(\sum_{\ell=1}^d h_2^{\ell-1}((m-1)h_1)^{d-\ell}\right).
\end{equation}
A monomial $h_1^{d-1}h_2^{d-1}$ is obtained from the product of a $k$-th summand and an $\ell$-th summand such that $k + \ell = d+1$. 
There are $d$ such pairs of summands $k, \ell \in \{1, \ldots, d\}$.
Each such monomial will have a coefficient of $(m-1)^{d-1}$. Hence the coefficient of $h_1^{d-1}h_2^{d-1}$ in~\eqref{kronecker-polynomial} is $d(m-1)^{d-1}$.
\end{example}

Now we find the dimension and degree of the singular vector variety $\mathcal{S}(\mathbf{R})$ for a generic representation $\mathbf{R}$ of a hyperquiver with a single hyperedge, as shown in Figure \ref{generalSingleHyperedge}.
\input{figures/oneedge}
\begin{corollary}\label{cor:single-hyperedge}
Let $H$ be a hyperquiver with one hyperedge with all entries of its tuple of vertices distinct. Let $\mathbf{R}$ be the representation that assigns the vector space $\mathbb{C}^{d_i}$ to each vertex $i$ and a generic tensor to the hyperedge. Then:
\begin{enumerate}[(a)]
\item The dimension of $\mathcal{S}(\mathbf{R})$ is $N = \sum_{i = 1}^{n-1} d_i - n + 1$
\item The degree of $\mathcal{S}(\mathbf{R})$ is
\begin{equation}{\label{degree-single-hyperedge}}
\sum_{k = 1}^{d_n} \sum_{\substack{k_1 + \cdots + k_{n-1} \\ = d_n - k}} \binom{d_n-k}{k_1,\ldots,k_{n-1}}\binom{N}{d_1-1-k_1,\ldots,d_{n-1}-1-k_{n-1},d_n-k}.    
\end{equation}
\end{enumerate}
\end{corollary}

\begin{proof}
The dimension of $\mathcal{S}(\mathbf{R})$ is $N = \left(\sum_{i = 1}^{n} d_i - n\right) - (d_n - 1) = \sum_{i = 1}^{n-1} d_i - n + 1$, by Theorem~\ref{thm:main}. 
The degree 
of $\mathcal{S}(\mathbf{R})$ is the coefficient of $h_1^{d_1-1} \cdots h_n^{d_n-1}$ in the product
$$\underbrace{\left(\sum_{i = 1}^n h_i\right)^N}_{(1)}\underbrace{\left(\sum_{k = 1}^{d_n}\left(\sum_{i = 1}^{n-1} h_i\right)^{d_n-k} h_n^{k-1}\right)}_{(2)}.$$
For each $k \in \{1,\ldots,d_n\}$, the monomial $h_1^{k_1}\cdots h_{n-1}^{k_{n-1}}h_n^{k-1}$ in the expansion of $(2)$ for some $k_1,\ldots,k_{n-1}$ such that $\sum_{i = 1}^{n-1} k_i = d_n - k$ has coefficient $\binom{d_n-k}{k_1,\ldots,k_{n-1}}$. This is combined with the monomial $h_1^{d_1-1-k_1}\cdots h_{n-1}^{d_n-1-k_{n-1}}h_n^{d_n-k}$ from the expansion of (1), which has coefficient
$\binom{N}{d_1-1-k_1,\ldots,d_{n-1}-1-k_{n-1},d_n-k}$. Multiplying these coefficients and summing over those $k_1,\ldots,k_{n-1}$ with $\sum_{i = 1}^{n-1} k_i = d_n - k$, we obtain
$$\sum_{\substack{k_1 + \cdots + k_{n-1} \\ = d_n - k}} \binom{d_n-k}{k_1,\ldots,k_{n-1}}\binom{N}{d_1-1-k_1,\ldots,d_{n-1}-1-k_{n-1},d_n-k}.$$
Summing over $k = 1,\ldots,d_n$ gives the result.
\end{proof}

When $d := d_1 = \cdots = d_n$, we can use Corollary \ref{cor:single-hyperedge} to find the degree of $\mathcal{S}(\mathbf{R})$, which is displayed in Table \ref{tab:degreeOfSingleEdgeHyperquiver} for $d = 1,\ldots,6$ and $n = 2,\ldots,6$. Observe that: (i) the degree row of $d = 2$ consist of the factorial numbers; and (ii) the degree column of $n = 2$ consist of powers of $2$. We explain these observations.
To see (i), if $d = 2$, then \eqref{degree-single-hyperedge} becomes
\begin{equation}{\label{degree-d=2}}
\sum_{k = 1}^{2} \sum_{\substack{k_1 + \ldots + k_{n-1} \\ = 2 - k}} \binom{2-k}{k_1,\ldots,k_{n-1}}\binom{n-1}{1-k_1,\ldots,1-k_{n-1},2-k}.
\end{equation}
When $k = 2$, the only summands satisfying $k_1 + \cdots + k_{n-1} = 2 - k$ is $k_1 = \cdots = k_{n-1} = 0$, which is 1 for the first factor and $(n-1)!$ for the second factor in \eqref{degree-d=2}. When $k = 1$, the only allowed indices are of the form $k_i = 1$ and $k_j = 0$ for all $i \neq j$, from which we get 1 for the first factor and $(n-1)!$ for the second factor in \eqref{degree-d=2}. Since there are $n-1$ such allowed indices, \eqref{degree-d=2} evaluates to $(n-1)! + (n-1)(n-1)! = n!$. For (ii), when $n = 2$, we have
\begin{align*}
\sum_{k = 1}^{d} \sum_{\substack{k_1 = d - k}} \binom{d-k}{k_1}\binom{d-1}{d-1-k_1,d-k} &= \sum_{k = 1}^{d} \binom{d-k}{d-k}\binom{d-1}{k-1,d-k} \\
&= \sum_{k = 0}^{d-1} \binom{d-1}{k,d-1-k} = \sum_{k = 0}^{d-1} \binom{d-1}{k} = 2^{d-1}.
\end{align*}
\input{figures/dimensionDegreeTable.tex}

\begin{example}[Empty singular vector variety]
\label{ex:emptySVV}
Consider the quiver in Figure \ref{noSingularVectorQuivers}, where the vertices are assigned vector spaces of dimension $d > 1$, and the two edges are assigned generic matrices $A, B \in \CC^{d \times d}$. 
Any singular vector would need to be an eigenvector of both matrices $A$ and $B$, but a pair of generic matrices $A$ and $B$ do not share an eigenvector. The emptiness of the singular vector variety is captured by Theorem~\ref{thm:main_specialcase}, as $N < 0$.
\input{figures/noSingularVectors-new.tex}
\end{example}

\begin{example}[Insufficiently generic representations]{\label{ex:insufficient}}
The quiver representations in Figure~\ref{fig:insufficient-examples} with $d > 1$ and generic matrix $A \in \mathbb{C}^{d \times d}$ do not satisfy the genericity conditions in Definition \ref{def:generic}. In Figure \ref{fig:insufficient-examples}(a), the only permutations $\sigma,\sigma^\prime$ on $\{0,1\}$ sending the matrix $A$ on one edge to the matrix $A$ on the other edge and vice versa are the identity permutations, which fail to satisfy the condition $\sigma(0) \neq \sigma^\prime(0)$, causing one of the edges to be redundant. The resulting singular vector variety has dimension $d-1$ and degree $2^{d-1}$ by Corollary \ref{cor:single-hyperedge}, rather than the expected dimension $0$ and degree $d$ in Example \ref{ex:generalised_eigenproblem}. In Figure \ref{fig:insufficient-examples}(b), the singular vectors are the non-zero points $\mathbf{x} \in \mathbb{C}^d$ such that $A^2\mathbf{x} = \lambda A\mathbf{x}$ for some $\lambda \in \mathbb{C}$, of which there are $d$ solutions, rather than the expected $0$ solutions in Theorem \ref{thm:main}.

\input{figures/insufficient-triangle}
\end{example}

\begin{example}[Periodic orbits of order $n$]
\label{ex:fixed_points}
Consider the hyperquiver representation in Figure \ref{fig:period-n-representation} with a generic tensor $T \in (\mathbb{C}^d)^{\otimes m}$. The singular vector tuples are the non-zero vectors $\mathbf{x}_1,\ldots,\mathbf{x}_n \in \mathbb{C}^d$ such that 
\begin{align*}
T(\,\cdot\,,\mathbf{x}_1,\ldots,\mathbf{x}_1) &= \lambda_1 \mathbf{x}_2 \\
T(\,\cdot\,,\mathbf{x}_2,\ldots,\mathbf{x}_2) &= \lambda_2 \mathbf{x}_3 \\
&\vdots \\
T(\,\cdot\,,\mathbf{x}_n,\ldots,\mathbf{x}_n) &= \lambda_n \mathbf{x}_1
\end{align*}
for some $\lambda_i \in \mathbb{C}$. In other words, each $\mathbf{x}_i$ is a periodic point of order $n$.

The hyperquiver representation is not generic in the sense of Definition \ref{def:generic} as edges with different tuples $v(e)$ up to permutation are assigned the same tensor $T$. Hence Theorem~\ref{thm:main} does not apply. Nonetheless, we predict the dimension and degree, using Theorem~\ref{thm:main}. The result predicts finitely many $n$-periodic points, by Corollary \ref{cor:finiteSVV}. Their count is predicted to be the coefficient of the monomial $h_1^{d-1}\ldots h_n^{d-1}$ in the polynomial
\begin{equation}{\label{eq:periodic-polynomial}}
\left(\sum_{k=1}^d h_2^{k-1}(\mu h_1)^{d-k}\right) \left(\sum_{k=1}^d h_3^{k-1}(\mu h_2)^{d-k}\right)\ldots\left(\sum_{k=1}^d h_1^{k-1}(\mu h_n)^{d-k}\right),
\end{equation}
by Theorem \ref{thm:main},where $\mu = m-1$. This monomial is obtained from the product of terms
$$(h_2^{k-1}(\mu h_1)^{d-k})(h_3^{k-1}(\mu h_2)^{d-k})\ldots(h_1^{k-1}(\mu h_n)^{d-k})$$
coming from each of the respective factors in \eqref{eq:periodic-polynomial}, for each $k \in [d]$. The coefficient of this product is $\mu^{n(d-k)}$. Thus, the coefficient of $h_1^{d-1}\ldots h_n^{d-1}$ in \eqref{eq:periodic-polynomial} is
\begin{equation*}{\label{eq:number-of-n-periodic-fixed-points}}
\sum_{k=1}^{d} \mu^{n(d-k)} = \frac{\mu^{nd} - 1}{\mu^n-1} = \frac{(m-1)^{nd} - 1}{(m-1)^n-1}.    
\end{equation*}
This turns out to be the correct number of period-$n$ fixed points, as proved in \cite[Corollary 3.2]{complex-dynamics-fornaess-sibony}. The number of eigenvectors of a generic tensor is the special case $n = 1$ (Example \ref{ex:eigenvectors}).
\input{figures/period-n}
\end{example}

Example \ref{ex:fixed_points} seems to suggest that the conditions defining a generic representation in Definition \ref{def:generic} can be weakened so that Theorem \ref{thm:main} holds for a larger class of hyperquiver representations. In Example \ref{ex:insufficient}, which gives the wrong dimensions and degrees, we notice that there are two edges assigned to the same matrix that point to the same target vertex. We conjecture that the genericity conditions in Definition \ref{def:generic} can be modified in the following way so that Theorem \ref{thm:main} still holds: given any two distinct edges $e,e^\prime \in E_r$, in Definition \ref{def:generic}(i.b), only require that the tuples of dimensions $(d_{t(e)},d_{s_1(e)},\ldots,d_{s_\mu(e)})$ and $(d_{t(e^\prime)},d_{s_1(e^\prime)},\ldots,d_{s_\mu(e^\prime)})$ agree up to permutation, and in part (i.c) we only require that $t(e) \neq t(e^\prime)$. These conditions cover Example \ref{ex:fixed_points}, while Example \ref{ex:insufficient} is still excluded.

\bigskip

In the remainder of this section, we explore connections to dynamical systems and message passing.

\begin{example}
A {\em parameterised dynamical system} is a continuous map $f:X \times P \to X$, where $X$ and $P$ are compact triangulable topological spaces, respectively called the {\em state} and {\em parameter} space of $f$. Taking homology with complex coefficients, we obtain a $\CC$-linear map
\[
H_kf: H_k(X \times P) \to H_k(X)
\] in each dimension $k \geq 0$. We know from the K\"{u}nneth formula \cite[Section 5.3]{spanier} that the domain of $H_kf$ is naturally isomorphic to the direct sum $\bigoplus_{i+j=k}H_i(X) \otimes H_j(P)$. Therefore, each $H_kf$ admits a component of the form 
\[
T_k: H_k(X) \otimes H_0(P) \to H_k(X),
\]
We say that a non-zero homology class $\xi \in H_k(X)$ is {\em fixed} by $f$ at a non-zero homology class $\eta \in H_0(P)$ whenever there exists a scalar $\lambda \in \CC$ satisfying $T_k(\xi \otimes \eta) = \lambda \cdot \xi$. The set of all such fixed homology classes (up to scaling) is the singular vector variety of the hyperquiver representation in Figure \ref{fig:homology-loop}.

Let $k := \dim H_k(X)$ and suppose $P$ has $d$ connected components; i.e., $\dim H_0(P) = d$. Then the singular vector variety has dimension $d-1$ and degree equal to the coefficient of $h_1^{k-1}h_2^{d-1}$ in the polynomial 
$(h_1+h_2)^{d-1}\sum_{j=1}^k (h_1+h_2)^{k-j}h_2^{j-1}$,
by Theorem \ref{thm:main}. The monomial $h_1^{k-1}h_2^{d-1}$ arises by pairing the term $\binom{k-j}{i}h_1^i h_2^{k-j-i} h_2^{j-1} = \binom{k-j}{i}h_1^i h_2^{k-i-1}$ in the expanded sum with the term $\binom{d-1}{k-i-1}h_1^{k-i-1}h_2^{(d-1)-(k-i-1)}$ in the expanded parentheses, for all $0 \leq i \leq k-j$ and $1 \leq j \leq k$. Thus, its coefficient is
$$\sum_{j=1}^k\sum_{i=0}^{k-j} \binom{k-j}{i}\binom{d-1}{k-i-1}$$
In particular, if $P$ is connected (i.e., $d = 1$), then there is exactly one non-zero homology class in $H_k(X)$ fixed by~$f$.

\input{figures/homology-loop}
\end{example}

\begin{example}
Our hyperquiver framework counts the fixed points of certain multilinear message passing operations, as we now describe.
Assign vectors $\x_i^{(0)} := \x_i \in \CC^{d_i}$ to each $i \in V$.
Apply the multilinear map $T_e$ to the vectors $(\x_{s_1(e)}^{(k)}, \ldots, \x_{s_
\mu(e)}^{(k)})$ at nodes in $s(e)$. 
Then, update the vector at the target vertex $t(e)$ to
 \begin{equation}
 	\label{eqn:our_update}
 	\x_{t(e)}^{(k+1)} := T_e( \x_{s(e)}^{(k)} ) \in \CC^{d_{t(e)}} .
 \end{equation}
 In the limit, one converges to a fixed point of the update steps.  The singular vector variety consists of tuples of directions in $\CC^{d_i}$ that are fixed under these operations, for any order of update steps.

We compare the update~\eqref{eqn:our_update} to message passing graph neural networks, see e.g.~\cite{gilmer2017neural,hamilton2020graph}. The vector at each vertex is the features of the vertex. The vectors typically lie in a vector space of the same dimension, as in Theorem~\ref{thm:main_specialcase}. Message passing operations take the form
\begin{equation}
    \label{eqn:general_update}
    \x_{i}^{(k+1)} = f( \{ \x_i^{(k)} \} \cup \{ \x_j^{(k)} : j \in \mathcal{N}(i) \}) ,
\end{equation} 
where $\mathcal N(i)$ is the neighbourhood of vertex $i$.
That is, the vector of features at node $i$ in the $(k+1)$-th step depends on the features of node $i$ and its neighbours at the $k$-th step.
Our update step in~\eqref{eqn:our_update} is a special case of~\eqref{eqn:general_update}. We relate~\eqref{eqn:our_update} to operations in the literature. 

The function $f$ in~\eqref{eqn:general_update} often involves a non-linearity, applied pointwise. In comparison, we focus on the (multi)linear setting, as discussed for example in~\cite{chen2017supervised}.
There, the authors study the optimisation landscapes of linear update steps,
relating them to power iteration algorithms. Our approach to count the locus of fixed points sheds insight into the global structure of this optimisation landscape, in the spirit of~\cite{catanese2006maximum,draisma2016euclidean}. 
Studying such fixed point conditions directly is the starting point of implicit deep learning~\cite{el2021implicit,gu2020implicit}.

The neighbourhood $\mathcal{N}(i)$, for us, consists of nodes $j$ that appear in a tuple $s(e)$ for some edge $e$ with $t(e) = i$.
Update steps are usually over a graph rather than a hypergraph. The tensor multiplications from~\eqref{eqn:our_update} incorporate higher-order interactions.
Such higher-order structure also appears in tensorised graph neural networks~\cite{hua2022high} and message passing simplicial networks~\cite{bodnar2021weisfeiler}.
\end{example}

\section{The singular vector bundle}
\label{sec:bundles}

In this section, we define the singular vector bundle. It is a vector bundle on $X = \prod_{i=1}^n \mathbb{P}(\mathbb{C}^{d_i})$ whose global sections are associated to hyperquiver representations. The zeros of a section are the singular vectors of the corresponding representation.

Following~\cite[Section 2]{friedland-ottaviani}, for each integer $d > 0$ we consider four vector bundles over $\mathbb{P}(\mathbb{C}^{d})$: the \textit{free} bundle $\mathscr{F}(d)$, the \textit{tautological} bundle $\mathscr{T}(d)$, the \textit{quotient} bundle $\mathscr{Q}(d)$, and the \textit{hyperplane} bundle $\mathscr{H}(d)$. Their fibres at each $[\mathbf{x}] \in \mathbb{P}(\mathbb{C}^{d})$ are

\begin{minipage}{0.35\linewidth}
    \begin{align*}{\label{fibres-of-line-bundles}}
\mathscr{F}(d)_{[\mathbf{x}]} &= \mathbb{C}^{d} \\
\mathscr{T}(d)_{[\mathbf{x}]} &= \Span  (\mathbf{x})
\end{align*}
\end{minipage}
\hfill
\begin{minipage}{0.45\linewidth}
    \begin{align*}
\mathscr{Q}(d)_{[\mathbf{x}]} &= \mathbb{C}^{d}/\Span  (\mathbf{x}) \\
\mathscr{H}(d)_{[\mathbf{x}]} &= \Span  (\mathbf{x})^{\vee}.
\end{align*}
\end{minipage}
\vspace{1em}

\noindent Here if $V$ is a vector space or vector bundle, then $V^\vee$ denotes its dual. Note that the hyperplane bundle $\mathscr{H}(d+1)$ is traditionally denoted in algebraic geometry by $\mathcal{O}_{\mathbb{P}^d}(1)$. We have a short exact sequence of vector bundles
\begin{equation}{\label{exact-sequence-tfq}}
0 \rightarrow \mathscr{T}(d_i) \rightarrow \mathscr{F}(d_i) \rightarrow \mathscr{Q}(d_i) \rightarrow 0.
\end{equation}
There are projection maps $\pi_i: X \rightarrow \mathbb{P}(\mathbb{C}^{d_i})$ with ${\pi_i(\mathbf{\chi}) = [\mathbf{x}_i]}$, where $\mathbf{\chi} = ([\mathbf{x}_1],\ldots,[\mathbf{x}_n])$. We pull back a vector bundle $\mathscr{B}$ over $\mathbb{P}(\mathbb{C}^{d_i})$ to a bundle $\pi_i^* \mathscr{B}$ over $X$, whose fiber at ${\mathbf{\chi}\in X}$ equals $\mathscr{B}_{[\x_i]}$. There is an exact sequence $0 \rightarrow \mathscr{T}(d_i)_{[\x_i]} \rightarrow \mathscr{F}(d_i)_{[\x_i]} \rightarrow \mathscr{Q}(d_i)_{[\x_i]} \rightarrow 0$ of vector spaces at every $[\x_i] \in \mathbb{P}(\mathbb{C}^{d_i})$. Hence there is an exact sequence of vector bundles
\begin{equation}{\label{exact-sequence-tfq-star}}
0 \rightarrow \pi_i^*\mathscr{T}(d_i) \rightarrow \pi_i^*\mathscr{F}(d_i) \rightarrow \pi_i^*\mathscr{Q}(d_i) \rightarrow 0.
\end{equation}

\begin{definition}{\label{def:singular-vector-bundle}}
Let $\mathbf{R} = (\mathbf{d},T)$ be a hyperquiver representation and let $X = \prod_{i=1}^n \mathbb{P}(\mathbb{C}^{d_i})$. 
For each hyperedge $e \in E$, we consider the following vector bundles over $X$.
$$\mathscr{T}(e) := \bigotimes_{j=1}^{\mu(e)} \pi_{s_j(e)}^*\mathscr{T}(d_{s_j(e)}), \qquad \mathscr{B}(e) := \Hom \left(\mathscr{T}(e),\pi_{t(e)}^*\mathscr{Q}(d_{t(e)})\right).$$
We define the \textit{singular vector bundle} of $\mathbf{R}$ over $X$ to be
$\mathscr{B}(\mathbf{R}) := \bigoplus_{e \in E} \mathscr{B}(e)$.
\end{definition}

The vector bundle $\mathscr{B}(\mathbf{R})$ depends on the hypergraph $H$ and the assigned vector spaces~$U$, but not on the multilinear maps $T$. 
It can be written in terms of a partition of edges as 
$\mathscr{B}(\mathbf{R})
 = \bigoplus_{r=1}^M\bigoplus_{e \in E_r}\mathscr{B}(e)$.
We will see that when $\mathbf{R}$ is a generic hyperquiver representation, the zero locus of a generic section of $\mathscr{B}(\mathbf{R})$ is the singular vector variety $\mathcal{S}(\mathbf{R})$. 
We make the following observations about its summands $\mathscr{B}(e)$.

\begin{proposition}
\label{prop:Be_properties}
Let $\mathscr{B}(e) = \Hom \left(\mathscr{T}(e),\pi_{t(e)}^*\mathscr{Q}(d_{t(e)})\right)$. Then the following hold.
\begin{itemize}
    \item[(a)] The fibre of $\mathscr{B}(e)$ at $\mathbf{\chi}$ is $\Hom \left(\Span  \left(\otimes_{j=1}^{\mu(e)} \mathbf{x}_{s_j(e)}\right),\mathbb{C}^{d_{t(e)}}/\Span  (\mathbf{x}_{t(e)})\right)$.
    \item[(b)] The bundle $\mathscr{B}(e)$ has rank $d_{t(e)} - 1$. 
    \item[(c)] We have the isomorphism $\mathscr{B}(e) = \left( \bigotimes_{j=1}^{\mu(e)} \pi_{s_j(e)}^*\mathscr{H}(d_{s_j(e)}) \right) \otimes \pi_{t(e)}^*\mathscr{Q}(d_{t(e)})$.
\end{itemize}
\end{proposition}

\begin{proof}
The bundle $\mathscr{T}(e)$ has fibres
\begin{align*}
\mathscr{T}(e)_{\mathbf{\chi}} &= \bigotimes_{j=1}^{\mu(e)} \pi_{s_j(e)}^*\mathscr{T}(d_{s_j(e)})_{\mathbf{\chi}} = 
\bigotimes_{j=1}^{\mu(e)} \mathscr{T}(d_{s_j(e)})_{[\mathbf{x}_{s_j(e)}]} \\ 
&= \bigotimes_{j=1}^{\mu(e)} \Span  (\mathbf{x}_{s_j(e)}) = 
\Span  \left(\otimes_{j=1}^{\mu(e)} \mathbf{x}_{s_j(e)}\right).
\end{align*}
The bundle $\pi_{t(e)}^*\mathscr{Q}(d_{t(e)})$ has fibre 
$\pi_{t(e)}^*\mathscr{Q}(d_{t(e)})_{\mathbf{\chi}} = \mathbb{C}^{d_{t(e)}}/\Span  (\mathbf{x}_{t(e)})$. This proves (a).
Then (b) follows, since the dimension of the fibre is $d_{t(e)} - 1$.
To prove (c), observe that
$\mathscr{B}(e) 
 \simeq \mathscr{T}(e)^\vee \otimes \pi_{t(e)}^*\mathscr{Q}(d_{t(e)})$ and that
\begin{align*}
\mathscr{T}(e)^\vee = \left(\bigotimes_{j=1}^{\mu(e)} \pi_{s_j(e)}^*\mathscr{T}(d_{s_j(e)})\right)^\vee &\simeq \bigotimes_{j=1}^{\mu(e)} \left(\pi_{s_j(e)}^*\mathscr{T}(d_{s_j(e)})\right)^\vee \\
&\simeq \bigotimes_{j=1}^{\mu(e)} \pi_{s_j(e)}^*\mathscr{T}(d_{s_j(e)})^\vee = \bigotimes_{j=1}^{\mu(e)} \pi_{s_j(e)}^*\mathscr{H}(d_{s_j(e)}). \qedhere
\end{align*}
\end{proof}

We relate the singular vector variety to the singular vector bundle.
The global sections of a vector bundle $\mathscr{B}$ are denoted by $\Gamma(\mathscr{B})$. They are the holomorphic maps $\sigma: X \to \mathscr{B}$ that send each $\mathbf{\chi} \in X$ to a point in $\mathscr{B}_\mathbf{\chi}$.
A global section of $\mathscr{B}(e)$ is a map sending each $\mathbf{\chi} \in X$ to an element in
\begin{equation}{\label{eq:T_r-notation}}
    \Hom \left(\Span  \left(\otimes_{j=1}^{\mu(e)} \mathbf{x}_{s_j(e)}\right),\mathbb{C}^{d_{t(e)}}/\Span  (\mathbf{x}_{t(e)})\right),
\end{equation} 
by Proposition~\ref{prop:Be_properties}(a). Definition \ref{def:generic}(ii) of a partition gives an equivalence relation between tensors assigned to $E_r$ via permutation of the modes. Following the notation of Definition \ref{def:generic}(iii), we denote by $T_r \in \mathbb{C}^{e_r}$ a representative for the class corresponding to $E_r$, for some $e_r \in E_r$, and we define $T_r(\mathbf{x}_{s(e)}) := T_e(\mathbf{x}_{s(e)})$ for all $e \in E_r$, where $T_e(\mathbf{x}_{s(e)})$ is defined in \eqref{eqn:te_xse}. A tensor $T \in \mathbb{C}^{e_r}$ determines a global section of $\mathscr{B}(e)$ for every $e \in E_r$, which we denote by $L_e(T)$. The map $L_e(T)$ sends $\mathbf{\chi}$ to the map
$$ 
\otimes_{j=1}^{\mu(e)} \mathbf{x}_{s_j(e)} \, \mapsto \, 
  \overline{T(\mathbf{x}_{s(e)})} \in \mathbb{C}^{d_{t(e)}}/\Span  (\mathbf{x}_{t(e)}) .
$$
where $\overline{T(\mathbf{x}_{s(e)})}$ is the image of $T(\mathbf{x}_{s(e)})$ in the quotient vector space $\mathbb{C}^{d_{t(e)}}/\Span  (\mathbf{x}_{t(e)})$.  In other words, following~\cite[Lemma 9]{friedland-ottaviani}, we define the map
\begin{align}
    \begin{split}
    L_e:\mathbb{C}^{e_r} & \longrightarrow \Gamma(\mathscr{B}(e)) \nonumber \\
    T & \longmapsto L_e(T).
\end{split}
\end{align}
We form the composite map
\begin{align}{\label{L-map}}
\begin{split}
L : \bigoplus_{r=1}^M \mathbb{C}^{e_r} & \longrightarrow \Gamma(\mathscr{B}(\mathbf{R})) \\ 
(T_1, \ldots, T_M) & \longmapsto \bigoplus_{r=1}^M \bigoplus_{e \in E_r} L_e(T_r).
\end{split}
\end{align}

We connect the global sections in the image of $L$ to the singular vector tuples of a hyperquiver representation, generalizing~\cite[Lemma 11]{friedland-ottaviani}.

\begin{proposition}{\label{prop:zeros-are-singular-vectors}}
Let $\mathbf{R} = (\mathbf{d},T)$ be a hyperquiver representation. Let $X = \prod_{i=1}^n \mathbb{P}(\mathbb{C}^{d_i})$ and let $\mathscr{B}(\mathbf{R})$ be the singular vector bundle, with $L: \bigoplus_{r = 1}^M \mathbb{C}^{e_r} \rightarrow \Gamma(\mathscr{B}(\mathbf{R}))$ the map in \eqref{L-map}. Then a point $\mathbf{\chi} \in X$ lies in the zero locus of the section $\sigma = L((T_r)_{r=1}^M)$ if and only if $\mathbf{\chi}$ is a singular vector tuple of $\mathbf{R}$.
\end{proposition}

\begin{proof}
$L((T_r)_{r=1}^M)(\mathbf{\chi})$ is the $|E|$-tuple of zero maps each in $\mathscr{B}(e)_{\mathbf{\chi}}$ if and only if for all $e \in E_r$ and $r \in [M]$, $L_e(T_r)(\mathbf{\chi})(\otimes_{j=1}^{\mu(e)} \mathbf{x}_{s_j(e)}) = \overline{0}$, if and only if $T_r(\mathbf{x}_{s(e)}) = \lambda_e \mathbf{x}_{t(e)}$ for some $\lambda_e \in \mathbb{C}$, if and only if $\mathbf{\chi}$ is a singular vector tuple of the hyperquiver representation $\mathbf{R}$.
\end{proof}

In light of the preceding result, it becomes necessary to determine the image of $L$ within $\Gamma(\mathscr{B}(\mathbf{R}))$. For this purpose, we make use of the following K\"{u}nneth formula for vector bundles. Note that $H^0(X,\mathscr{B}) := \Gamma(\mathscr{B})$.

\begin{proposition}[{K\"{u}nneth Formula, \cite[Proposition 9.2.4]{kempf_1993}}]{\label{prop:kunneth}}
Let $X$ and $Y$ be complex varieties and $\pi_X: X \times Y \rightarrow X$ and $\pi_Y: X \times Y \rightarrow Y$ be the projection maps. If $\mathscr{F}$ and $\mathscr{G}$ are vector bundles on $X$ and $Y$ respectively, then
$$H^n(X \times Y,\pi_X^*\mathscr{F} \otimes \pi_Y^*\mathscr{G}) \cong \bigoplus_{p+q=n} H^p(X,\mathscr{F}) \otimes H^q(Y,\mathscr{G}).$$
\end{proposition}

\noindent The following result, which generalises~\cite[Lemma 9 parts (1) and (2)]{friedland-ottaviani}, characterises the image of $L$.

\begin{proposition}{\label{prop:L-bijective}}
The linear map $L: \bigoplus_{r = 1}^M \mathbb{C}^{e_r} \rightarrow \Gamma(\mathscr{B}(\mathbf{R}))$ in \eqref{L-map} is bijective.
\end{proposition}

\begin{proof}
By the definition of $L$, it suffices to show for each $e \in E$ that $L_e$ is an injective linear map between vector spaces of the same dimension. First we show that $L_e$ is injective. Consider $e \in E_r$ and let $T \in \C^{e_r}$. If $T \neq 0$, then there exist $\mathbf{x}_{s_j(e)} \in \mathbb{C}^{d_{s_j(e)}}$ for $j \in [\mu(e)]$ with $\mathbf{v} := T(\mathbf{x}_{s(e)}) \neq 0$. Let $\mathbf{x}_{t(e)} \in \mathbb{C}^{d_{t(e)}} \setminus \Span(\mathbf{v})$. Then $L_e(T)(\mathbf{\chi})(\otimes_{j=1}^{\mu(e)} \mathbf{x}_{s_j(e)}) \neq \overline{0}$. Hence, the global section $L_e(T)$ is not the zero section.

We recursively apply the K\"{u}nneth formula in the case $n=0$ to obtain 
$$ H^0(X, \mathscr{B}(e)) = \bigotimes_{j=1}^{\mu(e)} H^0(X, \pi_{s_j(e)}^*\mathscr{H}(d_{s_j(e)})) \otimes H^0(X, \pi_{t(e)}^*\mathscr{Q}(d_{t(e)})).$$
It remains to compute the dimensions of the factors. We have $\dim H^0(X, \pi_i^* \mathscr{H}(d_i)) = d_i$ by results on the cohomology of line bundles over projective space \cite[Theorem 5.1]{Hartshorne}.
Finally, the short exact sequence \eqref{exact-sequence-tfq-star} gives a long exact sequence in cohomology
\[
0 \rightarrow \underbrace{H^0(X,\pi_i^*\mathscr{T}(d_i)}_{= 0}) \rightarrow H^0(X,\pi_i^*\mathscr{F}(d_i)) \rightarrow H^0(X,\pi_i^*\mathscr{Q}(d_i)) \rightarrow \underbrace{H^1(X,\pi_i^*\mathscr{T}(d_i))}_{= 0} \rightarrow \ldots.
\]
The underlined terms are 0, again by \cite[Theorem 5.1]{Hartshorne}.
Thus $\dim H^0(X, \pi_i^*\mathscr{Q}(d_i)) = d_i$, since $\dim H^0(X,\pi_i^*\mathscr{F}(d_i)) = d_i$. 
Hence $\dim H^0(X,\mathscr{B}(e)) = \prod_{j=1}^{m} d_{s_j(e)}$. This is the dimension of $\mathbb{C}^{e_r}$, so $L_e$ is a bijection.
\end{proof}

\section{Bertini-type theorem}
\label{sec:bertini} 

In this section, we relate the zeros of a generic section of a vector bundle to its top Chern class, cf.~\cite[Section 2.5]{friedland-ottaviani}. 
    This relation holds when the vector bundle is ``almost generated'', see Definition~\ref{def:almostGenerates}. 
    We refer the reader to Appendix~\ref{sec:intersection_theory} for relevant background on Chern classes and Chow rings.
    In this section, $X$ is any smooth complex projective variety. Recall that the global sections of $\mathscr{B}$, denoted $\Gamma(\mathscr{B})$, are the holomorphic maps $\sigma: X \to \mathscr{B}$ that send each $\mathbf{\chi} \in X$ to a point in the fibre $\mathscr{B}_\mathbf{\chi}$.
    
\begin{definition}{\label{def:generates}}
Let $X$ be a smooth projective variety and $\mathscr{B}$ a vector bundle over $X$. The vector bundle $\mathscr{B}$ is \textit{globally generated} if there exists a vector subspace $\Lambda \subseteq \Gamma(\mathscr{B})$ such that for all $\mathbf{\chi} \in X$, we have $\Lambda(\mathbf{\chi}) = \mathscr{B}_\mathbf{\chi}$, where $\Lambda(\mathbf{\chi}) := \{\sigma(\mathbf{\chi}) \; | \; \sigma \in \Lambda\}$.
\end{definition}

\begin{definition}{\label{def:almostGenerates}}
Let $X$ be a smooth projective variety and $\mathscr{B}$ a vector bundle over $X$. The vector bundle $\mathscr{B}$ is \textit{almost generated} if there exists a vector subspace $\Lambda \subseteq \Gamma(\mathscr{B})$ such that either $\mathscr{B}$ is globally generated, or there are $k \geq 1$ smooth irreducible proper subvarieties $Y_1,\ldots,Y_k$ of $X$, with $Y_0 = X$, such that:
\begin{enumerate}[label=(\roman*)]
    \item For all $i \geq 0$, there is a vector bundle $\mathscr{B}_i$ over $Y_i$, and for any $j \geq 0$, if $Y_i$ is a subvariety of $Y_j$, then $\mathscr{B}_i$ is a subbundle of $\mathscr{B}_j{\big|}_{Y_i}$
    \item $\Lambda(\mathbf{\chi}) \subseteq (\mathscr{B}_i)_\mathbf{\chi}$ for all $\mathbf{\chi} \in Y_i$ and $i \geq 0$
    \item For all $i \geq 0$, if $\alpha_i \subseteq [k]$ is the set of all $j \in [k]$ such that $Y_j$ is a proper subvariety of $Y_i$, then $\Lambda(\mathbf{\chi}) = (\mathscr{B}_i)_\mathbf{\chi}$ for all $\displaystyle \mathbf{\chi} \in Y_i \setminus \left(\cup_{j \in \alpha_i} Y_j\right).$
\end{enumerate}
\end{definition}

Now we state our Bertini-type theorem; cf.~\cite[Theorem 6]{friedland-ottaviani}. The zero locus of a section $\sigma \in \Gamma(\mathscr{B})$ is $Z(\sigma) := \{\mathbf{\chi} \in X \; | \; \sigma(\mathbf{\chi}) = 0\}$.
The top Chern class and top Chern number of $\mathscr{B}$, see Definition~\ref{def:chern-classes}, are denoted $c_r(\mathscr{B}) \in A^*(X)$ and $\nu(\mathscr{B}) \in \mathbb{Z}$, respectively. We assume $X \subseteq \mathbb{P}^D$ via some closed immersion $s: X \xhookrightarrow[]{} \mathbb{P}^D$, and identify the top Chern class of $\mathscr{B}$ with its pushforward under the embedding: $c_r(\mathscr{B}) = s_*(c_r(\mathscr{B})) \in A^*(\mathbb{P}^D)$, see Remark \ref{rem:closed-immersion-pushforward}.

\begin{theorem}[Bertini-Type Theorem]{\label{thm:bertiniTheorem}}
Let $X \subseteq \mathbb{P}^D$ be a smooth irreducible complex projective variety of dimension $d$, and $\mathscr{B}$ a vector bundle of rank $r$ over $X$, almost generated by a vector subspace $\Lambda \subseteq \Gamma(\mathscr{B})$. Let $\sigma \in \Lambda$ be a generic section with $Z(\sigma) \subseteq X$ its zero locus.
\begin{enumerate}[(a)]
    \item If $r > d$, then $Z(\sigma)$ is empty
    \item If $r = d$, then $Z(\sigma)$ consists of $\nu(\mathscr{B})$ points. Furthermore, if $\rk   \mathscr{B}_i > \dim Y_i$ for all $i \geq 1$, then each point has multiplicity 1 and does not lie on $\displaystyle \cup_{i = 1}^k Y_i$.
    \item If $r < d$, then $Z(\sigma)$ is empty or smooth of pure dimension $d - r$. In the latter case, the degree of $Z(\sigma)$ is $\nu\left(\mathscr{B}{\big|}_L\right)$, where $L \subseteq \mathbb{P}^D$ is the intersection of $d-r$ generic hyperplanes in $\mathbb{P}^D$. If $\nu\left(\mathscr{B}{\big|}_L\right) \neq 0$, then $Z(\sigma)$ is non-empty.
\end{enumerate}
\end{theorem}

\begin{remark}
The above theorem generalises~\cite[Theorem 6]{friedland-ottaviani}, where parts (a) and (b) appear. We add part (c). 
Compared to~\cite[Theorem 6]{friedland-ottaviani}, our extra assumption $\rk \mathscr{B}_i > \dim Y_i$ for $i > 0$ in (b) appears because it is absent from Definition~\ref{def:almostGenerates}, whereas it appears in its analogue~\cite[Definition 5]{friedland-ottaviani}.
\end{remark}

The first Chern class of the hyperplane bundle $c_1(\mathscr{H}(D+1)) = c_1(\mathcal{O}_{\mathbb{P}^D}(1))$ is the same as the class of a generic hyperplane in $\mathbb{P}^D$, by Definition \ref{def:chern-classes}(v). Thus, there will be a unique integer $N$ such that $c_r(\mathscr{B}) c_1(\mathcal{O}_{\mathbb{P}^D}(1))^N$ is equal to some integer multiple of the class of a point $[p] \in A^*(\mathbb{P}^D)$. Theorem \ref{thm:bertiniTheorem}(c) states that this integer is $\nu\left(\mathscr{B}{\big|}_L\right)$ and that this is the degree of $Z(\sigma)$. This integer is also the \textit{degree} of the class $c_r(\mathscr{B})$ \cite[Proposition 1.21]{eisenbud_harris_2016}.

To prove Theorem~\ref{thm:bertiniTheorem}, we use the following results.

\begin{theorem}[{Fiber Dimension Theorem ~\cite[Theorem 1.25]{shafarevich}}]{\label{fiberDimensionTheorem}}

Let $f: X \rightarrow Y$ be a dominant morphism of irreducible varieties. Then there exists an open set $U \subseteq Y$ such that for all $y \in U$, $\dim X = \dim Y + \dim(f^{-1}(y))$.

\end{theorem}

\begin{theorem}[{Generic Smoothness Theorem \cite[Corollary III.10.7]{Hartshorne}}]{\label{genericSmoothness}}
Let $f: X \rightarrow Y$ be a morphism of irreducible complex varieties. If $X$ is smooth, then there exists an open subset $U \subseteq Y$ such that $f|_{f^{-1}(U)}$ is smooth. Furthermore, if $f$ is not dominant, then $f^{-1}(U) = \varnothing$.
\end{theorem}

\begin{proof}[Proof of Theorem \ref{thm:bertiniTheorem}]
Consider $I = \{(\mathbf{\chi},\sigma) \in X \times \Lambda \; | \; \sigma(\mathbf{\chi}) = 0\}$ with projection maps
\begin{center}
\begin{tikzcd}
  & I \arrow[ld, "p"'] \arrow[rd, "q"] &   \\
X &                                    & \Lambda
\end{tikzcd}
\end{center}
Then $I$ is a vector bundle over $X$. Since the base space $X$ is irreducible, so is the total space~$I$.
We show that $\dim I = \dim \Lambda + d - r$. 
The map $p$ is surjective, and hence dominant, since the zero section lies in $\Lambda$. There exists an open set $U \subseteq X$ such that $\dim I = d + \dim (p^{-1}(\mathbf{\chi}))$ for all $\mathbf{\chi} \in U$, 
 by Theorem~\ref{fiberDimensionTheorem}.
The fibre $p^{-1}(\mathbf{\chi}) \simeq \{ \sigma \in \Lambda : \sigma(\mathbf{\chi}) = 0 \}$ consists of sections in $\Lambda$ that vanish at $\mathbf{\chi}$.
Consider the evaluation map $\{\mathbf{\chi}\} \times \Lambda \rightarrow \mathscr{B}_\mathbf{\chi}$ that sends 
$(\mathbf{\chi},\sigma)$ to $\sigma(\mathbf{\chi})$.
This is a linear map of vector spaces and its kernel is isomorphic to $p^{-1}(\mathbf{\chi})$.
Let $Y := \cup_{i = 1}^k Y_i$, where the $Y_i$ are from Definition~\ref{def:almostGenerates}.
The variety $Y$ is a proper subvariety of $X$. For each $\mathbf{\chi} \in X \setminus Y$, the evaluation map is surjective, by Definition \ref{def:almostGenerates}(iii). Thus, the evaluation map has rank $r$ and nullity $\dim \Lambda - r$. Hence $\dim(p^{-1}(\mathbf{\chi})) = \dim \Lambda - r$
for all $\mathbf{\chi} \in U \cap (X \setminus Y)$.
Therefore $\dim I = \dim \Lambda + d - r$. 

The fiber $q^{-1}(\sigma) \simeq \{ \mathbf{\chi} \in X : \sigma(\mathbf{\chi}) = 0 \}$ is the zero locus $Z(\sigma)$. We show that the map $q$ is dominant if and only if $q^{-1}(\sigma) \neq \varnothing$ for generic $\sigma \in \Lambda$. If $q$ is dominant, then there exists an open set $W \subseteq \Lambda$ such that $q^{-1}(\sigma)$ is smooth of codimension $\dim I - \dim \Lambda = d - r$ for all $\sigma \in W$, by Theorems~\ref{fiberDimensionTheorem} and \ref{genericSmoothness}. In particular, $q^{-1}(\sigma)$ is non-empty. Conversely if $q$ is not dominant, then there is an open set $W \subseteq \Lambda$ such that $q^{-1}(\sigma) = \varnothing$ for all $\sigma \in W$, by Theorem \ref{genericSmoothness}.

Now we show that $Z(\sigma) \neq \varnothing$ for generic $\sigma \in \Lambda$ if and only if $c_r(\mathscr{B}) \neq 0$. If $Z(\sigma) = \varnothing$, then the existence of a nowhere vanishing section of $\mathscr{B}$ implies that $c_r(\mathscr{B}) = 0$ \cite[Lemma 3.2]{fulton_1998}. Conversely, if $Z(\sigma) \neq \varnothing$, then the map $q$ is dominant, so $Z(\sigma)$ is smooth of codimension $d-r$. If $c_r(\mathscr{B}) = 0$, then $0 = c_r(\mathscr{B}) = [Z(\sigma)]$ by Definition \ref{def:chern-classes}(ii), which is a contradiction since the degree of a non-empty projective variety is a positive integer \cite[Proposition I.7.6.a]{Hartshorne}. In particular, if $r = d$ and $\nu(\mathscr{B}) = 0$, then $Z(\sigma) = \varnothing$.

The map $q$ is not dominant if 
$\dim I < \dim \Lambda $; i.e., if 
$r > d$. This proves (a) and the emptiness possibility in (c).
It remains to consider the case $r \leq d$ with the map $q$ dominant and generic $\sigma \in \Lambda$.

$Z(\sigma) \subseteq \mathbb{P}^D$ is smooth of dimension $d-r$. It is pure dimensional by \cite[Example 3.2.16]{fulton_1998}. When $r = d$, we have $[Z(\sigma)] = c_r(\mathscr{B}) = \nu(\mathscr{B})[p]$ for some $p \in X$, by Definition \ref{def:chern-classes}(ii), so the zero locus consists of $\nu(\mathscr{B})$ points.  It remains to relate the degree to the top Chern class for $r < d$. The degree of $Z(\sigma)$ is the number of points in the intersection of $Z(\sigma)$ with $d-r$ generic hyperplanes $\mathbb{P}^D$. Denote the intersection of $d-r$ such hyperplanes by $L$. Let $L \xhookrightarrow[]{j} \mathbb{P}^D$ be its inclusion. We have $[Z(\sigma)] = c_r(\mathscr{B})$ by Definition \ref{def:chern-classes}(ii) and seek $[L] c_r(\mathscr{B})$. We compute in $A^*(\mathbb{P}^D)$:
\begin{align}
{\label{topChernNumberCalculation}}
[L]c_r(\mathscr{B}) &= j_*([L])c_r(\mathscr{B}) && \text{(definition of pushforward)} \nonumber \\
&= j_*(j^*(c_r(\mathscr{B})) [L]) && \text{(projection formula)} \nonumber \\
&= j_*(c_r(j^*\mathscr{B})  [L]) = j_*(c_r\left(\mathscr{B}\big|_L\right) [L]) && \text{(Definition \ref{def:chern-classes}(iv))} \nonumber \\
&= j_*(\nu\left(\mathscr{B}\big|_L\right)[p] [L]) && \text{(definition of top Chern number)} \nonumber \\
&= \nu\left(\mathscr{B}\big|_L\right)j_*([p] [L]) && \text{(pushforward is a morphism)} \nonumber \\
&= \nu\left(\mathscr{B}\big|_L\right)j_*([p]) = \nu\left(\mathscr{B}\big|_L\right)[p] && \text{(intersection with a point)} 
\end{align}
for some point $p \in L$. Thus, the degree of $Z(\sigma)$ is $\nu\left(\mathscr{B}{\big|}_L\right)$. As a corollary, we obtain that if $\nu(\mathscr{B}) \neq 0$ or $\nu\left(\mathscr{B}{\big|}_L\right) \neq 0$, then $Z(\sigma) \neq \varnothing$. This proves the dimension and degree statements in (b) and (c).

Lastly, we show that when $r=d$ and the additional assumptions of (b) hold, the points in $Z(\sigma)$ are generically of multiplicity $1$ and do not lie on $Y$.
Smoothness in Theorem~\ref{genericSmoothness} shows that each of the finitely many points in $q^{-1}(\sigma)$ are of multiplicity 1.
We have $\rk \mathscr{B}_i > \dim Y_i$ for all $i \geq 1$. Hence $\dim(p^{-1}(Y_i)) = \dim Y_i + \dim \Lambda - \rk \mathscr{B}_i < \dim \Lambda$. Thus, $\dim(p^{-1}(Y)) < \dim \Lambda$, and using the fact that the projection $\mathbb{P}^n \times \mathbb{A}^m \rightarrow \mathbb{A}^m$ is a closed map, we deduce that $q$ is a closed map.
Hence $q(p^{-1}(Y))$ is a proper subvariety of $\Lambda$.
For all points in the open set $\sigma \in W \cap W^\prime$, where $W^\prime = \Lambda \setminus q(p^{-1}(Y))$, the fibre $q^{-1}(\sigma)$ contains no points in $Y$.
\end{proof}

\begin{remark}
     Our proof of Theorem~\ref{thm:bertiniTheorem}, is analogous to the proofs in~\cite{friedland-ottaviani} of their Theorems 4 and 6. Their proof uses \cite[Example 3.2.16]{fulton_1998}, which is equivalent to axiom (ii) in Definition~\ref{def:chern-classes}. Our proof adds the Chern number computation for case (c).
\end{remark}

\section{Generating the singular vector bundle}
\label{sec:almost_generated}

In this section we show that $\mathscr{B}(\mathbf{R})$ is almost generated, so that Theorem~\ref{thm:bertiniTheorem} may be applied to it. 
We generalise the singular vector bundle to a bundle $\mathscr{B}(\mathbf{R},F)$, for a subset of hyperedges $F \subseteq E$. 
The zeros of a global section of $\mathscr{B}(\mathbf{R},F)$ are singular vectors with singular value zero along the edges in $F$.
We show that $\mathscr{B}(\mathbf{R},F)$ is almost generated. This will later yield not only the dimension and degree of the singular vector variety $\mathcal{S}(\mathbf{R})$ in Theorem~\ref{thm:main}, but also the final statement about the non-existence of a zero singular value.

\begin{definition}{\label{def:singular-vector-bundle-with-F}}
Let $\mathbf{R} = (\mathbf{d},T)$ be a hyperquiver representation and let $X = \prod_{i=1}^n \mathbb{P}(\mathbb{C}^{d_i})$. Given $F \subseteq E$, we define
$$\mathscr{B}(e,F) = \begin{cases} \Hom \left(\mathscr{T}(e),\pi_{t(e)}^*\mathscr{Q}(d_{t(e)})\right) & \text{if } e \notin F \\ \Hom \left(\mathscr{T}(e),\pi_{t(e)}^*\mathscr{F}(d_{t(e)})\right) & \text{if } e \in F .\end{cases}$$
It has fibres
\begin{align*}
\mathscr{B}(e,F)_{\mathbf{\chi}} &= \begin{cases} \Hom \left(\Span  \left(\otimes_{j=1}^{\mu(e)} \mathbf{x}_{s_j(e)}\right),\mathbb{C}^{d_{t(e)}}/\Span  (\mathbf{x}_{t(e)})\right) & \text{if } e \notin F \\ \Hom \left(\Span  \left(\otimes_{j=1}^{\mu(e)} \mathbf{x}_{s_j(e)}\right),\mathbb{C}^{d_{t(e)}}\right) & \text{if } e \in F ,\end{cases}
\end{align*}
where $\mathbf{\chi} = ([\mathbf{x}_1],\ldots,[\mathbf{x}_n])$. The \textit{singular vector bundle} of $\mathbf{R}$ over $X$ with respect to $F$ is
$\mathscr{B}(\mathbf{R},F) = \bigoplus_{e \in E} \mathscr{B}(e,F)$.
\end{definition}

The singular vector bundle $\mathscr{B}(\mathbf{R})$ from Definition~\ref{def:singular-vector-bundle} is $\mathscr{B}(\mathbf{R},\varnothing)$.

\begin{proposition}{\label{prop:rank-of-singular-vector-bundle-with-F}}
The bundle $\mathscr{B}(\mathbf{R},F)$ has rank $\sum_{e \in E} (d_{t(e)} - 1) + |F|$.
\end{proposition}

\begin{proof}
The rank of $\mathscr{B}(\mathbf{R},F)$ is $\sum_{e \in E} \rk \mathscr{B}(e,F)$.
For $e \notin F$, 
$\rk \mathscr{B}(e,F) = d_{t(e)} - 1$, as in Proposition~\ref{prop:Be_properties}(b).
For $e \in F$,
$\rk \mathscr{B}(e,F) = \rk \Hom (\mathscr{T}(e),\pi_{t(e)}^*\mathscr{F}(d_{t(e)})) = d_{t(e)}$.
\end{proof}

We construct global sections for $\mathscr{B}(\mathbf{R},F)$ whose zero loci correspond to singular vectors with zero singular value along the edges in $F$.
Define the map
\begin{align}{\label{L_e-map-F}}
    L_{e,F}:\mathbb{C}^{e_r} &\longrightarrow \Gamma(\mathscr{B}(e,F)) \nonumber \\
    L_{e,F}(T)(\mathbf{\chi})(\otimes_{j=1}^{\mu(e)} \mathbf{x}_{s_j(e)}) &= \begin{cases}
    \overline{T(\mathbf{x}_{s(e)})} \in \mathbb{C}^{d_{t(e)}}/\Span  (\mathbf{x}_{t(e)}) & e \notin F \\
 T(\mathbf{x}_{s(e)}) \in \mathbb{C}^{d_{t(e)}} & e \in F.
    \end{cases}
\end{align}
We define the composite map
\begin{align}{\label{L-map-F}}
L_F : \bigoplus_{r = 1}^M \mathbb{C}^{e_r} &\rightarrow \Gamma(\mathscr{B}(\mathbf{R},F)) \\
L_F = &\bigoplus_{r = 1}^M \bigoplus_{e \in E_r} L_{e,F}. \nonumber
\end{align}

We connect the global sections in the image of $L_F$ to the singular vector tuples of $\mathbf{R}$, generalizing Proposition \ref{prop:zeros-are-singular-vectors} and~\cite[Lemma 11]{friedland-ottaviani}.

\begin{proposition}{\label{prop:zeros-are-singular-vectors-F}}
Let $\mathscr{B}(\mathbf{R},F)$ be the singular vector bundle with respect to $F$ and $L_F: \bigoplus_{r = 1}^M \mathbb{C}^{e_r} \rightarrow \Gamma(\mathscr{B}(\mathbf{R},F))$ the linear map in \eqref{L-map-F}. A point $\mathbf{\chi} = ([\mathbf{x}_1],\ldots,[\mathbf{x}_n]) \in X$ lies in the zero locus of the section $\sigma = L_F((T_r)_{r=1}^M)$ if and only if $\mathbf{\chi}$ is a singular vector tuple of $\mathbf{R}$ with zero singular value along all edges in $F$.
\end{proposition}

\begin{proof}
The image
$L_F((T_r)_{r=1}^M)( \mathbf{\chi})$ is the tuple of zero maps each in $\mathscr{B}(e,F)_\mathbf{\chi}$ if and only if for all $e \in E_r$ and $r \in [M]$, $L_{e,F}(T_r)(\mathbf{\chi})(\otimes_{j=1}^{\mu(e)} \mathbf{x}_{s_j(e)})$ is the zero vector in the appropriate case of \eqref{L_e-map-F}, if and only if $T_r(\mathbf{x}_{s(e)}) = \lambda_e \mathbf{x}_{t(e)}$ for some $\lambda_e \in \mathbb{C}$ with $\lambda_e = 0$ if $e \in F$, if and only if $\mathbf{\chi}$ is a singular vector tuple of the hyperquiver representation $\mathbf{R}$, with zero singular values along the edges of $F$.
\end{proof}

\begin{definition}    \label{def:isotopicQuadric}
The \textit{isotropic quadric} $Q_n = \{\mathbf{v}\in \mathbb{C}^n : \mathbf{v}^\top \mathbf{v} = 0\}$ is the quadric hypersurface in $\mathbb{C}^n$ of isotropic vectors. Here, the contraction $\mathbf{v}^\top\mathbf{v}$ is with respect to the standard real inner product, see Remark \ref{remark:choice-of-bases}. The variety $Q_n$ is defined by a homogeneous equation and hence has dimension $n-1$. We consider it as a subvariety $\mathbb{P}(Q_n)$ of $\mathbb{P}^n$.
\end{definition}

\begin{definition}
If $T \in \mathbb{C}^e$ is a tensor and $\mathbf{x}_{s_j(e)} \in \mathbb{C}^{d_{s_j(e)}}$ are vectors for $j \in \{0,1,\ldots,\mu\}$, then we denote by $T(\mathbf{x_e}) := T_e (\mathbf{x}_{s_0(e)}, \x_{s_1(e)}, \ldots, \x_{s_\mu(e)}) = \mathbf{x}_{t(e)}^\top T(\mathbf{x}_{s(e)}) \in \mathbb{C}$ the contraction of the tensor $T$ by the vectors $\mathbf{x}_{s_j(e)}$, where $T(\mathbf{x}_{s(e)})$ is the vector defined in \eqref{eqn:te_xse}.
\end{definition}

We give a necessary and sufficient condition for when the maps in~\eqref{L_e-map-F} generate the vector space $\mathscr{B}(e)_{\mathbf{\chi}}$. 
This generalises~\cite[Lemma 8]{friedland-ottaviani} from a single tensor to a hyperquiver representation.
Later, in our proof that $\mathscr{B}(\mathbf{R},F)$ is almost generated, we apply this condition to the vector subbundles $\mathscr{B}_i$ in Definition \ref{def:almostGenerates}. This will allow us to associate a single tensor to each piece of the partition.
 
\begin{lemma}{\label{lem:isotropic-lemma}}
Let $H = (V,E)$ be a hyperquiver, $E = \coprod_{r=1}^M E_r$ be a partition, and assign vector spaces $\mathbb{C}^{d_i}$ to each vertex $i \in V$. Fix a collection of vectors $\mathbf{x}_i \in \mathbb{C}^{d_i}\setminus\{0\}$ for $i \in [n]$ and $\mathbf{y}_e \in \mathbb{C}^{d_{t(e)}}$ for $e \in E$. Fix $F \subseteq E$ a subset of hyperedges. Let $G_r$ be the hyperedges $e \in E_r \setminus F$ such that $\mathbf{x}_{t(e)}$ is isotropic. Then for all $r \in [M]$, the following are equivalent:

\begin{enumerate}[(a)]
    \item There exist tensors $T_r \in \mathbb{C}^{e_r}$ for some $e_r \in E_r$ satisfying the equations
\begin{align}
\overline{T_r(\mathbf{x}_{s(e)})} &= \overline{\mathbf{y}_e} \in \mathbb{C}^{d_{t(e)}}/\Span  (\mathbf{x}_{t(e)}) && e \in E_r \setminus F {\label{tensorEquation}} \\
T_r(\mathbf{x}_{s(e)}) \; &= \mathbf{y}_e \in \mathbb{C}^{d_{t(e)}} && e \in E_r \cap F. {\label{tensorEquation2}}
\end{align}

\item 
Given any pair of edges $e, e^\prime \in (F \cap E_r) \cup G_r$, we have
\begin{equation}{\label{isotropicEquations}}
\mathbf{x}_{t(e)}^\top \mathbf{y}_e = \mathbf{x}_{t(e^\prime)}^\top \mathbf{y}_{e^\prime}.
\end{equation}

\end{enumerate}
\end{lemma}
\begin{proof}
$(a \Rightarrow b):$ There is a tensor $T_r$ satisfying \eqref{tensorEquation} if and only if there are scalars $\lambda_e \in \mathbb{C}$ such that $T_r(\mathbf{x}_{s(e)})  = \mathbf{y}_e + \lambda_e \mathbf{x}_{t(e)}$ for all $e \in E_r \setminus F$. 
Multiplying both sides by 
$\mathbf{x}_{t(e)}$ gives $T_r(\mathbf{x}_e) = \mathbf{x}^\top_{t(e)}\mathbf{y}_e + \lambda_e \mathbf{x}_{t(e)}^\top \mathbf{x}_{t(e)}$.
Similarly, from \eqref{tensorEquation2} we obtain, for $e \in F \cap E_r$, the condition $T_r(\mathbf{x}_e) = \mathbf{x}^\top_{t(e)}\mathbf{y}_e$. 
The scalar $T_r(\mathbf{x}_e)$ only depends on $r$ via the set $E_r$.
Thus for any pair of edges $e,e^\prime \in E_r$, we have
\begin{equation*}
\mathbf{x}_{t(e)}^\top \mathbf{y}_e + \lambda_e \mathbf{x}_{t(e)}^\top \mathbf{x}_{t(e)} = \mathbf{x}_{t(e^\prime)}^\top \mathbf{y}_{e^\prime} + \lambda_{e^\prime} \mathbf{x}_{t(e^\prime)}^\top \mathbf{x}_{t(e^\prime)}
\end{equation*}
where $\lambda_e = 0$ for $e \in F \cap E_r$. 
For the hyperedges in $G_r$, the terms 
$\mathbf{x}_{t(e)}^\top \mathbf{x}_{t(e)}$ vanish.
Hence \eqref{isotropicEquations} holds for all $e,e^\prime \in (F \cap E_r) \cup G_r$.

$(b \Rightarrow a):$ 
Let $\mu_r \in \mathbb{C}$ be the value of \eqref{isotropicEquations}
if $(F \cap E_r) \cup G_r \neq \varnothing$ and zero otherwise. Define 
\begin{equation}{\label{eq:lambda-e}}
    \lambda_e = \begin{cases} 0 & e \in (F \cap E_r) \cup G_r \\ 
({\mathbf{x}_{t(e)}^\top \mathbf{x}_{t(e)}})^{-1}(\mu_r - \mathbf{x}_{t(e)}^\top \mathbf{y}_e) & \text{otherwise}.\end{cases}
\end{equation} 
Choose some $e_r \in E_r$. We show that, for such a choice of $\lambda_e$, there exists a tensor $T_r \in \C^{e_r}$ that satisfies 
\begin{equation}
    \label{eqn:Tr_cond}
    T_r(\mathbf{x}_{s(e)})  = \mathbf{y}_e + \lambda_e \mathbf{x}_{t(e)} 
\end{equation}
for all $e \in E_r$, 
and hence there exists a tensor $T_r$ that satisfies \eqref{tensorEquation} and \eqref{tensorEquation2}.
A change of basis in each $\mathbb{C}^{d_i}$ does not affect the existence or non-existence of solutions to \eqref{eqn:Tr_cond}. Consider the change of basis that sends each $\x_i$ to the first standard basis vector in $\mathbb{C}^{d_i}$, which we denote
by $\mathbf{e}_{i,1} = (1,0,\ldots,0)^\top$. For each $e \in E_r$, there is a permutation $\sigma$ of $\{0,1,\ldots,\mu\}$ sending $v(e)$ to $v(e_r)$ by Definition \ref{def:generic}(i.b). Then~\eqref{eqn:Tr_cond} becomes the condition
\[ \left(T_r \right)_{1, \ldots, 1, \ell,1,\ldots,1} = (\mathbf{y}_e)_\ell + \lambda_e \delta_{1, \ell} \text{ for all } \ell \in [d_{t(e)}] ,\]
where $\delta_{i,j}$ is the Kronecker delta and the $\ell$ on the left hand side appears in position $\sigma(0)$.
We define $T_r$ to be the tensor whose non-zero entries are given by the above equation. This is well-defined since by Definition \ref{def:generic}(i.c) setting $e^\prime = e_r$, we have $\sigma(0) \neq \sigma^\prime(0)$.
It remains to show that we do not attempt to assign different values to the same entry of $T_r$. When $\ell = 1$, we assign the value 
$(\mathbf{y}_e)_1 + \lambda_e$. By \eqref{eq:lambda-e}, if $e \in (F \cap E_r) \cup G_r$, then $(\mathbf{y}_e)_1 + \lambda_e = (\mathbf{y}_e)_1 + 0 = (\mathbf{y}_e)_1 = \mu_r$ by \eqref{isotropicEquations} and the definition of $\mu_r$, and otherwise, $(\mathbf{y}_e)_1 + \lambda_e = (\mathbf{y}_e)_1 + (\mu_r - (\mathbf{y}_e)_1) = \mu_r$. Thus, for all edges $e \in E_r$, $(\mathbf{y}_e)_1 + \lambda_e = \mu_r$.
\end{proof}

To conclude this section, we 
show that $\mathscr{B} := \mathscr{B}(\mathbf{R},F)$ 
satisfies the conditions of Definition~\ref{def:almostGenerates}.
This shows that $\mathscr{B}$ is almost generated.
First we define the subvarieties $Y_i$ and the vector bundles $\mathscr{B}_i$ over $Y_i$ that appear in Definition~\ref{def:almostGenerates}.

We use the following notation. A linear functional $\varphi: \mathbb{C}^{d_{t(e)}}/\Span(\mathbf{x}_{t(e)}) \rightarrow \mathbb{C}$ can be uniquely represented by a vector $\mathbf{u} \in \mathbb{C}^{d_{t(e)}}$ such that $\mathbf{u}^\top \mathbf{x}_{t(e)} = 0$ and $\varphi([\mathbf{y}]) = \mathbf{u}^\top\mathbf{y}$,~\cite[Lemma 7]{friedland-ottaviani}. In particular when $\mathbf{x}_{t(e)} \in Q_{t(e)}$, we abbreviate $\mathbf{x}_{t(e)}^\top [\mathbf{y}]$ to $\mathbf{x}_{t(e)}^\top \mathbf{y}$.

For a subset $\alpha \subseteq [n]$,
define the smooth proper irreducible subvariety 
\[ Y_\alpha = X_1 \times \cdots \times X_n , \quad \text{where} \quad X_i = \begin{cases} \mathbb{P}(Q_i) & i \in \alpha \\ 
\mathbb{P}(\mathbb{C}^{d_i}) & i \notin \alpha.\end{cases} \]
In particular, $Y_\varnothing = X$.
Fix $F \subseteq E$ and define $F^\prime = \{t(e)\}_{e \in F}$. 
Fix $\alpha \subseteq [n] \setminus F^\prime$.
 Let $G_r \subseteq E_r \setminus F$ denote the edges whose target vertex lies in $\alpha$.
Define $\mathscr{B}_\alpha$ to be the vector bundle over $Y_\alpha$ whose fiber at $\mathbf{\chi} = ([\mathbf{x}_1],\ldots,[\mathbf{x}_n]) \in Y_\alpha$ is the subspace $U(\alpha,\mathbf{\chi})$ of linear maps $\tau = (\tau_e)_{e \in E} \in (\mathscr{B})_\mathbf{\chi}$ 
satisfying 
\begin{equation}{\label{tau-equations}}
\mathbf{x}_{t(e)}^\top \tau_{e}(\otimes_{j=1}^{\mu(e)} \mathbf{x}_{s_j(e)}) = \mathbf{x}_{t(e^\prime)}^\top \tau_{e^\prime}(\otimes_{j=1}^{\mu(e^\prime)} \mathbf{x}_{s_j(e^\prime)}),
\end{equation}
for any edges $e,e^\prime \in (F \cap E_r) \cup G_r$, for every $r \in [M]$.

\begin{proposition}\label{prop:almost_generates}
Let the map $L_F$ be as in \eqref{L-map-F}.
For any subset of hyperedges $F \subseteq E$, the vector subspace
$\Lambda = \;L_F\left(\bigoplus_{r = 1}^M \mathbb{C}^{e_r}\right)$ almost generates $\mathscr{B}(\mathbf{R},F)$.
\end{proposition}

\begin{proof}

We first show that the vector bundles $\mathscr{B}_\alpha$
satisfy Definition \ref{def:almostGenerates}(i). If $\alpha, \beta \subseteq [n]\setminus F^\prime$, then $\alpha \subsetneq \beta$ if and only if $Y_\beta$ is a proper subvariety of $Y_\alpha$. Furthermore, $\mathscr{B}_\beta$ is a subbundle of $\mathscr{B}_\alpha\big\vert_{Y_\alpha}$, since $U(\beta,\mathbf{\chi})$ is a vector subspace of $U(\alpha,\mathbf{\chi})$.

Next we prove that Definition \ref{def:almostGenerates}(ii) holds.
Recall that $\Lambda(\mathbf{\chi}) := \{\sigma(\mathbf{\chi}) \; | \; \sigma \in \Lambda\}$. We show that $\Lambda(\mathbf{\chi}) \subseteq (\mathscr{B}_\alpha)_\mathbf{\chi}$.
If $\mathbf{\chi} \in Y_\alpha$, then an element of $\Lambda(\mathbf{\chi})$ is an $|E|$-tuple of linear maps $L_{e,F}(T_r)(\mathbf{\chi})$ for some tensors $T_r \in \mathbb{C}^{e_r}$, $r \in [M]$. By the proof of $(a \Rightarrow b)$ in Lemma \ref{lem:isotropic-lemma}, $\tau_e := L_{e,F}(T_r)(\mathbf{\chi})$ satisfy \eqref{tau-equations}, so $\Lambda(\mathbf{\chi}) \subseteq (\mathscr{B}_\alpha)_\mathbf{\chi}$.

Finally we show that Definition \ref{def:almostGenerates}(iii) holds.
If $\mathbf{\chi}$ lies on $Y_\alpha$ but not on any proper subvariety $Y_\beta$, then every $(\tau_e)_{e \in E} \in (\mathscr{B}_\alpha)_\mathbf{\chi}$ satisfies~\eqref{tau-equations} and no additional equations. Thus there exist tensors $T_r$ with $L_{e,F}(T_r) = \tau_e$ for $e \in E_r$ and $\tau \in \Lambda(\mathbf{\chi})$,
by Lemma \ref{lem:isotropic-lemma}. Hence, $\Lambda(\mathbf{\chi}) = (\mathscr{B}_\alpha)_\mathbf{\chi}$.
\end{proof}

\section{The top Chern class of the singular vector bundle}
\label{sec:top_chern}

In this section we compute the top Chern class of the singular vector bundle $\mathscr{B}(\mathbf{R})$, generalizing~\cite[Lemma 3]{friedland-ottaviani}. Combining this computation with Theorem~\ref{thm:bertiniTheorem} and Proposition~\ref{prop:almost_generates} finds the degree of the singular vector variety, 
completing the proof of Theorem~\ref{thm:main}.

\begin{proposition}{\label{prop:top-chern-class}}
Let $\mathbf{R} = (\mathbf{d},T)$ be a hyperquiver representation and $\mathscr{B}(\mathbf{R})$ be the singular vector bundle over $X = \prod_{i=1}^n \mathbb{P}(\mathbb{C}^{d_i})$. Then the top Chern class of $\mathscr{B}(\mathbf{R})$ is
$$\prod_{e \in E}\sum_{k = 1}^{d_{t(e)}} h_{t(e)}^{k-1} h_{s(e)}^{d_{t(e)}-k} , \quad \text{where} \quad h_{s(e)} = \sum_{j=1}^{\mu(e)} h_{s_j(e)}, $$
in the Chow ring $A^*(X) \cong \mathbb{Z}[h_1,\ldots,h_n]/(h_1^{d_1},\ldots,h_n^{d_n})$.
\end{proposition}

\begin{proof}
We seek the Chern polynomial $C(t,\mathscr{B}(\mathbf{R}))$. The coefficient of its highest power of $t$ is the top Chern class.
The Chern polynomial is multiplicative over short exact sequences, see Definition~\ref{def:chern-classes}(iii). Hence
\begin{equation}{\label{chern-polynomial-exact}}
C(t,\mathscr{F}(d)) = C(t,\mathscr{T}(d))C(t,\mathscr{Q}(d)),
\end{equation}
by \eqref{exact-sequence-tfq}.
We compute $C(t,\mathscr{T}(d))$.
Let $h \in A^*(\mathbb{P}(\mathbb{C}^{d})) \cong \mathbb{Z}[h]/(h^d)$ be the class of a hyperplane in $\mathbb{P}(\mathbb{C}^{d})$. By Definition \ref{def:chern-classes}(i)-(iii), $h$ is the first Chern class $c_1(\mathscr{H}(d))$ and the Chern polynomial of $\mathscr{H}(d)$ is $C(t, \mathscr{H}(d)) = 1 + h t$. Thus $C(t,\mathscr{T}(d)) = C(-t,\mathscr{H}(d)^\vee) = 1-h t$, by 
Proposition~\ref{chernProperties}(b).

Next we compute $C(t,\mathscr{Q}(d))$.
We have $C(t,\mathscr{F}(d)) = 1$, by Proposition~\ref{chernProperties}(a).
The Chern polynomial of $\mathscr{Q}(d)$ is the inverse of $(1 - h t)$, 
by \eqref{chern-polynomial-exact}.
Using the formal factorisation $1 - x^n = \prod_{k = 0}^n (1 - \zeta_n^kx)$ over $A^*(X)\otimes\mathbb{C}$, we therefore have
$$C(t,\mathscr{Q}(d)) = \sum_{j = 0}^{d-1} (ht)^j = \frac{1 - (ht)^{d-1}}{1 - ht} = \frac{\prod_{k = 0}^{d-1}(1 - \zeta_{d}^kht)}{1 - ht} = \prod_{k = 1}^{d-1} (1-\zeta_{d}^kht)$$
where $\zeta_{d} \in \mathbb{C}$ is a $d$-th root of unity.

We have $c_1(\pi_i^*\mathscr{H}(d_i)) = \pi_i^*c_1(\mathscr{H}(d_i)) = \pi_i^*h_i = h_i \in A^*(X)$, by Definition \ref{def:chern-classes}(iv) and Definition \ref{pushforwardPullback}(ii). Thus the Chern polynomials of $\pi_i^*\mathscr{H}(d_i)$, $\pi_i^*\mathscr{T}(d_i)$, and $\pi_i^*\mathscr{Q}(d_i)$ equal those of $\mathscr{H}(d)$, $\mathscr{T}(d)$, and $\mathscr{Q}(d)$ respectively but with $h$ replaced by $h_i \in A^*(X)$, by \eqref{exact-sequence-tfq-star}.

We have found the Chern roots of $\pi_i^*\mathscr{H}(d_i)$ and $\pi_i^*\mathscr{Q}(d_i)$, so we obtain Chern characters
$\ch(\pi_i^*\mathscr{H}(d_i)) = \exp(h_i)$ and $\ch(\pi_i^*\mathscr{Q}(d_i)) = \sum_{k = 1}^{d_i-1}\exp (-\zeta_{d_i}^kh_i)$.
By Propositions \ref{prop:Be_properties}(c) and \ref{chernProperties}(c), the Chern character $\ch(\mathscr{B}(e))$ equals
\begin{align*} \ch\left(\bigotimes_{j=1}^{\mu(e)} \pi_{s_j(e)}^*\mathscr{H}(d_{s_j(e)}) \otimes \pi_{t(e)}^*\mathscr{Q}(d_{t(e)})\right) & = \ch\left(\bigotimes_{j=1}^{\mu(e)} \pi_{s_j(e)}^*\mathscr{H}(d_{s_j(e)})\right)\ch(\pi_{t(e)}^*\mathscr{Q}(d_{t(e)})) \\
&= \left(\prod_{j=1}^{\mu(e)} \exp (h_{s_j(e)})\right)\left(\sum_{k = 1}^{d_{t(e)}-1} \exp (-\zeta_{d_{t(e)}}h_{t(e)})\right) \\ & = \sum_{k = 1}^{d_{t(e)}-1}\exp \left(\sum_{j=1}^{\mu(e)} h_{s_j(e)}-\zeta_{d_{t(e)}}^kh_{t(e)}\right)  .  
\end{align*}
Switching to Chern polynomial form, we obtain
$$C(t,\mathscr{B}(e)) = \prod_{k = 1}^{d_{t(e)}-1} \left(1 + \left(\sum_{j=1}^{\mu(e)} h_{s_j(e)} - \zeta_{d_{t(e)}}^kh_{t(e)}\right)t\right).$$
This product has degree $(d_{t(e)}-1)$ in $t$, with top coefficient
$$\prod_{k = 1}^{d_{t(e)}-1} \left(\sum_{j=1}^{\mu(e)} h_{s_j(e)} - \zeta_{t(e)}^kh_{t(e)}\right).$$
It follows from Definition~\ref{def:chern-classes}(iii) that
$C(t,\mathscr{B}(\mathbf{R})) = \prod_{e \in E} C(t,\mathscr{B}(e))$.
The product has degree $(\sum_{e \in E} d_{t(e)} - |E|)$ in $t$, with top coefficient (i.e., top Chern class of $\mathscr{B}(\mathbf{R})$) equal to
$$\prod_{e \in E} \prod_{k = 1}^{d_{t(e)}-1} \left(\sum_{j=1}^{\mu(e)} h_{s_j(e)} - \zeta_{t(e)}^kh_{t(e)}\right).$$ Finally, the formal identity $x^n - y^n = \prod_{k = 0}^n (x - \zeta_n^k y)$ gives
\begin{align*} \prod_{e \in E} \prod_{k = 1}^{d_{t(e)}-1} \left(\sum_{j=1}^{\mu(e)} h_{s_j(e)} - \zeta_{t(e)}^kh_{t(e)}\right) & = \prod_{e \in E} \frac{\left(\sum_{j=1}^{\mu(e)} h_{s_j(e)}\right)^{d_{t(e)}-1} - h_{t(e)}^{d_{t(e)}-1}}{\sum_{j=1}^{\mu(e)} h_{s_j(e)} - h_{t(e)}} \\ & = \prod_{e \in E} \sum_{k = 0}^{d_{t(e)}-1}\left(\sum_{j=1}^{\mu(e)} h_{s_j(e)}\right)^{d_{t(e)}-1-k} h_{t(e)}^k \in A^*(M). \qedhere
\end{align*} 
\end{proof}

\noindent To conclude, we now prove our main theorem.

\begin{proof}[Proof of Theorem \ref{thm:main}]
The zero locus of a generic global section of $\mathscr{B} := \mathscr{B}(\mathbf{R},F)$ is the singular vector variety $\mathcal{S}(\mathbf{R})$, with zero singular values along the edges in $F$, by 
Propositions~\ref{prop:zeros-are-singular-vectors} and \ref{prop:zeros-are-singular-vectors-F}. 
The singular vector bundle $\mathscr{B}$ from Definition~\ref{def:singular-vector-bundle-with-F} is almost generated, by Proposition \ref{prop:almost_generates}. Hence our Bertini-type theorem Theorem \ref{thm:bertiniTheorem} applies to it, to characterise the zeros of a generic section. It remains to derive the polynomial \eqref{eq:main-thm-polynomial}, prove the emptiness statement for $\mathcal{S}(\mathbf{R})$ as well as its dimension and degree, and prove the statement regarding finitely many singular vector tuples.

We first consider the case $F = \varnothing$. The top Chern class $c_r(\mathscr{B})$ is given by Proposition~\ref{prop:top-chern-class}. If $N = d - r = 0$, then $\mathcal{S}(\mathbf{R})$ has the claimed number of points by Theorem \ref{thm:bertiniTheorem}(b). Suppose $r < d$. Let $s: X \xhookrightarrow{} \mathbb{P}^D$ be the Segre embedding and let $[l] \in A^*(\mathbb{P}^D)$ be the class of a hyperplane. Continuing \eqref{topChernNumberCalculation}, we have
\begin{align}
{\label{topChernNumberCalculation-2}}
\nu\left(\mathscr{B}\big|_L\right)[p] &= [L]c_r(\mathscr{B}) = [L]s_*(c_r(\mathscr{B})) = [l]^Ns_*(c_r(\mathscr{B})) && \text{(definition of pushforward)} \nonumber \\
&= s_*(s^*([l]^N)c_r(\mathscr{B})) = s^*([l]^N)c_r(\mathscr{B}) && \text{(projection formula)} \nonumber \\
&= s^*([l])^Nc_r(\mathscr{B}) = (\textstyle{\sum_{i = 1}^n h_i})^N c_r(\mathscr{B}) && \text{(\cite[Example 8.4.3]{fulton_1998})}
\end{align}
where $A^*(X) \cong \mathbb{Z}[h_1,\ldots,h_n]/(h_1^{d_1},\ldots,h_n^{d_n})$, giving us the polynomial \eqref{eq:main-thm-polynomial}.

We prove the emptiness statement by showing that $\nu\left(\mathscr{B}\big|_L\right) = 0$ if and only if $c_r(\mathscr{B}) = 0$. By the proof of Theorem \ref{thm:bertiniTheorem}, $c_r(\mathscr{B}) = 0$ if and only if $\mathcal{S}(\mathbf{R}) = \varnothing$. If $c_r(\mathscr{B}) = 0$, then $\nu\left(\mathscr{B}\big|_L\right) = 0$ by \eqref{topChernNumberCalculation-2}. Conversely, if $c_r(\mathscr{B}) \neq 0$, then there exists a monomial $h_1^{a_1}\ldots h_n^{a_n}$ in $c_r(\mathscr{B})$ such that $a_i < d_i$ and $\sum_{i=1}^n a_i = r$. There exists a monomial $h_1^{d_1-1-a_1^\prime}\ldots h_n^{d_n-1-a_n^\prime}$ in $\left(\sum_{i=1}^n h_i\right)^{d-r}$ such that $\sum_{i=1}^na_i^\prime = r$. Thus, these monomials pair in the product $[L]c_r(\mathscr{B})$ to form the monomial $[p] = h_1^{d_1-1}\ldots h_n^{d_n-1}$. The coefficient of this monomial is $\nu\left(\mathscr{B}\big|_L\right)$, which is non-zero. Therefore if $\nu\left(\mathscr{B}\big|_L\right) \neq 0$, $\mathcal{S}(\mathbf{R})$ has the claimed dimension and degree by Theorem \ref{thm:bertiniTheorem}.

It remains to prove the last sentence of the theorem, which pertains to the case $N=0$. Fix $\varnothing \neq \alpha \subseteq [n]$ and define
$\mathscr{B}_\alpha$ as in the proof of
Proposition~\ref{prop:almost_generates}.
Then $\rk \mathscr{B}_\alpha = \rk \mathscr{B} - (|\alpha|-1) > \rk \mathscr{B} - |\alpha| = \dim(X) - |\alpha| = \dim(Y_\alpha)$ as the fibers of $\mathscr{B}_\alpha$ are vector subspaces of the fibers of $\mathscr{B}$ cut down by $|\alpha|-1$ linearly independent equations \eqref{tau-equations}. Thus, every singular vector has multiplicity 1 and is non-isotropic by Theorem \ref{thm:bertiniTheorem}(b).
Finally, if $F \neq \varnothing$ then $\rk \mathscr{B} > \dim(X)$ by \eqref{prop:rank-of-singular-vector-bundle-with-F}, so $\mathbf{R}$ has no singular values equal to 0, by Theorem \ref{thm:bertiniTheorem}(a).
\end{proof} 

As a final note, one may be interested in different homogenisations of singular vectors, such as \textit{H-eigenvectors}, as opposed to our non-homogeneous definition \eqref{eq:singular-vector-hyperquiver} of singular vectors which are often called \textit{Z-eigenvectors} \cite{qi2005eigenvalues}. For a positive integer $\ell$, we define an \textit{$\ell$-homogeneous singular vector tuple} of a hyperquiver representation $\mathbf{R} = (\mathbf{d},T)$ to be a singular vector tuple $([\x_1],\ldots,[\x_n]) \in \P(\CC^{d_n}) \times \cdots \times \P(\CC^{d_1})$ satisfying
\begin{align*}
    [T(\x_{s(e)})] = [\x_{t(e)}^{\odot \ell}]
\end{align*}
in $\P\left(\CC^{d_{t(e)}}\right)$ for all edges $e \in E$, where $\odot$ is the entry-wise Hadamard product of a vector. A singular vector tuple corresponds to $\ell = 1$. Then in fact, all of Theorem \ref{thm:main} is still true when $S(\mathbf{R})$ is replaced with the set of all $\ell$-homogeneous singular vector tuples $S^\ell(\mathbf{R})$, and when the polynomial \eqref{eq:main-thm-polynomial} is replaced with the polynomial
\begin{align}{\label{eq:main-thm-polynomial-homo}}
    \left(\sum_{i \in V} h_i\right)^N \cdot \prod_{e \in E} \left( \sum_{k = 1}^{d_{t(e)}} 
    (\ell h_{t(e)})^{k-1}
    h_{s(e)}^{d_{t(e)}-k} \right), \quad \text{where} \quad h_{s(e)} := \sum_{j=1}^{\mu(e)} h_{s_j(e)}.
\end{align}
Let us consider, for example, counting the number of $\ell$-eigenvectors of a generic tensor, as in Example \ref{ex:eigenvectors}. The polynomial \eqref{eq:main-thm-polynomial-homo} in this case gives us
$$\sum_{k=1}^d (\ell h)^{k-1}((m-1)h)^{d-k} = \left(\sum_{k=1}^d \ell^{k-1} (m-1)^{d-k}\right)h^{d-1} = \frac{(m-1)^d -\ell^d}{(m-1) - \ell}h^{d-1}$$
so it is $d(m-1)^{d-1}$ if $\ell = m-1$ and $\frac{(m-1)^d -\ell^d}{(m-1) - \ell}$ otherwise, thus generalising \cite[Theorem 2.2]{eigenconfiguration}.

We briefly outline the modifications needed to prove this more general result. Looking to the beginning of Section \ref{sec:bundles}, we seek to replace the tautological line bundle $\mathscr{T}(d)$, whose fiber at $[\x] \in \P(\CC^d)$ is $\text{span}(\x)$, with the line bundle $\mathscr{S}(d)$, whose fiber at $[\x]$ is $\text{span}(\x^{\odot \ell})$. We also want to consider the vector bundle  $\mathscr{P}(d_i)$ fitting into the short exact sequence
\begin{equation*}
0 \rightarrow \mathscr{S}(d_i) \rightarrow \mathscr{F}(d_i) \rightarrow \mathscr{P}(d_i) \rightarrow 0.
\end{equation*}
The line bundle $\mathscr{S}(d)$ is isomorphic to the line bundle $\mathscr{T}(d)^{\otimes \ell}$, hence they have the same Chern polynomial. One computes the Chern polynomial to be $C(\mathscr{T}(d)^{\otimes \ell},t) = 1 - \ell h t$ and deduces from there the top Chern class intersected with the necessary number of hyperplanes to be the polynomial \eqref{eq:main-thm-polynomial-homo} by following the computation analogously to earlier in this section.

\section{Inverse tensor eigenvalue problems}\label{sec:inverse-problems}

The problem we have considered up until now in this paper is the following: Given a hyperquiver $H = (V,E)$ and a representation $\mathbf{R} = (\mathbf{d},T)$ assigning tensors $T_e \in \CC^e$ to each edge $e \in E$, what is the space of singular vector tuples $([\mathbf{x}_1],\ldots,[\mathbf{x}_n]) \in \prod_{i = 1}^n \mathbb{P}(\CC^{d_i})$? In this section, we consider the converse problem: Given a hyperquiver $H = (V,E)$ and an assignment of vectors $[\mathbf{x}_i] \in \mathbb{P}(\CC^{d_i})$ to each vertex $i \in [n]$, what is the space of tensors $(T_e)_{e \in E} \in \prod_{e \in E} \CC^e$ realising $([\mathbf{x}_1],\ldots,[\mathbf{x}_n])$ as a singular vector tuple?

We make this problem more precise and more general. Let $H = (V,E)$ be a hyperquiver, $\mathbf{d} = (d_1,\ldots,d_n)$ be a dimension vector, $\lambda_e \in \CC$ be scalars for $e \in E$, and $\mathbf{x}_i \in \CC^{d_i}$ be vectors for $i \in [n]$. Let $E = \coprod_{r = 1}^M E_r$ be a partition of the hyperedges such that for all $r \in [M]$ and for any two edges $e,e^\prime \in E_r$, the tuple of dimensions $(d_{t(e)},d_{s_1(e)},\ldots,d_{s_\mu(e)})$ and $(d_{t(e^\prime)},d_{s_1(e^\prime)},\ldots,d_{s_\mu(e^\prime)})$ have the same number of entries and agree up to permutation. Following the notation outlined below equation \eqref{eq:T_r-notation}, fix an edge $e_r \in E_r$ for each $r \in [M]$, and fix the permutation $\sigma_{e,e_r}$ of the set $\{0,1,\ldots,\mu\}$ for each $e \in E_r$, $e \neq e^\prime$ for which the tuples of dimensions for $e_r$ and $e$ agree. Then the problem is to find tensors $T_e \in \CC^{e}$ such that
\begin{equation}{\label{eq:inverse-tensor-eigenvalue-problem}}
    T_e(\mathbf{x}_{s(e)}) = \lambda_e \cdot \mathbf{x}_{t(e)}
\end{equation}
for all $e \in E$, where any tensor $T_e$ assigned to an edge $e \in E_r$, $e \neq e_r$ agrees with the tensor $T_r := T_{e_r}$ assigned to the edge $e_r \in E_r$ up to the chosen permutation $\sigma_{e,e_r}$. Note that the conditions this partition of $E$ has to satisfy are weaker than property (i) of a partition of a hyperquiver in Definition \ref{def:generic} of a generic representation, since the latter requires that the tuples of vertices $v(e)$ and $v(e^\prime)$ for $e,e^\prime \in E_r$ agree up to permutation. We call this problem the \textit{inverse singular value problem} for hyperquivers.

This problem is a significant generalisation of the inverse matrix and tensor eigenvalue problems, which are concerned with finding a matrix or tensor having a prescribed set of eigenvalues or eigenvectors. The matrix problem, and its many variations that involve imposing structural constraints on the matrix, have a range of applications including principal component analysis, control theory, numerical analysis, and inverse problems, see \cite{chu-inverse-1998,Chu_Golub_2002} for a review of the topic. In the tensor problem, essentially the only case that has been studied is where no structural constraints are imposed, and even then the problem is much more difficult and less well-understood. The solvability of the inverse tensor eigenvalue problem \cite{YH17} and some geometric properties of the space of characteristic polynomials of tensors \cite{GGTV23} are known, however even basic quantities such as the dimension of this space are not known, which is crucial for numerical methods in algebraic geometry \cite{breiding-compressed,condition-book}. The inverse tensor eigenvalue problem has seen applications in inverse problems for PDEs \cite{zayed-inverse} and higher-order Markov chains \cite{Li_Markov_Chain_Inverse}.

In this section, we solve the inverse singular value problem for hyperquivers by providing an algorithm that finds tensors $T_e$ to place on all the edges $e \in E$ of the hyperquiver satisfying \eqref{eq:inverse-tensor-eigenvalue-problem} and the partition conditions. We generalize the approach from \cite{yokohama2,LDZ21} solving the inverse tensor eigenvalue problem in a number of different directions. One is that the eigenvectors are not required to be homogeneous, i.e. \textit{H-eigenvectors} \cite{qi2005eigenvalues}, and any homogenisation, e.g. \textit{Z-eigenvectors}, can be used. Another is that the vectors do not have to be eigenvectors, but more generally singular vectors arranged in any hyperquiver structure. Our algorithm works for both real and complex singular vector tuples.

We define the inner product on two tensors $A,B \in \mathbb{C}^{d_1\times \ldots\times d_m}$ to be
\begin{equation*}
    \langle A, B \rangle = \sum_{\substack{i_1 = 1, \ldots, d_1 \\ \ldots \\ i_m = 1, \ldots, d_m}} A_{i_1, \ldots, i_m}\overline{B_{i_1, \ldots, i_m}}
\end{equation*}
where the bar denotes the conjugate of a complex number. We use the same bar to denote the conjugate of a complex vector. This gives rise to the Frobenius norm of a tensor
\begin{equation*}
    \|A\| = \sqrt{\langle A, A \rangle}
\end{equation*}
We also use this notation for complex vectors, i.e. for $m = 1$. Note that when contracting tensors with vectors, however, we are still using the standard real inner product with complex vectors as before, and not this standard Hermitian inner product, see Remark \ref{remark:choice-of-bases}.

\begin{algorithm}
\caption{Inverse singular value problem for $H = (V,E)$, $E = \coprod_{r = 1}^M E_r$, $\lambda_e \in \CC$, $\x_i \in \CC^{d_i}$}\label{alg:inverse-singular-value}
\begin{flushleft}
\textbf{Input:} $T_r^{(1)} \in \mathbb{C}^{e_r}$ for $r \in [M]$ initialisations \\
\textbf{Output:} $T_r \in \mathbb{C}^{e_r}$ for $r \in [M]$ solving the inverse singular value problem, or no solution
\end{flushleft}
\begin{algorithmic}[1]
\For{$r = 1,\ldots,M$}
\State $e_1 \gets e_r$
\State $E_r \gets \{e_1,\ldots,e_m\}$
\For{$\ell = 1,\ldots,m$}
\State $\mathbf{r}_\ell^{(1)} \gets \lambda_{e_\ell} \mathbf{x}_{t(e_\ell)} - T_r^{(1)}(\mathbf{x}_{s(e_\ell)}) \in \mathbb{C}^{d_{t(e_\ell)}}$
\EndFor
\State $\mathbf{r}^{(1)} \gets \begin{pmatrix}
{\mathbf{r}^{(1)}_{1}}^\top &|& \cdots &|& {\mathbf{r}^{(1)}_{m}}^\top
\end{pmatrix}^\top \in \mathbb{C}^{D_r}, \;\; D_r := \sum_{\ell=1}^m d_{t(e_\ell)}$
\State $\displaystyle P^{(1)} \gets \sum_{\ell=1}^m \left(\bigotimes_{j=0}^{\sigma_{\ell,1}(0)-1}\overline{\x_{s_{1,j}(e_\ell)}}\right)\otimes \mathbf{r}_\ell^{(1)} \otimes \left(\bigotimes_{j=\sigma_{\ell,1}(0)+1}^{\mu}\overline{\x_{s_{1,j}(e_\ell)}}\right) \in \CC^{e_r}$
\State $k \gets 1$
\While{$\mathbf{r}^{(k)} \neq 0$}
\If{$P^{(k)} = 0$}
    \State halt   \Comment{No solution exists}
\Else
    \State $\displaystyle T_r^{(k+1)} = T_r^{(k)} + \frac{\|\mathbf{r}^{(k)}\|^2}{\|P^{(k)}\|^2}P^{(k)}$ \label{alg:T-update}
    \For{$\ell = 1,\ldots,m$}
    \State $\mathbf{r}_\ell^{(k+1)} \gets \lambda_{e_\ell} \mathbf{x}_{t(e_\ell)} - T_r^{(k+1)}(\mathbf{x}_{s(e_\ell)})$ \label{alg:r-ell-def}
    \EndFor
    \State $\mathbf{r}^{(k+1)} \gets \begin{pmatrix}
    {\mathbf{r}^{(k+1)}_{1}}^\top &|& \cdots &|& {\mathbf{r}^{(k+1)}_{m}}^\top
    \end{pmatrix}^\top$ \label{alg:r-def}
    \State $\displaystyle P^{(k+1)} \gets \sum_{\ell=1}^m \left(\bigotimes_{j=0}^{\sigma_{\ell,1}(0)-1}\overline{\x_{s_{1,j}(e_\ell)}}\right)\otimes \mathbf{r}_\ell^{(k+1)} \otimes \left(\bigotimes_{j=\sigma_{\ell,1}(0)+1}^{\mu}\overline{\x_{s_{1,j}(e_\ell)}}\right) + \frac{\|\mathbf{r}^{(k+1)}\|^2}{\|\mathbf{r}^{(k)}\|^2}P^{(k)}$ \label{alg:P-update}
    \State $k \gets k + 1$
\EndIf
\EndWhile
\EndFor
\end{algorithmic}
\end{algorithm}

The algorithm is presented in Algorithm \ref{alg:inverse-singular-value}. For each $r \in [M]$, we enumerate the edges $E_r = \{e_1,\ldots,e_m\}$ and assume without loss of generality that $e_1 = e_r$. Since every tensor $T_e$ assigned to $e \in E_r, e \neq e_r$ agrees with the tensor $T_r$ up to permutation, it suffices to find the tensors $T_r$ for $r \in [M]$. We write $T_r(\mathbf{x}_{s(e)}) := T_e(\mathbf{x}_{s(e)})$ for all $e \in E_r$, and we use the same notation for any other tensor $P \in \CC^{e_r}$ of the same dimensions as $T_r$. For brevity, we write $\sigma_{\ell,1} := \sigma_{e_\ell,e_1}$ and $\sigma_{1,\ell} := \sigma^{-1}_{\ell,1}$ for $\ell \in [m]$, where $\sigma_{1,1}$ denotes the identity permutation, and we write $s_{1,j}(e_\ell) := s_{\sigma_{1,\ell}(j)}(e_\ell)$. For an edge $e_\ell \in E_r$, $\ell \neq 1$, recall that $\sigma_{\ell,1}(0)$ denotes the position of the vertex $t(e_\ell)$ in the tuple of vertices $v(e_1) = (t(e_1),s_1(e_1),\ldots,s_\mu(e_1))$.

Let us illustrate the algorithm with the inverse tensor eigenvalue problem: given $\x \in \mathbb{C}^d$ and $\lambda \in \CC$, find $T \in (\mathbb{C}^d)^{\otimes m}$ such that $T(\, \cdot \,,\x,\ldots,\x) = \lambda \x$. If $T^{(1)}$ is an initialisation, then line \ref{alg:T-update} of Algorithm \ref{alg:inverse-singular-value} gives us
\begin{align*}
    T^{(2)} &= T^{(1)} + \frac{\|\lambda \x - T^{(1)}(\, \cdot \,,\x,\ldots,\x)\|^2}{\|(\lambda \x - T^{(1)}(\, \cdot \,,\x,\ldots,\x)) \otimes \x \otimes \cdots \otimes \x\|^2}\left(\lambda \x - T^{(1)}(\, \cdot \,,\x,\ldots,\x)\right) \otimes \overline{\x} \otimes \cdots \otimes \overline{\x} \\
    &= T^{(1)} + \frac{\|\lambda \x - T^{(1)}(\, \cdot \,,\x,\ldots,\x)\|^2}{\|\lambda \x - T^{(1)}(\, \cdot \,,\x,\ldots,\x)\|^2 \|\x\|^{2(m-1)}}\left(\lambda \x - T^{(1)}(\, \cdot \,,\x,\ldots,\x)\right) \otimes \overline{\x} \otimes \cdots \otimes \overline{\x} \\
    &= T^{(1)} + \frac{1}{\|\x\|^{2(m-1)}}\left(\lambda \x - T^{(1)}(\, \cdot \,,\x,\ldots,\x)\right) \otimes \overline{\x} \otimes \cdots \otimes \overline{\x}
\end{align*}
In fact, the algorithm gives all solutions to the problem.

\begin{theorem}
All solutions to the inverse tensor eigenvalue problem are of the form
\begin{align}{\label{eq:inverse-tensor-eigenvector-solution}}
    T_0 + \frac{1}{\|\x\|^{2(m-1)}}\left(\lambda \x - T_0(\, \cdot \,,\x,\ldots,\x)\right) \otimes \overline{\x} \otimes \cdots \otimes \overline{\x},
\end{align}
where $T_0 \in (\mathbb{C}^d)^{\otimes m}$ is arbitrary.
\end{theorem}

\begin{proof}
Contracting \eqref{eq:inverse-tensor-eigenvector-solution} with $\x$ along all components except the first, we have
\begin{align*}
    T_0(\, \cdot \,,\x,\ldots,\x) + \frac{1}{\|\x\|^{2(m-1)}}\left(\lambda \x - T_0(\, \cdot \,,\x,\ldots,\x)\right)\cdot\langle \x,\x\rangle \cdots \langle \x,\x\rangle \\
    = T_0(\, \cdot \,,\x,\ldots,\x) + \frac{1}{\|\x\|^{2(m-1)}}\left(\lambda \x - T_0(\, \cdot \,,\x,\ldots,\x)\right)\cdot\|\x\|^{2(m-1)} = \lambda \x
\end{align*}
thus \eqref{eq:inverse-tensor-eigenvector-solution} is a solution. Conversely, if $T$ is a solution, then $\lambda \x - T(\, \cdot \,,\x,\ldots,\x) = 0$ so \eqref{eq:inverse-tensor-eigenvector-solution} reduces to $T_0 = T$.
\end{proof}

We now provide a convergence analysis showing that this algorithm finds tensors $T_r \in \CC^{e_r}$ solving the singular value problem for hyperquivers within a finite number of iterations, and a proof showing that it detects if no solution to the problem exists. The algorithm and its analysis still all hold if a different homogenisation is used, e.g. replacing every instance of $\lambda_{e_\ell}\x_{t(e_\ell)}$ with $\lambda_{e_\ell}\x_{t(e_\ell)}^{\odot \ell}$ where $\odot$ is the entry-wise Hadamard product. For $k \geq 1$, define
\begin{equation*}
    \alpha(k) := \frac{\|\mathbf{r}^{(k)}\|^2}{\|P^{(k)}\|^2}, \quad\quad \beta(k+1) := \frac{\|\mathbf{r}^{(k+1)}\|^2}{\|\mathbf{r}^{(k)}\|^2}, \quad\quad \beta(1) := 0
\end{equation*}

\begin{lemma}{\label{lem:orthogonality-r-P}}
For a fixed $r \in [M]$, let $\{\mathbf{r}^{(k)}\}_{k \geq 1}$ and $\{P^{(k)}\}_{k \geq 1}$ be the sequences generated by Algorithm \ref{alg:inverse-singular-value}. Then for all $i,j \geq 1$ with $i \neq j$, 
\begin{equation*}
\left\langle \mathbf{r}^{(i)}, \mathbf{r}^{(j)} \right\rangle = 0 \quad \text{and} \quad \left\langle P^{(i)}, P^{(j)} \right\rangle = 0.
\end{equation*}
\end{lemma}
\begin{proof}
Let $E = \{e_1,\ldots,e_m\}$. We first see that for $k \geq 1$ and all $\ell \in [m]$,
\begin{flalign}
    \mathbf{r}_\ell^{(k+1)} &= \lambda_{e_\ell} \mathbf{x}_{t(e_\ell)} - T_r^{(k+1)}(\mathbf{x}_{s(e_\ell)}) &\text{(line\;\ref{alg:r-ell-def})} \nonumber \\
    &= \lambda_{e_\ell} \mathbf{x}_{t(e_\ell)} - T_r^{(k)}(\mathbf{x}_{s(e_\ell)}) - \alpha(k) \cdot P^{(k)}(\mathbf{x}_{s(e_\ell)}) &\text{(line\;\ref{alg:T-update})} \nonumber \\
    &= \mathbf{r}_\ell^{(k)} - \alpha(k) \cdot P^{(k)}(\mathbf{x}_{s(e_\ell)}) &\label{eq:r_ell_interation} 
\end{flalign}
and thus
\begin{flalign}
    \mathbf{r}^{(k+1)} &=  \begin{pmatrix}
    {\mathbf{r}^{(k+1)}_{1}}^\top &|& \cdots &|& {\mathbf{r}^{(k+1)}_{m}}^\top
    \end{pmatrix}^\top &&\text{(line\;\ref{alg:r-def})} \nonumber \\
    &= \begin{pmatrix}
    {\mathbf{r}^{(k)}_{1}}^\top &|& \cdots &|& {\mathbf{r}^{(k)}_{m}}^\top
    \end{pmatrix}^\top - \alpha(k) \cdot \begin{pmatrix}
    {P^{(k)}(\mathbf{x}_{s(e_1)})}^\top &|& \cdots &|& {P^{(k)}(\mathbf{x}_{s(e_m)})}^\top
    \end{pmatrix}^\top &&\text{(eq.\;\ref{eq:r_ell_interation})} \nonumber \\
    &= \mathbf{r}^{(k)} - \alpha(k) \cdot \begin{pmatrix}
    {P^{(k)}(\mathbf{x}_{s(e_1)})}^\top &|& \cdots &|& {P^{(k)}(\mathbf{x}_{s(e_m)})}^\top
    \end{pmatrix}^\top &&\label{eq:r_k+1-iteration}
\end{flalign}
Additionally, for any $i,j \geq 1$, we see that
\begin{flalign}
    \left\langle \mathbf{r}^{(i+1)}, \mathbf{r}^{(j)} \right\rangle &= \left\langle \mathbf{r}^{(i)}, \mathbf{r}^{(j)} \right\rangle - \alpha(i) \cdot \sum_{\ell=1}^m \left\langle P^{(i)}(\mathbf{x}_{s(e_\ell)}), \mathbf{r}^{(j)}_\ell \right\rangle &&\text{(eq.\;\ref{eq:r_k+1-iteration})} \nonumber \\ 
    &= \left\langle \mathbf{r}^{(i)}, \mathbf{r}^{(j)} \right\rangle &&\nonumber \\ &\quad\quad - \alpha(i) \cdot \sum_{\ell=1}^m \left\langle P^{(i)}, \left(\bigotimes_{k=0}^{\sigma_{\ell,1}(0)-1}\overline{\x_{s_{1,k}(e_\ell)}}\right)\otimes \mathbf{r}_\ell^{(j)} \otimes \left(\bigotimes_{k=\sigma_{\ell,1}(0)+1}^{\mu}\overline{\x_{s_{1,k}(e_\ell)}}\right) \right\rangle &&\nonumber \\
    &= \left\langle \mathbf{r}^{(i)}, \mathbf{r}^{(j)} \right\rangle &&\nonumber \\ &\quad\quad - \alpha(i) \cdot \left\langle P^{(i)}, \sum_{\ell=1}^m \left(\bigotimes_{k=0}^{\sigma_{\ell,1}(0)-1}\overline{\x_{s_{1,k}(e_\ell)}}\right)\otimes \mathbf{r}_\ell^{(j)} \otimes \left(\bigotimes_{k=\sigma_{\ell,1}(0)+1}^{\mu}\overline{\x_{s_{1,k}(e_\ell)}}\right) \right\rangle &&\nonumber \\
    &= \left\langle \mathbf{r}^{(i)}, \mathbf{r}^{(j)} \right\rangle - \alpha(i) \cdot \left\langle P^{(i)}, P^{(j)} - \beta(j)\cdot P^{(j-1)} \right\rangle &&\text{(line\;\ref{alg:P-update})} \nonumber \\
    &= \left\langle \mathbf{r}^{(i)}, \mathbf{r}^{(j)} \right\rangle - \alpha(i) \cdot \left\langle P^{(i)}, P^{(j)} \right\rangle + \alpha(i)\cdot\beta(j)\cdot \left\langle P^{(i)}, P^{(j-1)} \right\rangle &&\label{eq:r_i+1_r_j}
\end{flalign}
where $P^{(0)} := 0$, and
\begin{align}
    \left\langle P^{(i+1)}, P^{(j)} \right\rangle 
    &= \sum_{\ell=1}^m \left\langle  \left(\bigotimes_{k=0}^{\sigma_{\ell,1}(0)-1}\overline{\x_{s_{1,k}(e_\ell)}}\right)\otimes \mathbf{r}_\ell^{(i+1)} \otimes \left(\bigotimes_{k=\sigma_{\ell,1}(0)+1}^{\mu}\overline{\x_{s_{1,k}(e_\ell)}}\right), P^{(j)} \right\rangle && \nonumber \\ &\quad\quad + \beta(i+1) \cdot \left\langle P^{(i)}, P^{(j)} \right\rangle && \text{(line\;\ref{alg:P-update})} \nonumber \\
    &= \sum_{\ell=1}^m \left\langle  \mathbf{r}_\ell^{(i+1)}, P^{(j)}(\x_{s(e_\ell)}) \right\rangle + \beta(i+1) \cdot \left\langle P^{(i)}, P^{(j)} \right\rangle && \nonumber \\
    &= \sum_{\ell=1}^m \left\langle  \mathbf{r}_\ell^{(i+1)}, 
    \frac{1}{\alpha(j)}(\mathbf{r}_\ell^{(j)} - \mathbf{r}_\ell^{(j+1)}) \right\rangle + \beta(i+1) \cdot \left\langle P^{(i)}, P^{(j)} \right\rangle && \text{(eq.\;\ref{eq:r_ell_interation})} \nonumber \\
    &= \frac{1}{\alpha(j)}\sum_{\ell=1}^m \left(\left\langle \mathbf{r}_\ell^{(i+1)},\mathbf{r}_\ell^{(j)} \right\rangle - \left\langle \mathbf{r}_\ell^{(i+1)},\mathbf{r}_\ell^{(j+1)} \right\rangle \right) + \beta(i+1) \cdot \left\langle P^{(i)}, P^{(j)} \right\rangle && \label{eq:P_i+1_P_j}
\end{align}

Since $\left\langle \mathbf{r}^{(i)}, \mathbf{r}^{(j)} \right\rangle = \overline{\left\langle \mathbf{r}^{(j)}, \mathbf{r}^{(i)} \right\rangle}$ and likewise for $P^{(i)}$ and $P^{(j)}$, and since the statement of the lemma is symmetric in $i$ and $j$, it suffices to prove the lemma for $i \geq j$. We proceed by induction, proving it for all $i = k$ and all $j = 1,\ldots,k-1$. The base case $j = 1$ and $i = 2$ is similar to the inductive step, so we only show the latter. Assuming the inductive hypothesis, then for $i = k+1$ and any $j = 1,\ldots,k$, we have
\begin{flalign}
    \left\langle \mathbf{r}^{(k+1)}, \mathbf{r}^{(j)} \right\rangle &= \left\langle \mathbf{r}^{(k)}, \mathbf{r}^{(j)} \right\rangle - \alpha(k) \cdot \left\langle P^{(k)}, P^{(j)} \right\rangle + \alpha(k)\cdot\beta(j)\cdot \left\langle P^{(k)}, P^{(j-1)} \right\rangle && \text{(eq.\;\ref{eq:r_i+1_r_j})} \nonumber \\
    &= \delta_{j,k} \cdot \|\mathbf{r}^{(k)}\|^2 - \delta_{j,k} \cdot \frac{\|\mathbf{r}^{(k)}\|^2}{\|P^{(k)}\|^2} \cdot \|P^{(k)}\|^2 + 0 = 0 && \label{eq:r12=0}
\end{flalign}
and
\begin{align}
    \left\langle P^{(k+1)}, P^{(j)} \right\rangle &= \frac{1}{\alpha(k)}\sum_{\ell=1}^m \left(\left\langle \mathbf{r}_\ell^{(k+1)},\mathbf{r}_\ell^{(j)} \right\rangle - \left\langle \mathbf{r}_\ell^{(k+1)},\mathbf{r}_\ell^{(j+1)} \right\rangle \right) + \beta(k+1) \cdot \left\langle P^{(k)}, P^{(j)} \right\rangle && \text{(eq.\;\ref{eq:P_i+1_P_j})} \nonumber \\ 
    &= \frac{1}{\alpha(k)}\sum_{\ell=1}^m \left(0 - \delta_{j,k} \cdot \|\mathbf{r}_\ell^{(k+1)}\|^2 \right) + \delta_{j,k} \cdot \beta(k+1) \cdot \|P^{(k)}\|^2 && \text{(eq.\;\ref{eq:r12=0})} \nonumber \\
    &= -\delta_{j,k}\cdot\frac{1}{\alpha(k)}\sum_{\ell=1}^m \|\mathbf{r}_\ell^{(k+1)}\|^2 + \delta_{j,k} \cdot \beta(k+1) \cdot \|P^{(k)}\|^2 && \nonumber \\
    &= -\delta_{j,k} \cdot \frac{\|P^{(k)}\|^2}{\|\mathbf{r}^{(k)}\|^2}\|\mathbf{r}^{(k+1)}\|^2 + \delta_{j,k} \cdot \frac{\|\mathbf{r}^{(k+1)}\|^2}{\|\mathbf{r}^{(k)}\|^2} \cdot \|P^{(k)}\|^2 = 0 \nonumber
\end{align}
where $\delta_{j,k} = 1$ if $j = k$ and 0 otherwise.
\end{proof}

\begin{lemma}{\label{lem:orthogonal-T}}
For a fixed $r \in [M]$, let $\{\mathbf{r}^{(k)}\}_{k \geq 1}$, $\{P^{(k)}\}_{k \geq 1}$, and $\{T_r^{(k)}\}_{k \geq 1}$ be the sequences generated by Algorithm \ref{alg:inverse-singular-value}, for any initialisation $T_r^{(1)}$. Suppose that $\widehat{T_r}$ is a tensor solving the inverse singular value problem \eqref{eq:inverse-tensor-eigenvalue-problem} for all $e \in E_r$. Then for all $k \geq 1$, 
\begin{equation}{\label{eq:orthogonal-T}}
\left\langle \widehat{T_r} - T_r^{(k)}, P^{(k)} \right\rangle = \|\mathbf{r}^{(k)}\|^2
\end{equation}
\end{lemma}
\begin{proof}
Let $E = \{e_1,\ldots,e_m\}$. We proceed by induction. The base case $k = 1$ is similar to the inductive step, so we only show the latter. Assuming the inductive hypothesis, then for step $k + 1$, first we see that
\begin{flalign}
    \left\langle \widehat{T_r} - T_r^{(k+1)}, P^{(k)} \right\rangle &= \left\langle \widehat{T_r} - T_r^{(k)}, P^{(k)} \right\rangle - \alpha(k) \cdot \left\langle P^{(k)}, P^{(k)} \right\rangle &\text{(line\;\ref{alg:T-update})} \nonumber\\ 
    &= \|\mathbf{r}^{(k)}\|^2 - \frac{\|\mathbf{r}^{(k)}\|^2}{\|P^{(k)}\|^2} \cdot \|P^{(k)}\|^2 = 0 &\label{eq:T-k+1-P-k=0} 
\end{flalign}
Thus,
\begin{alignat}{2}
    &\left\langle \widehat{T_r} - T_r^{(k+1)}, P^{(k+1)} \right\rangle & \nonumber \\ &= \left\langle \widehat{T_r} - T_r^{(k+1)}, \sum_{\ell=1}^m \left(\bigotimes_{j=0}^{\sigma_{\ell,1}(0)-1}\overline{\x_{s_{1,j}(e_\ell)}}\right)\otimes \mathbf{r}_\ell^{(k+1)} \otimes \left(\bigotimes_{j=\sigma_{\ell,1}(0)+1}^{\mu}\overline{\x_{s_{1,j}(e_\ell)}}\right) \right\rangle & \nonumber \\ 
    &\quad\quad\quad\quad + \beta(k+1) \cdot \left\langle \widehat{T_r} - T_r^{(k+1)}, P^{(k)} \right\rangle & \text{(line\;\ref{alg:P-update})} \nonumber \\ &=\sum_{\ell=1}^m\left\langle \widehat{T_r} - T_r^{(k+1)}, \left(\bigotimes_{j=0}^{\sigma_{\ell,1}(0)-1}\overline{\x_{s_{1,j}(e_\ell)}}\right)\otimes \mathbf{r}_\ell^{(k+1)} \otimes \left(\bigotimes_{j=\sigma_{\ell,1}(0)+1}^{\mu}\overline{\x_{s_{1,j}(e_\ell)}}\right) \right\rangle + 0 & \text{(eq.\;\ref{eq:T-k+1-P-k=0})} \nonumber \\ 
    &=\sum_{\ell=1}^m\left\langle \widehat{T_r}(\x_{s(e_\ell)}) - T_r^{(k+1)}(\x_{s(e_\ell)}), \mathbf{r}_\ell^{(k+1)} \right\rangle &  \nonumber \\
    &=\sum_{\ell=1}^m\left\langle \lambda_{e_\ell}\cdot\x_{t(e_\ell)} - T_r^{(k+1)}(\x_{s(e_\ell)}), \mathbf{r}_\ell^{(k+1)} \right\rangle & \text{($\widehat{T}_r$ is a solution)} \nonumber \\
    &=\sum_{\ell=1}^m\left\langle \mathbf{r}_\ell^{(k+1)}, \mathbf{r}_\ell^{(k+1)} \right\rangle = \sum_{\ell=1}^m\|\mathbf{r}_\ell^{(k+1)}\|^2 = \|\mathbf{r}^{(k+1)}\|^2 & \nonumber \qedhere
\end{alignat}
\end{proof}

\begin{theorem}{\label{thm:tensor-algorithm}}
If the inverse singular value problem \eqref{eq:inverse-tensor-eigenvalue-problem} is solvable, then Algorithm \ref{alg:inverse-singular-value} will find a solution $\{\widehat{T}_r\}_{r = 1}^M$ in a finite number of iterations for any initialisation $\{T_r^{(1)}\}_{r = 1}^M$. Otherwise, the problem has no solution if and only if there exists a $K \geq 1$ such that $\mathbf{r}^{(K)} \neq 0$ and $P^{(K)} = 0$ in some iteration $r \in [M]$ for $T_r$.
\end{theorem}
\begin{proof}
Let $r \in [M]$. If $\mathbf{r}^{(k)} = 0$ for any $k \geq 1$, then the inverse singular value problem has a solution $\widehat{T_r} = T_r^{(k)}$ by definition (line \ref{alg:r-ell-def}). Additionally, if $P^{(k)} \neq 0$ for all $k \geq 1$, then the iterations of Algorithm \ref{alg:inverse-singular-value} (in particular, lines \ref{alg:T-update} and \ref{alg:P-update}) are well-defined for all $k \geq 1$. By Lemma \ref{lem:orthogonality-r-P}, $\{\mathbf{r}^{(k)}\}_{k \geq 1}$ forms an orthogonal sequence in $\mathbb{C}^{D_r}$, and hence there must exist a finite iteration number $K \geq 1$ such that $\mathbf{r}^{(K)} = 0$, which means that $T_r^{(K)}$ is a solution. This proves the first part of the theorem and one direction of the second part. Finally, if there exists a $K \geq 1$ such that $\mathbf{r}^{(K)} \neq 0$ and $P^{(K)} = 0$, then this would contradict \eqref{eq:orthogonal-T} of Lemma \ref{lem:orthogonal-T}, unless there was no solution $\widehat{T_r}$. This proves the second part of the theorem.
\end{proof}

\bigskip

{\bf Acknowledgements.}
We thank Giorgio Ottaviani and Shmuel Friedland for helpful discussions. TM is supported by a Mathematical Institute Scholarship and an NSERC PGS D Scholarship. VN is supported by the EPSRC grant EP/R018472/1 and US AFOSR grant FA9550-22-1-0462. AS is supported by the NSF (DMR-2011754).

\bibliographystyle{plain}
\bibliography{main.bib}

\input{A-appendix}

\end{document}

%% file: preamble.tex
\usepackage{graphicx}
\usepackage[all]{xy}
\usepackage{color} 
\usepackage[dvipsnames]{xcolor}
\usepackage{latexsym}
\usepackage{amssymb,amsmath,amstext,mathtools,mathrsfs}
\usepackage{epsfig}
\usepackage{graphics}
\usepackage{graphicx}
\usepackage{float}
\usepackage{euscript}
\usepackage{ifthen}
\usepackage{wrapfig}
\usepackage{amsthm}
\usepackage{multicol}
\usepackage{multirow}
\usepackage{paralist}
\usepackage{verbatim}
\usepackage[shortlabels]{enumitem}
\usepackage[colorlinks=true, urlcolor=blue, linkcolor=blue, citecolor=blue]{hyperref}
\usepackage[margin=1in]{geometry}
\usepackage{enumitem}
\usepackage{algorithmicx}
\usepackage{algpseudocode}
\usepackage{algorithm}
\usepackage{placeins}
\usepackage{caption}
\usepackage{subcaption}
\usepackage{longtable}
\usepackage{lscape}
\usepackage{multicol}
\usepackage{tikz}
\usepackage{tikz-cd}
\usepackage{listings}
\usepackage{xcolor}
\usepackage{diagbox}
\usepackage{array}
\usepackage{xcolor}
\usepackage{soul}
\usepackage[normalem]{ulem}
\usepackage{pifont}
\usepackage{xr}

\usepackage[makeroom]{cancel}
\hypersetup
{
	colorlinks,	%
	citecolor=blue,%
	filecolor=black,%
	linkcolor=red,%
	urlcolor=blue
}

\newsavebox{\overlongequation}

\definecolor{codegreen}{rgb}{0,0.6,0}
\definecolor{codegray}{rgb}{0.5,0.5,0.5}
\definecolor{codepurple}{rgb}{0.58,0,0.82}
\definecolor{backcolour}{rgb}{0.95,0.95,0.92}

\lstdefinestyle{mystyle}{
    backgroundcolor=\color{backcolour},   
    commentstyle=\color{codegreen},
    keywordstyle=\color{magenta},
    numberstyle=\tiny\color{codegray},
    stringstyle=\color{codepurple},
    basicstyle=\ttfamily\footnotesize,
    breakatwhitespace=false,         
    breaklines=true,                 
    captionpos=b,                    
    keepspaces=true,                 
    numbers=left,                    
    numbersep=5pt,                  
    showspaces=false,                
    showstringspaces=false,
    showtabs=false,                  
    tabsize=2
}

\numberwithin{equation}{section}

\newtheorem{theorem}{Theorem}[section]
\newtheorem*{theorem*}{Theorem}
\newtheorem{lemma}[theorem]{Lemma}
\newtheorem{proposition}[theorem]{Proposition}
\newtheorem{corollary}[theorem]{Corollary}

\theoremstyle{definition}
\newtheorem{definition}[theorem]{Definition}

\newtheorem{example}[theorem]{Example}

\newtheorem{remark}[theorem]{Remark}

\usepackage{scalerel,stackengine}
\stackMath

\def\CC{{\mathbb C}}

\newcommand\reallywidehat[1]{%
\savestack{\tmpbox}{\stretchto{%
  \scaleto{%
    \scalerel*[\widthof{\ensuremath{#1}}]{\kern-.6pt\bigwedge\kern-.6pt}%
    {\rule[-\textheight/2]{1ex}{\textheight}}
  }{\textheight}%
}{0.5ex}}%
\stackon[1pt]{#1}{\tmpbox}%
}

\DeclareMathOperator{\rk}{rank}
\DeclareMathOperator{\Span}{span}

\DeclareMathOperator{\ch}{ch}
\DeclareMathOperator{\codim}{codim}

\DeclareMathOperator{\Hom}{Hom}

\newcommand{\QED}{\ifhmode\unskip\nobreak\fi\quad {\rm Q.E.D.}} 

\newcommand{\C}{\mathbb{C}}

\renewcommand{\P}{\mathbb{P}}

\newcommand{\RR}{\mathbb{R}}
\newcommand{\cS}{\mathcal{S}}

\newcommand{\x}{\mathbf{x}}
\newcommand{\y}{\mathbf{y}}

\newcommand{\bR}{\mathbf{R}}

\definecolor{darkgreen}{rgb}{0,0.5,0}
\definecolor{byzantium}{rgb}{0.44, 0.16, 0.39}
\definecolor{brilliantrose}{rgb}{1.0, 0.33, 0.64}
\definecolor{brinkpink}{rgb}{0.98, 0.38, 0.5}
\definecolor{ao(english)}{rgb}{0.0, 0.5, 0.0}	\definecolor{britishracinggreen}{rgb}{0.0, 0.26, 0.15}
\definecolor{chartreuse(web)}{rgb}{0.5, 1.0, 0.0}
\definecolor{electricgreen}{rgb}{0.0, 1.0, 0.0}
\definecolor{electriclime}{rgb}{0.8, 1.0, 0.0}
\definecolor{emerald}{rgb}{0.31, 0.78, 0.47}



%% file: figures/quiver-examples.tex
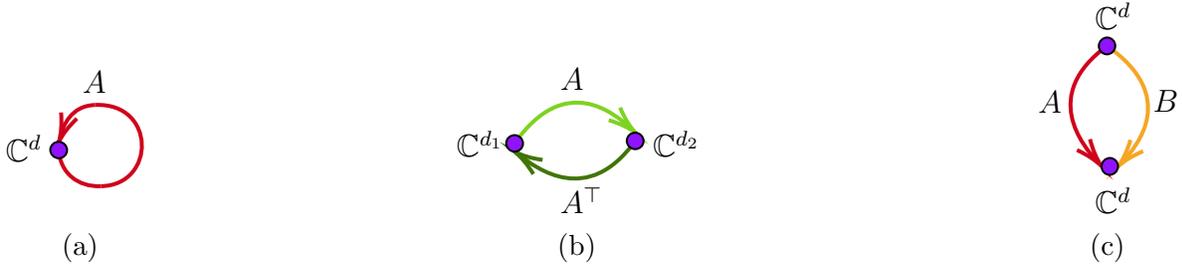
\begin{figure}[htbp]
     \centering
     \begin{subfigure}[b]{0.145\textwidth}
         \centering
\tikzset{every picture/.style={line width=0.75pt}} 

\begin{tikzpicture}[x=0.75pt,y=0.75pt,yscale=-1,xscale=1]

\draw [color={rgb, 255:red, 208; green, 2; blue, 27 }  ,draw opacity=1 ][line width=1.5]    (143.5,452.3) .. controls (131.33,452.69) and (127.48,462.79) .. (126.05,468) ;
\draw [shift={(125.29,470.89)}, rotate = 291.05] [color={rgb, 255:red, 208; green, 2; blue, 27 }  ,draw opacity=1 ][line width=1.5]    (14.21,-4.28) .. controls (9.04,-1.82) and (4.3,-0.39) .. (0,0) .. controls (4.3,0.39) and (9.04,1.82) .. (14.21,4.28)   ;
\draw [color={rgb, 255:red, 208; green, 2; blue, 27 }  ,draw opacity=1 ][line width=1.5]    (146.5,493.3) .. controls (174.5,492.8) and (174.5,450.8) .. (143.5,452.3) ;
\draw [color={rgb, 255:red, 208; green, 2; blue, 27 }  ,draw opacity=1 ][line width=1.5]    (125.19,475.04) .. controls (126,486.8) and (135,493.8) .. (146.5,493.3) ;
\draw  [fill={rgb, 255:red, 144; green, 19; blue, 254 }  ,fill opacity=1 ] (125.29,470.89) .. controls (127.63,470.97) and (129.49,472.9) .. (129.43,475.19) .. controls (129.38,477.47) and (127.43,479.26) .. (125.09,479.18) .. controls (122.75,479.1) and (120.9,477.18) .. (120.95,474.89) .. controls (121.01,472.6) and (122.95,470.81) .. (125.29,470.89) -- cycle ;

\draw (135.8,433.55) node [anchor=north west][inner sep=0.75pt]    {$A$};
\draw (97,466) node [anchor=north west][inner sep=0.75pt]    {$\mathbb{C}^d$};
\draw (134.8,494.55) node [anchor=north west][inner sep=0.75pt]    {$\phantom{A}$};
\end{tikzpicture}
         \caption{}
         \label{fig:Jordan-quiver}
     \end{subfigure}
     \hfill
     \begin{subfigure}[b]{0.2\textwidth}
         \centering
\tikzset{every picture/.style={line width=0.75pt}} 

\begin{tikzpicture}[x=0.75pt,y=0.75pt,yscale=-1,xscale=1]

\draw [color={rgb, 255:red, 126; green, 211; blue, 33 }  ,draw opacity=1 ][line width=1.5]    (247.69,471.04) .. controls (273.17,434.82) and (298.07,455.79) .. (306.44,463.62) ;
\draw [shift={(308.6,465.65)}, rotate = 221.99] [color={rgb, 255:red, 126; green, 211; blue, 33 }  ,draw opacity=1 ][line width=1.5]    (14.21,-4.28) .. controls (9.04,-1.82) and (4.3,-0.39) .. (0,0) .. controls (4.3,0.39) and (9.04,1.82) .. (14.21,4.28)   ;
\draw [color={rgb, 255:red, 65; green, 117; blue, 5 }  ,draw opacity=1 ][line width=1.5]    (308.5,469.8) .. controls (284.52,503.28) and (258.9,483.74) .. (249.87,476.85) ;
\draw [shift={(247.59,475.18)}, rotate = 33.2] [color={rgb, 255:red, 65; green, 117; blue, 5 }  ,draw opacity=1 ][line width=1.5]    (14.21,-4.28) .. controls (9.04,-1.82) and (4.3,-0.39) .. (0,0) .. controls (4.3,0.39) and (9.04,1.82) .. (14.21,4.28)   ;
\draw  [fill={rgb, 255:red, 144; green, 19; blue, 254 }  ,fill opacity=1 ] (247.79,466.89) .. controls (250.13,466.97) and (251.99,468.9) .. (251.93,471.19) .. controls (251.88,473.47) and (249.93,475.26) .. (247.59,475.18) .. controls (245.25,475.1) and (243.4,473.18) .. (243.45,470.89) .. controls (243.51,468.6) and (245.45,466.81) .. (247.79,466.89) -- cycle ;
\draw  [fill={rgb, 255:red, 144; green, 19; blue, 254 }  ,fill opacity=1 ] (308.6,465.65) .. controls (310.94,465.73) and (312.79,467.66) .. (312.74,469.95) .. controls (312.69,472.23) and (310.74,474.02) .. (308.4,473.94) .. controls (306.06,473.86) and (304.21,471.94) .. (304.26,469.65) .. controls (304.31,467.36) and (306.26,465.57) .. (308.6,465.65) -- cycle ;

\draw (269.3,431.55) node [anchor=north west][inner sep=0.75pt]    {$A$};
\draw (269.3,491.55) node [anchor=north west][inner sep=0.75pt]    {$A^{\top }$};
\draw (217,463) node [anchor=north west][inner sep=0.75pt]    {$\mathbb{C}^{d_1}$};
\draw (316,463) node [anchor=north west][inner sep=0.75pt]    {$\mathbb{C}^{d_2}$};

\end{tikzpicture}
         \caption{}
         \label{fig:Loop-quiver}
     \end{subfigure}
     \hfill
     \begin{subfigure}[b]{0.2\textwidth}
         \centering
\tikzset{every picture/.style={line width=0.75pt}} 

\begin{tikzpicture}[x=0.75pt,y=0.75pt,yscale=-1,xscale=1]

\draw [color={rgb, 255:red, 245; green, 166; blue, 35 }  ,draw opacity=1 ][line width=1.5]    (484.49,435.19) .. controls (520.76,460.58) and (499.85,485.53) .. (492.04,493.92) ;
\draw [shift={(490.01,496.08)}, rotate = 311.86] [color={rgb, 255:red, 245; green, 166; blue, 35 }  ,draw opacity=1 ][line width=1.5]    (14.21,-4.28) .. controls (9.04,-1.82) and (4.3,-0.39) .. (0,0) .. controls (4.3,0.39) and (9.04,1.82) .. (14.21,4.28)   ;
\draw [color={rgb, 255:red, 208; green, 2; blue, 27 }  ,draw opacity=1 ][line width=1.5]    (484.49,435.19) .. controls (450.68,459.98) and (472.31,485.68) .. (479.75,493.75) ;
\draw [shift={(481.72,495.91)}, rotate = 230.14] [color={rgb, 255:red, 208; green, 2; blue, 27 }  ,draw opacity=1 ][line width=1.5]    (14.21,-4.28) .. controls (9.04,-1.82) and (4.3,-0.39) .. (0,0) .. controls (4.3,0.39) and (9.04,1.82) .. (14.21,4.28)   ;
\draw  [fill={rgb, 255:red, 144; green, 19; blue, 254 }  ,fill opacity=1 ] (490.01,496.08) .. controls (489.94,498.43) and (488.02,500.28) .. (485.73,500.24) .. controls (483.44,500.19) and (481.64,498.25) .. (481.72,495.91) .. controls (481.8,493.56) and (483.72,491.71) .. (486.01,491.75) .. controls (488.3,491.8) and (490.09,493.74) .. (490.01,496.08) -- cycle ;
\draw  [fill={rgb, 255:red, 144; green, 19; blue, 254 }  ,fill opacity=1 ] (488.63,435.28) .. controls (488.56,437.62) and (486.64,439.48) .. (484.35,439.43) .. controls (482.06,439.38) and (480.27,437.44) .. (480.34,435.1) .. controls (480.42,432.76) and (482.34,430.9) .. (484.63,430.95) .. controls (486.92,431) and (488.71,432.94) .. (488.63,435.28) -- cycle ;

\draw (448.3,456.55) node [anchor=north west][inner sep=0.75pt]    {$A$};
\draw (506.3,456.05) node [anchor=north west][inner sep=0.75pt]    {$B$};
\draw (477,412) node [anchor=north west][inner sep=0.75pt]    {$\mathbb{C}^d$};
\draw (477,505) node [anchor=north west][inner sep=0.75pt]    {$\mathbb{C}^d$};

\end{tikzpicture}

         \caption{}
         \label{fig:Kronecker-quiver}
     \end{subfigure}
        \caption{Quiver representations corresponding to (a) the eigenvectors of a matrix, (b) the singular vectors of a matrix, (c) the generalised eigenvectors of a pair of matrices.}
        \label{fig:quiver-examples}
\end{figure}

%% file: figures/tensor-eigenvector.tex
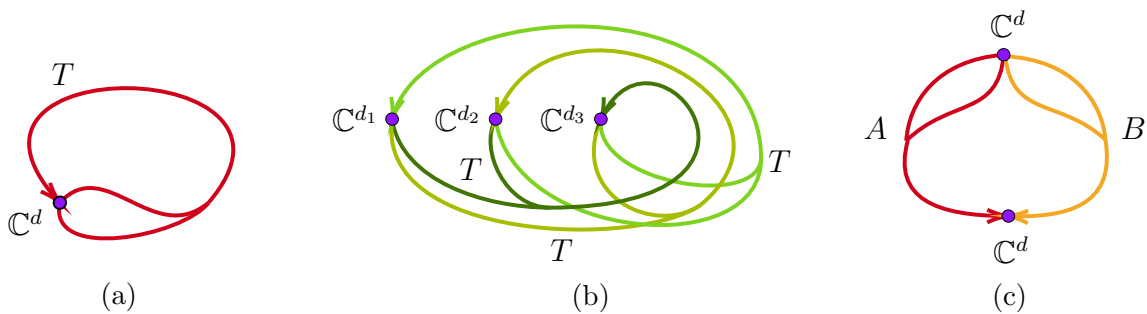
\begin{figure}[htbp]
     \centering
     \begin{subfigure}[b]{0.325\textwidth}
\tikzset{every picture/.style={line width=0.7pt}} 

\begin{tikzpicture}[x=0.75pt,y=0.75pt,yscale=-0.75,xscale=0.75]

\draw [color={rgb, 255:red, 208; green, 2; blue, 27 }  ,draw opacity=1 ][line width=1.5]    (487.23,505.56) .. controls (573.32,405.3) and (292.58,400.84) .. (385.79,503.99) ;
\draw [shift={(387.23,505.56)}, rotate = 227.06] [color={rgb, 255:red, 208; green, 2; blue, 27 }  ,draw opacity=1 ][line width=1.5]    (14.21,-4.28) .. controls (9.04,-1.82) and (4.3,-0.39) .. (0,0) .. controls (4.3,0.39) and (9.04,1.82) .. (14.21,4.28)   ;
\draw [color={rgb, 255:red, 208; green, 2; blue, 27 }  ,draw opacity=1 ][line width=1.5]    (387.23,505.56) .. controls (378.75,545.8) and (469.75,528.3) .. (487.23,505.56) ;
\draw [color={rgb, 255:red, 208; green, 2; blue, 27 }  ,draw opacity=1 ][line width=1.5]    (387.23,505.56) .. controls (417.75,472.3) and (450.25,536.3) .. (487.23,505.56) ;
\draw  [fill={rgb, 255:red, 144; green, 19; blue, 254 }  ,fill opacity=1 ] (383.43,503.89) .. controls (384.4,501.75) and (386.88,500.77) .. (388.98,501.69) .. controls (391.08,502.61) and (391.99,505.09) .. (391.02,507.23) .. controls (390.06,509.36) and (387.57,510.34) .. (385.48,509.42) .. controls (383.38,508.5) and (382.47,506.02) .. (383.43,503.89) -- cycle ;

\draw (350,507.91) node [anchor=north west][inner sep=0.75pt]    {$\mathbb{C}^{d}$};
\draw (380,410) node [anchor=north west][inner sep=0.75pt]    {$T$};
\end{tikzpicture}
         \caption{}
         \label{fig:loop-hyperquiver-ex}
     \end{subfigure}
     \hfill
     \begin{subfigure}[b]{0.4203\textwidth}
\tikzset{every picture/.style={line width=0.1pt}} 

\begin{tikzpicture}[x=0.75pt,y=0.75pt,yscale=-0.75,xscale=0.75]

\draw [color={rgb, 255:red, 65; green, 117; blue, 5 }  ,draw opacity=1 ][line width=1.5]    (382.82,596.23) .. controls (389,637.9) and (445,655.8) .. (489,655.8) ;
\draw [color={rgb, 255:red, 65; green, 117; blue, 5 }  ,draw opacity=1 ][line width=1.5]    (452.42,596.23) .. controls (441,617.1) and (459.4,655.8) .. (489,655.8) ;
\draw [color={rgb, 255:red, 166; green, 189; blue, 0 }  ,draw opacity=1 ][line width=1.5]    (382.82,596.23) .. controls (371.5,679.3) and (549.1,683.4) .. (590.5,653.8) ;
\draw [color={rgb, 255:red, 166; green, 189; blue, 0 }  ,draw opacity=1 ][line width=1.5]    (590.5,653.8) .. controls (560.1,675.4) and (501.5,648.3) .. (523.22,596.23) ;
\draw [color={rgb, 255:red, 126; green, 211; blue, 33 }  ,draw opacity=1 ][line width=1.5]    (452.42,596.23) .. controls (464.2,678.7) and (632.5,692.8) .. (631,621.8) ;
\draw [color={rgb, 255:red, 126; green, 211; blue, 33 }  ,draw opacity=1 ][line width=1.5]    (523.22,596.23) .. controls (517.8,639.5) and (631,657.8) .. (631,621.8) ;
\draw [color={rgb, 255:red, 126; green, 211; blue, 33 }  ,draw opacity=1 ][line width=1.5]    (631,621.8) .. controls (629.02,505.97) and (408.48,512.89) .. (383.52,593.76) ;
\draw [shift={(382.82,596.23)}, rotate = 284.28] [color={rgb, 255:red, 126; green, 211; blue, 33 }  ,draw opacity=1 ][line width=1.5]    (14.21,-4.28) .. controls (9.04,-1.82) and (4.3,-0.39) .. (0,0) .. controls (4.3,0.39) and (9.04,1.82) .. (14.21,4.28)   ;
\draw [color={rgb, 255:red, 65; green, 117; blue, 5 }  ,draw opacity=1 ][line width=1.5]    (489,655.8) .. controls (666.51,655.5) and (560.67,520.54) .. (524.3,593.94) ;
\draw [shift={(523.22,596.23)}, rotate = 294.22] [color={rgb, 255:red, 65; green, 117; blue, 5 }  ,draw opacity=1 ][line width=1.5]    (14.21,-4.28) .. controls (9.04,-1.82) and (4.3,-0.39) .. (0,0) .. controls (4.3,0.39) and (9.04,1.82) .. (14.21,4.28)   ;
\draw [color={rgb, 255:red, 166; green, 189; blue, 0 }  ,draw opacity=1 ][line width=1.5]    (590.5,653.8) .. controls (675.08,582.65) and (491.35,495.18) .. (452.99,594.72) ;
\draw [shift={(452.42,596.23)}, rotate = 289.99] [color={rgb, 255:red, 166; green, 189; blue, 0 }  ,draw opacity=1 ][line width=1.5]    (14.21,-4.28) .. controls (9.04,-1.82) and (4.3,-0.39) .. (0,0) .. controls (4.3,0.39) and (9.04,1.82) .. (14.21,4.28)   ;
\draw  [fill={rgb, 255:red, 144; green, 19; blue, 254 }  ,fill opacity=1 ] (382.97,592.09) .. controls (385.31,592.2) and (387.14,594.15) .. (387.06,596.44) .. controls (386.97,598.72) and (385.01,600.49) .. (382.67,600.38) .. controls (380.33,600.26) and (378.5,598.32) .. (378.58,596.03) .. controls (378.67,593.74) and (380.63,591.98) .. (382.97,592.09) -- cycle ;
\draw  [fill={rgb, 255:red, 144; green, 19; blue, 254 }  ,fill opacity=1 ] (452.57,592.09) .. controls (454.91,592.2) and (456.74,594.15) .. (456.66,596.44) .. controls (456.57,598.72) and (454.61,600.49) .. (452.27,600.38) .. controls (449.93,600.26) and (448.1,598.32) .. (448.18,596.03) .. controls (448.27,593.74) and (450.23,591.98) .. (452.57,592.09) -- cycle ;
\draw  [fill={rgb, 255:red, 144; green, 19; blue, 254 }  ,fill opacity=1 ] (523.37,592.09) .. controls (525.71,592.2) and (527.54,594.15) .. (527.46,596.44) .. controls (527.37,598.72) and (525.41,600.49) .. (523.07,600.38) .. controls (520.73,600.26) and (518.9,598.32) .. (518.98,596.03) .. controls (519.07,593.74) and (521.03,591.98) .. (523.37,592.09) -- cycle ;

\draw (341,585.41) node [anchor=north west][inner sep=0.75pt]    {$\mathbb{C}^{d_{1}}$};
\draw (410,585.41) node [anchor=north west][inner sep=0.75pt]    {$\mathbb{C}^{d_{2}}$};
\draw (481,585.41) node [anchor=north west][inner sep=0.75pt]    {$\mathbb{C}^{d_{3}}$};
\draw (488,675.61) node [anchor=north west][inner sep=0.75pt]    {$T$};
\draw (427,621) node [anchor=north west][inner sep=0.75pt]    {$T$};
\draw (634.5,615.8) node [anchor=north west][inner sep=0.75pt]    {$T$};

\end{tikzpicture}
         \caption{}
         \label{fig:Jordan-hyperquiver-ex}
     \end{subfigure}
     \hfill
     \begin{subfigure}[b]{0.24\textwidth}
\tikzset{every picture/.style={line width=0.3pt}} 

\begin{tikzpicture}[x=0.75pt,y=0.75pt,yscale=-0.75,xscale=0.75]

\draw [color={rgb, 255:red, 208; green, 2; blue, 27 }  ,draw opacity=1 ][line width=1.5]    (388.53,391.24) .. controls (343.35,394.88) and (324.71,423.12) .. (323.04,449.14) ;
\draw [color={rgb, 255:red, 208; green, 2; blue, 27 }  ,draw opacity=1 ][line width=1.5]    (388.53,391.24) .. controls (384.84,433.36) and (352.68,420.26) .. (323.04,449.14) ;
\draw [color={rgb, 255:red, 245; green, 166; blue, 35 }  ,draw opacity=1 ][line width=1.5]    (389.54,392.23) .. controls (435.12,394.71) and (454.61,422.47) .. (456.96,448.44) ;
\draw [color={rgb, 255:red, 245; green, 166; blue, 35 }  ,draw opacity=1 ][line width=1.5]    (389.54,392.23) .. controls (394.33,434.24) and (426.38,420.33) .. (456.96,448.44) ;
\draw  [fill={rgb, 255:red, 144; green, 19; blue, 254 }  ,fill opacity=1 ] (388.63,387.1) .. controls (390.97,387.18) and (392.82,389.1) .. (392.77,391.39) .. controls (392.71,393.68) and (390.77,395.47) .. (388.43,395.39) .. controls (386.09,395.3) and (384.23,393.38) .. (384.29,391.09) .. controls (384.34,388.8) and (386.29,387.01) .. (388.63,387.1) -- cycle ;
\draw [color={rgb, 255:red, 208; green, 2; blue, 27 }  ,draw opacity=1 ][line width=1.5]    (324.04,449.14) .. controls (313.22,488.98) and (344.7,499.38) .. (388.97,499.78) ;
\draw [shift={(391.69,499.79)}, rotate = 179.99] [color={rgb, 255:red, 208; green, 2; blue, 27 }  ,draw opacity=1 ][line width=1.5]    (14.21,-4.28) .. controls (9.04,-1.82) and (4.3,-0.39) .. (0,0) .. controls (4.3,0.39) and (9.04,1.82) .. (14.21,4.28)   ;
\draw [color={rgb, 255:red, 245; green, 166; blue, 35 }  ,draw opacity=1 ][line width=1.5]    (456.96,448.44) .. controls (462.88,486.03) and (441.87,501.19) .. (393.68,499.89) ;
\draw [shift={(390.69,499.79)}, rotate = 2.29] [color={rgb, 255:red, 245; green, 166; blue, 35 }  ,draw opacity=1 ][line width=1.5]    (14.21,-4.28) .. controls (9.04,-1.82) and (4.3,-0.39) .. (0,0) .. controls (4.3,0.39) and (9.04,1.82) .. (14.21,4.28)   ;
\draw  [fill={rgb, 255:red, 144; green, 19; blue, 254 }  ,fill opacity=1 ] (391.79,495.64) .. controls (394.13,495.72) and (395.99,497.65) .. (395.93,499.94) .. controls (395.88,502.22) and (393.93,504.01) .. (391.59,503.93) .. controls (389.25,503.85) and (387.4,501.93) .. (387.45,499.64) .. controls (387.51,497.35) and (389.45,495.56) .. (391.79,495.64) -- cycle ;

\draw (292.8,431.61) node [anchor=north west][inner sep=0.75pt]    {$A$};
\draw (465.8,431.61) node [anchor=north west][inner sep=0.75pt]    {$B$};
\draw (377.8,360) node [anchor=north west][inner sep=0.75pt]    {$\mathbb{C}^{d}$};
\draw (379.8,511.91) node [anchor=north west][inner sep=0.75pt]    {$\mathbb{C}^{d}$};

\end{tikzpicture}

         \caption{}
         \label{fig:Kronecker-hyperquiver-ex}
     \end{subfigure}
    \caption{Examples of hyperquiver representations corresponding to (a) the eigenvectors of a tensor, (b) the singular vectors of a tensor, and (c) the generalised eigenvectors of a pair of tensors.}
        \label{fig:hyperquiver-examples}
\end{figure}

%% file: figures/labelled-edges.tex
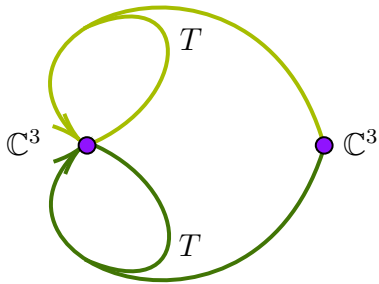
\begin{figure}[htbp]
    \centering
\tikzset{every picture/.style={line width=0.75pt}} 

\begin{tikzpicture}[x=0.75pt,y=0.75pt,yscale=-1,xscale=1]

\draw [color={rgb, 255:red, 166; green, 189; blue, 0 }  ,draw opacity=1 ][line width=1.5]    (489.89,485.51) .. controls (456.01,367.64) and (307.67,422.82) .. (364.39,480.7) ;
\draw [shift={(366.17,482.45)}, rotate = 223.57] [color={rgb, 255:red, 166; green, 189; blue, 0 }  ,draw opacity=1 ][line width=1.5]    (14.21,-4.28) .. controls (9.04,-1.82) and (4.3,-0.39) .. (0,0) .. controls (4.3,0.39) and (9.04,1.82) .. (14.21,4.28)   ;
\draw [color={rgb, 255:red, 166; green, 189; blue, 0 }  ,draw opacity=1 ][line width=1.5]    (370.29,485.51) .. controls (424.17,463.45) and (425.67,400.95) .. (368.67,426.45) ;
\draw [color={rgb, 255:red, 65; green, 117; blue, 5 }  ,draw opacity=1 ][line width=1.5]    (490.64,483.89) .. controls (456.78,601.79) and (308.06,546.75) .. (364.85,488.82) ;
\draw [shift={(366.63,487.06)}, rotate = 136.45] [color={rgb, 255:red, 65; green, 117; blue, 5 }  ,draw opacity=1 ][line width=1.5]    (14.21,-4.28) .. controls (9.04,-1.82) and (4.3,-0.39) .. (0,0) .. controls (4.3,0.39) and (9.04,1.82) .. (14.21,4.28)   ;
\draw [color={rgb, 255:red, 65; green, 117; blue, 5 }  ,draw opacity=1 ][line width=1.5]    (370.76,484) .. controls (424.78,506.01) and (426.34,568.51) .. (369.19,543.06) ;

\draw  [fill={rgb, 255:red, 144; green, 19; blue, 254 }  ,fill opacity=1 ] (370.44,481.37) .. controls (372.78,481.48) and (374.61,483.43) .. (374.52,485.72) .. controls (374.44,488) and (372.48,489.77) .. (370.14,489.66) .. controls (367.79,489.55) and (365.97,487.6) .. (366.05,485.31) .. controls (366.13,483.02) and (368.1,481.26) .. (370.44,481.37) -- cycle ;
\draw  [fill={rgb, 255:red, 144; green, 19; blue, 254 }  ,fill opacity=1 ] (490.04,481.37) .. controls (492.38,481.48) and (494.21,483.43) .. (494.12,485.72) .. controls (494.04,488) and (492.08,489.77) .. (489.74,489.66) .. controls (487.39,489.55) and (485.57,487.6) .. (485.65,485.31) .. controls (485.73,483.02) and (487.7,481.26) .. (490.04,481.37) -- cycle ;

\draw (328,476) node [anchor=north west][inner sep=0.75pt]    {$\mathbb{C}^{3}$};
\draw (498,476) node [anchor=north west][inner sep=0.75pt]    {$\mathbb{C}^{3}$};
\draw (415.63,425) node [anchor=north west][inner sep=0.75pt]    {$T$};
\draw (415,529) node [anchor=north west][inner sep=0.75pt]    {$T$};

\end{tikzpicture}
    \vspace{-2em}
    \caption{The light-green hyperedge is the contraction $T(\, \cdot \,, \x, \y)$ and the dark-green hyperedge is the contraction $T(\x, \, \cdot \,, \y)$,
where $\x, \y \in \C^3$ are on the left and right vertices respectively. }
    \label{fig:labelled-edges}
\end{figure}

%% file: figures/partially-symmetric.tex
\begin{figure}[htbp]
    \centering
    
\tikzset{every picture/.style={line width=0.75pt}} 

\begin{tikzpicture}[x=0.75pt,y=0.75pt,yscale=-1,xscale=1]

\draw [color={rgb, 255:red, 166; green, 189; blue, 0 }  ,draw opacity=1 ][line width=1.5]    (514.55,539.18) .. controls (480.68,421.31) and (332.34,476.49) .. (389.05,534.37) ;
\draw [shift={(390.83,536.12)}, rotate = 223.57] [color={rgb, 255:red, 166; green, 189; blue, 0 }  ,draw opacity=1 ][line width=1.5]    (14.21,-4.28) .. controls (9.04,-1.82) and (4.3,-0.39) .. (0,0) .. controls (4.3,0.39) and (9.04,1.82) .. (14.21,4.28)   ;
\draw [color={rgb, 255:red, 166; green, 189; blue, 0 }  ,draw opacity=1 ][line width=1.5]    (394.95,539.18) .. controls (448.83,517.12) and (450.33,454.62) .. (393.33,480.12) ;
\draw [color={rgb, 255:red, 65; green, 117; blue, 5 }  ,draw opacity=1 ][line width=1.5]    (392.97,540.22) .. controls (452.04,547.91) and (464.21,631.78) .. (513.65,544.65) ;
\draw [shift={(514.4,543.32)}, rotate = 119.28] [color={rgb, 255:red, 65; green, 117; blue, 5 }  ,draw opacity=1 ][line width=1.5]    (14.21,-4.28) .. controls (9.04,-1.82) and (4.3,-0.39) .. (0,0) .. controls (4.3,0.39) and (9.04,1.82) .. (14.21,4.28)   ;
\draw [color={rgb, 255:red, 65; green, 117; blue, 5 }  ,draw opacity=1 ][line width=1.5]    (392.97,540.22) .. controls (357.67,583.29) and (435,597.29) .. (477.67,583.95) ;
\draw  [fill={rgb, 255:red, 144; green, 19; blue, 254 }  ,fill opacity=1 ] (395.1,535.04) .. controls (397.44,535.15) and (399.27,537.09) .. (399.19,539.38) .. controls (399.11,541.67) and (397.14,543.44) .. (394.8,543.32) .. controls (392.46,543.21) and (390.63,541.27) .. (390.72,538.98) .. controls (390.8,536.69) and (392.76,534.92) .. (395.1,535.04) -- cycle ;
\draw  [fill={rgb, 255:red, 144; green, 19; blue, 254 }  ,fill opacity=1 ] (514.7,535.04) .. controls (517.04,535.15) and (518.87,537.09) .. (518.79,539.38) .. controls (518.71,541.67) and (516.74,543.44) .. (514.4,543.32) .. controls (512.06,543.21) and (510.23,541.27) .. (510.32,538.98) .. controls (510.4,536.69) and (512.36,534.92) .. (514.7,535.04) -- cycle ;

\draw (351,529) node [anchor=north west][inner sep=0.75pt]    {$\mathbb{C}^{3}$};
\draw (525.23,529) node [anchor=north west][inner sep=0.75pt]    {$\mathbb{C}^{3}$};
\draw (443.97,481.02) node [anchor=north west][inner sep=0.75pt]    {$\mathcal{T}$};
\draw (465.3,556.69) node [anchor=north west][inner sep=0.75pt]    {$\mathcal{T}$};

\end{tikzpicture}

    \caption{A representation corresponding to the symmetric singular vector tuples of a partially symmetric tensor $T_{ijk} = T_{jik}$. The light-green and dark-green edges are the contractions $T(\x,\,\cdot\,,\mathbf{y})$ and $T(\x,\x,\,\cdot\,)$ respectively.}
    \label{fig:partially-symmetric}
\end{figure}

%% file: figures/oneedge.tex
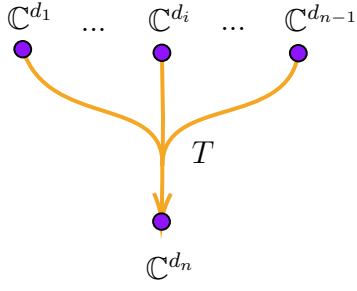
\begin{figure}[htbp]
    \centering
    
\tikzset{every picture/.style={line width=0.75pt}} 

\begin{tikzpicture}[x=0.75pt,y=0.75pt,yscale=-1,xscale=1]

\draw [color={rgb, 255:red, 245; green, 166; blue, 35 }  ,draw opacity=1 ][line width=1.5]    (85.42,120.63) .. controls (98.2,161.07) and (157.4,143.47) .. (155.8,179.47) ;
\draw [color={rgb, 255:red, 245; green, 166; blue, 35 }  ,draw opacity=1 ][line width=1.5]    (155.62,122.43) .. controls (155.8,144.27) and (156.6,151.47) .. (155.8,179.47) ;
\draw [color={rgb, 255:red, 245; green, 166; blue, 35 }  ,draw opacity=1 ][line width=1.5]    (155.8,179.47) .. controls (155.8,143.47) and (219,161.07) .. (224.42,122.63) ;
\draw [color={rgb, 255:red, 245; green, 166; blue, 35 }  ,draw opacity=1 ][line width=1.5]    (155.8,179.47) .. controls (155.8,180.94) and (155.8,193.25) .. (155.5,204.68) ;
\draw [shift={(155.42,207.63)}, rotate = 271.79] [color={rgb, 255:red, 245; green, 166; blue, 35 }  ,draw opacity=1 ][line width=1.5]    (14.21,-4.28) .. controls (9.04,-1.82) and (4.3,-0.39) .. (0,0) .. controls (4.3,0.39) and (9.04,1.82) .. (14.21,4.28)   ;
\draw  [fill={rgb, 255:red, 144; green, 19; blue, 254 }  ,fill opacity=1 ] (155.57,203.49) .. controls (157.91,203.6) and (159.74,205.55) .. (159.66,207.84) .. controls (159.57,210.12) and (157.61,211.89) .. (155.27,211.78) .. controls (152.93,211.66) and (151.1,209.72) .. (151.18,207.43) .. controls (151.27,205.14) and (153.23,203.38) .. (155.57,203.49) -- cycle ;
\draw  [fill={rgb, 255:red, 144; green, 19; blue, 254 }  ,fill opacity=1 ] (85.57,116.49) .. controls (87.91,116.6) and (89.74,118.55) .. (89.66,120.84) .. controls (89.57,123.12) and (87.61,124.89) .. (85.27,124.78) .. controls (82.93,124.66) and (81.1,122.72) .. (81.18,120.43) .. controls (81.27,118.14) and (83.23,116.38) .. (85.57,116.49) -- cycle ;
\draw  [fill={rgb, 255:red, 144; green, 19; blue, 254 }  ,fill opacity=1 ] (155.77,118.29) .. controls (158.11,118.4) and (159.94,120.35) .. (159.86,122.64) .. controls (159.77,124.92) and (157.81,126.69) .. (155.47,126.58) .. controls (153.13,126.46) and (151.3,124.52) .. (151.38,122.23) .. controls (151.47,119.94) and (153.43,118.18) .. (155.77,118.29) -- cycle ;
\draw  [fill={rgb, 255:red, 144; green, 19; blue, 254 }  ,fill opacity=1 ] (224.57,118.49) .. controls (226.91,118.6) and (228.74,120.55) .. (228.66,122.84) .. controls (228.57,125.12) and (226.61,126.89) .. (224.27,126.78) .. controls (221.93,126.66) and (220.1,124.72) .. (220.18,122.43) .. controls (220.27,120.14) and (222.23,118.38) .. (224.57,118.49) -- cycle ;

\draw (76,95.81) node [anchor=north west][inner sep=0.75pt]    {$\mathbb{C}^{d_1}$};
\draw (216,96.81) node [anchor=north west][inner sep=0.75pt]    {$\mathbb{C}^{d_{n-1}}$};
\draw (146,221.81) node [anchor=north west][inner sep=0.75pt]    {$\mathbb{C}^{d_n}$};
\draw (147.2,96.61) node [anchor=north west][inner sep=0.75pt]    {$\mathbb{C}^{d_i}$};
\draw (113.2,108.47) node [anchor=north west][inner sep=0.75pt]    {$...$};
\draw (183.6,107.87) node [anchor=north west][inner sep=0.75pt]    {$...$};
\draw (169.2,165.74) node [anchor=north west][inner sep=0.75pt]    {$T$};

\end{tikzpicture}

    \caption{A hyperquiver with a single hyperedge and a representation}
    \label{generalSingleHyperedge}
\end{figure}

%% file: figures/dimensionDegreeTable.tex
\begin{table}[htbp]
\centering
\resizebox{0.75\columnwidth}{!}
{\begin{tabular}{|l|l|l|l|l|l|}
\hline
\backslashbox{$d$}{$n$} \multirow{1}{*}{} & \multicolumn{1}{l|}{2} & \multicolumn{1}{l|}{3} & \multicolumn{1}{l|}{4} & \multicolumn{1}{l|}{5}                                & \multicolumn{1}{l|}{6}                                 \\ 
\cline{2-6}
\hline
1                 & 1               & 1               & 1               & 1                                              & 1                                               \\ 
\hline
2                 & 2              & 6               & 24             & 120                                             & 720                                             \\ 
\hline
3                 & 4               & 66              & 1980            & 93240                                           & 6350400                                         \\ 
\hline
4                 & 8              & 840             & 218400           & 110510000                                      & 96864800000                                     \\ 
\hline
5                  & 16               & 11410            & 27512100        & $1.5873 \times 10^{11}$  & $1.89313 \times 10^{15}$  \\ 
\hline
6                 & 32              & 160776          & 3741400000      & $2.54601 \times 10^{14}$ & $4.26416 \times 10^{19}$  \\
\hline
\end{tabular}}
\caption{The degree of the singular vector variety $\mathcal{S}(\mathbf{R})$ of the hyperquiver in Figure \ref{generalSingleHyperedge} with $d_1 = ... = d_n = d$ and generic tensor $T$. The dimension of $\mathcal{S}(\mathbf{R})$ is $N = (d-1)(n-1)$. In particular, $\mathcal{S}(\mathbf{R})$ is positive-dimensional except in the first row.}\label{tab:degreeOfSingleEdgeHyperquiver}
\end{table}

%% file: figures/noSingularVectors-new.tex
\begin{figure}[htbp]
    \centering
\tikzset{every picture/.style={line width=0.75pt}} 

\begin{tikzpicture}[x=0.75pt,y=0.75pt,yscale=-1,xscale=1]

\draw [color={rgb, 255:red, 74; green, 144; blue, 226 }  ,draw opacity=1 ][line width=1.5]    (435.69,145.79) .. controls (488.69,87.34) and (504.53,158.75) .. (442.81,146.55) ;
\draw [shift={(439.93,145.94)}, rotate = 13] [color={rgb, 255:red, 74; green, 144; blue, 226 }  ,draw opacity=1 ][line width=1.5]    (14.21,-4.28) .. controls (9.04,-1.82) and (4.3,-0.39) .. (0,0) .. controls (4.3,0.39) and (9.04,1.82) .. (14.21,4.28)   ;
\draw [color={rgb, 255:red, 208; green, 2; blue, 27 }  ,draw opacity=1 ][line width=1.5]    (435.69,145.79) .. controls (383.53,88.04) and (363.4,159.83) .. (429.42,146.08) ;
\draw [shift={(431.45,145.64)}, rotate = 167.11] [color={rgb, 255:red, 208; green, 2; blue, 27 }  ,draw opacity=1 ][line width=1.5]    (14.21,-4.28) .. controls (9.04,-1.82) and (4.3,-0.39) .. (0,0) .. controls (4.3,0.39) and (9.04,1.82) .. (14.21,4.28)   ;
\draw  [fill={rgb, 255:red, 144; green, 19; blue, 254 }  ,fill opacity=1 ] (435.79,141.64) .. controls (438.13,141.72) and (439.99,143.65) .. (439.93,145.94) .. controls (439.88,148.22) and (437.93,150.01) .. (435.59,149.93) .. controls (433.25,149.85) and (431.4,147.93) .. (431.45,145.64) .. controls (431.51,143.35) and (433.45,141.56) .. (435.79,141.64) -- cycle ;

\draw (368.8,121.3) node [anchor=north west][inner sep=0.75pt]    {$A$};
\draw (484,121.3) node [anchor=north west][inner sep=0.75pt]    {$B$};
\draw (426.8,115.95) node [anchor=north west][inner sep=0.75pt]    {$\mathbb{C}^{d}$};

\end{tikzpicture}
    \caption{A quiver representation with empty singular vector variety}
    \label{noSingularVectorQuivers}
\end{figure}
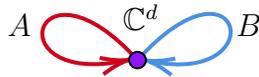

%% file: figures/insufficient-triangle.tex
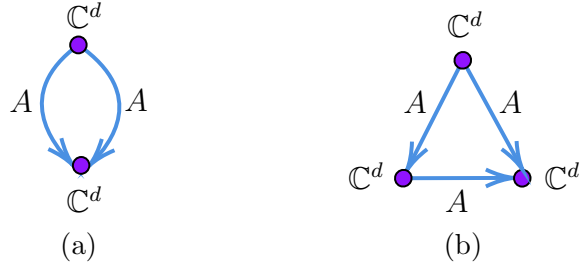
\begin{figure}[htbp]
     \centering
     \begin{subfigure}[b]{0.145\textwidth}
         \centering
\tikzset{every picture/.style={line width=0.75pt}} 

\begin{tikzpicture}[x=0.75pt,y=0.75pt,yscale=-1,xscale=1]

\draw [color={rgb, 255:red, 74; green, 144; blue, 226 }  ,draw opacity=1 ][line width=1.5]    (484.49,435.19) .. controls (520.76,460.58) and (499.85,485.53) .. (492.04,493.92) ;
\draw [shift={(490.01,496.08)}, rotate = 311.86] [color={rgb, 255:red, 74; green, 144; blue, 226 }  ,draw opacity=1 ][line width=1.5]    (14.21,-4.28) .. controls (9.04,-1.82) and (4.3,-0.39) .. (0,0) .. controls (4.3,0.39) and (9.04,1.82) .. (14.21,4.28)   ;
\draw [color={rgb, 255:red, 74; green, 144; blue, 226 }  ,draw opacity=1 ][line width=1.5]    (484.49,435.19) .. controls (450.68,459.98) and (472.31,485.68) .. (479.75,493.75) ;
\draw [shift={(481.72,495.91)}, rotate = 230.14] [color={rgb, 255:red, 74; green, 144; blue, 226 } ,draw opacity=1 ][line width=1.5]    (14.21,-4.28) .. controls (9.04,-1.82) and (4.3,-0.39) .. (0,0) .. controls (4.3,0.39) and (9.04,1.82) .. (14.21,4.28)   ;
\draw  [fill={rgb, 255:red, 144; green, 19; blue, 254 }  ,fill opacity=1 ] (490.01,496.08) .. controls (489.94,498.43) and (488.02,500.28) .. (485.73,500.24) .. controls (483.44,500.19) and (481.64,498.25) .. (481.72,495.91) .. controls (481.8,493.56) and (483.72,491.71) .. (486.01,491.75) .. controls (488.3,491.8) and (490.09,493.74) .. (490.01,496.08) -- cycle ;
\draw  [fill={rgb, 255:red, 144; green, 19; blue, 254 }  ,fill opacity=1 ] (488.63,435.28) .. controls (488.56,437.62) and (486.64,439.48) .. (484.35,439.43) .. controls (482.06,439.38) and (480.27,437.44) .. (480.34,435.1) .. controls (480.42,432.76) and (482.34,430.9) .. (484.63,430.95) .. controls (486.92,431) and (488.71,432.94) .. (488.63,435.28) -- cycle ;

\draw (448.3,456.55) node [anchor=north west][inner sep=0.75pt]    {$A$};
\draw (506.3,456.05) node [anchor=north west][inner sep=0.75pt]    {$A$};
\draw (477,412) node [anchor=north west][inner sep=0.75pt]    {$\mathbb{C}^d$};
\draw (477,505) node [anchor=north west][inner sep=0.75pt]    {$\mathbb{C}^d$};

\end{tikzpicture}
         \caption{}
         \label{fig:insufficient-kronecker}
     \end{subfigure}
     \hphantom{hello world}
     \begin{subfigure}[b]{0.2\textwidth}
         \centering

\tikzset{every picture/.style={line width=0.75pt}} 

\begin{tikzpicture}[x=0.75pt,y=0.75pt,yscale=-1,xscale=1]

\draw [color={rgb, 255:red, 74; green, 144; blue, 226 }  ,draw opacity=1 ][line width=1.5]    (225.33,655.51) -- (278.19,655.37) ;
\draw [shift={(281.19,655.36)}, rotate = 179.85] [color={rgb, 255:red, 74; green, 144; blue, 226 }  ,draw opacity=1 ][line width=1.5]    (14.21,-4.28) .. controls (9.04,-1.82) and (4.3,-0.39) .. (0,0) .. controls (4.3,0.39) and (9.04,1.82) .. (14.21,4.28)   ;
\draw [color={rgb, 255:red, 74; green, 144; blue, 226 }  ,draw opacity=1 ][line width=1.5]    (255.33,596.01) -- (228.44,648.97) ;
\draw [shift={(227.09,651.64)}, rotate = 296.92] [color={rgb, 255:red, 74; green, 144; blue, 226 }  ,draw opacity=1 ][line width=1.5]    (14.21,-4.28) .. controls (9.04,-1.82) and (4.3,-0.39) .. (0,0) .. controls (4.3,0.39) and (9.04,1.82) .. (14.21,4.28)   ;
\draw  [fill={rgb, 255:red, 144; green, 19; blue, 254 }  ,fill opacity=1 ] (221.54,653.84) .. controls (222.51,651.7) and (224.99,650.72) .. (227.09,651.64) .. controls (229.18,652.56) and (230.1,655.04) .. (229.13,657.18) .. controls (228.16,659.31) and (225.68,660.29) .. (223.58,659.37) .. controls (221.48,658.45) and (220.57,655.97) .. (221.54,653.84) -- cycle ;
\draw  [fill={rgb, 255:red, 144; green, 19; blue, 254 }  ,fill opacity=1 ] (281.19,655.36) .. controls (281.3,653.02) and (283.24,651.19) .. (285.53,651.27) .. controls (287.82,651.35) and (289.59,653.31) .. (289.48,655.65) .. controls (289.37,657.99) and (287.43,659.82) .. (285.14,659.74) .. controls (282.85,659.66) and (281.08,657.7) .. (281.19,655.36) -- cycle ;
\draw [color={rgb, 255:red, 74; green, 144; blue, 226 }  ,draw opacity=1 ][line width=1.5]    (255.33,596.01) -- (284.09,648.64) ;
\draw [shift={(285.53,651.27)}, rotate = 241.35] [color={rgb, 255:red, 74; green, 144; blue, 226 }  ,draw opacity=1 ][line width=1.5]    (14.21,-4.28) .. controls (9.04,-1.82) and (4.3,-0.39) .. (0,0) .. controls (4.3,0.39) and (9.04,1.82) .. (14.21,4.28)   ;
\draw  [fill={rgb, 255:red, 144; green, 19; blue, 254 }  ,fill opacity=1 ] (251.54,594.34) .. controls (252.51,592.2) and (254.99,591.22) .. (257.09,592.14) .. controls (259.18,593.06) and (260.1,595.54) .. (259.13,597.68) .. controls (258.16,599.81) and (255.68,600.79) .. (253.58,599.87) .. controls (251.48,598.95) and (250.57,596.47) .. (251.54,594.34) -- cycle ;

\draw (196,646) node [anchor=north west][inner sep=0.75pt]    {$\mathbb{C}^{d}$};
\draw (246,569.5) node [anchor=north west][inner sep=0.75pt]    {$\mathbb{C}^{d}$};
\draw (223.92,610) node [anchor=north west][inner sep=0.75pt]    {$A$};
\draw (272,610) node [anchor=north west][inner sep=0.75pt]    {$A$};
\draw (295,646) node [anchor=north west][inner sep=0.75pt]    {$\mathbb{C}^{d}$};
\draw (244.92,660.21) node [anchor=north west][inner sep=0.75pt]    {$A$};

\end{tikzpicture}

         \caption{}
         \label{fig:insufficient-triangle}
     \end{subfigure}
        \caption{Insufficiently generic quiver representations}
        \label{fig:insufficient-examples}
\end{figure}

%% file: figures/period-n.tex
\begin{figure}[htbp]
    \centering

\tikzset{every picture/.style={line width=0.75pt}} 

\begin{tikzpicture}[x=0.75pt,y=0.75pt,yscale=-1.3,xscale=1.3]

\draw [color={rgb, 255:red, 74; green, 144; blue, 226 }  ,draw opacity=1 ][line width=1.5]    (212.5,478.45) .. controls (214.07,482.08) and (275,452.45) .. (218.5,432.45) ;
\draw [color={rgb, 255:red, 74; green, 144; blue, 226 }  ,draw opacity=1 ][line width=1.5]    (165.5,425.45) .. controls (57.42,435.1) and (76.45,459.65) .. (96.79,477.09) ;
\draw [shift={(99,478.95)}, rotate = 219.81] [color={rgb, 255:red, 74; green, 144; blue, 226 }  ,draw opacity=1 ][line width=1.5]    (14.21,-4.28) .. controls (9.04,-1.82) and (4.3,-0.39) .. (0,0) .. controls (4.3,0.39) and (9.04,1.82) .. (14.21,4.28)   ;
\draw [color={rgb, 255:red, 74; green, 144; blue, 226 }  ,draw opacity=1 ][line width=1.5]    (165.5,425.45) .. controls (161.5,458.95) and (135,424.95) .. (91,441.45) ;
\draw [color={rgb, 255:red, 74; green, 144; blue, 226 }  ,draw opacity=1 ][line width=1.5]    (99,478.95) .. controls (115,503.45) and (125.5,495.95) .. (152.5,493.45) ;
\draw [color={rgb, 255:red, 74; green, 144; blue, 226 }  ,draw opacity=1 ][line width=1.5]    (99,478.95) .. controls (149,463.45) and (117.5,494.95) .. (152.5,493.45) ;
\draw [color={rgb, 255:red, 74; green, 144; blue, 226 }  ,draw opacity=1 ][line width=1.5]    (99,478.95) .. controls (117.62,493.16) and (155.45,499.21) .. (193.19,485.8) ;
\draw [shift={(195.5,484.95)}, rotate = 159.36] [color={rgb, 255:red, 74; green, 144; blue, 226 }  ,draw opacity=1 ][line width=1.5]    (14.21,-4.28) .. controls (9.04,-1.82) and (4.3,-0.39) .. (0,0) .. controls (4.3,0.39) and (9.04,1.82) .. (14.21,4.28)   ;
\draw [color={rgb, 255:red, 74; green, 144; blue, 226 }  ,draw opacity=1 ][line width=1.5]    (212.5,478.45) .. controls (230.76,469.48) and (273.86,427.48) .. (167.12,425.48) ;
\draw [shift={(165.5,425.45)}, rotate = 0.84] [color={rgb, 255:red, 74; green, 144; blue, 226 }  ,draw opacity=1 ][line width=1.5]    (14.21,-4.28) .. controls (9.04,-1.82) and (4.3,-0.39) .. (0,0) .. controls (4.3,0.39) and (9.04,1.82) .. (14.21,4.28)   ;
\draw [color={rgb, 255:red, 74; green, 144; blue, 226 }  ,draw opacity=1 ][line width=1.5]    (165.5,425.45) .. controls (140.89,425.72) and (145.5,410.45) .. (91,441.45) ;
\draw [color={rgb, 255:red, 74; green, 144; blue, 226 }  ,draw opacity=1 ][line width=1.5]    (218.5,432.45) .. controls (250.34,450.42) and (190,460.95) .. (212.5,478.45) ;
\draw  [fill={rgb, 255:red, 144; green, 19; blue, 254 }  ,fill opacity=1 ] (161.71,423.78) .. controls (162.67,421.65) and (165.16,420.67) .. (167.25,421.59) .. controls (169.35,422.51) and (170.26,424.99) .. (169.29,427.12) .. controls (168.33,429.26) and (165.84,430.24) .. (163.75,429.32) .. controls (161.65,428.39) and (160.74,425.92) .. (161.71,423.78) -- cycle ;
\draw  [fill={rgb, 255:red, 144; green, 19; blue, 254 }  ,fill opacity=1 ] (95.21,477.28) .. controls (96.17,475.15) and (98.66,474.17) .. (100.75,475.09) .. controls (102.85,476.01) and (103.76,478.49) .. (102.79,480.62) .. controls (101.83,482.76) and (99.34,483.74) .. (97.25,482.82) .. controls (95.15,481.89) and (94.24,479.42) .. (95.21,477.28) -- cycle ;

\draw (198,482) node [anchor=north west][inner sep=0.75pt]  [rotate=-359.27]  {$...$};
\draw (208.87,468) node [anchor=north west][inner sep=0.75pt]  [rotate=-359.27]  {$...$};
\draw (112,480) node [anchor=north west][inner sep=0.75pt]  [rotate=-359.27]  {$...$};
\draw (142,433) node [anchor=north west][inner sep=0.75pt]  [rotate=-359.27]  {$...$};
\draw (232.42,425) node [anchor=north west][inner sep=0.75pt]    {$T$};
\draw (139.42,477) node [anchor=north west][inner sep=0.75pt]    {$T$};
\draw (82.42,427) node [anchor=north west][inner sep=0.75pt]    {$T$};
\draw (158,405) node [anchor=north west][inner sep=0.75pt]    {$\mathbb{C}^{d}$};
\draw (76,480.91) node [anchor=north west][inner sep=0.75pt]    {$\mathbb{C}^{d}$};

\end{tikzpicture}
    \caption{A hyperquiver representing a period-$n$ orbit}
    \label{fig:period-n-representation}
\end{figure}

%% file: figures/homology-loop.tex
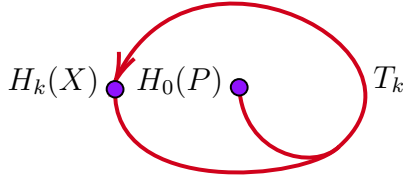
\begin{figure}[htbp]
    \centering
\tikzset{every picture/.style={line width=0.75pt}} 

\begin{tikzpicture}[x=0.75pt,y=0.75pt,yscale=-1,xscale=1]

\draw [color={rgb, 255:red, 208; green, 2; blue, 27 }  ,draw opacity=1 ][line width=1.5]    (382.82,596.23) .. controls (381,650.95) and (479.5,642.45) .. (496,624.95) ;
\draw [color={rgb, 255:red, 208; green, 2; blue, 27 }  ,draw opacity=1 ][line width=1.5]    (445.42,595.23) .. controls (447.5,629.45) and (483,637.95) .. (496,624.95) ;
\draw [color={rgb, 255:red, 208; green, 2; blue, 27 }  ,draw opacity=1 ][line width=1.5]    (496,624.95) .. controls (549.96,576.44) and (421.61,510.78) .. (384.08,589.66) ;
\draw [shift={(382.97,592.09)}, rotate = 293.56] [color={rgb, 255:red, 208; green, 2; blue, 27 }  ,draw opacity=1 ][line width=1.5]    (14.21,-4.28) .. controls (9.04,-1.82) and (4.3,-0.39) .. (0,0) .. controls (4.3,0.39) and (9.04,1.82) .. (14.21,4.28)   ;
\draw  [fill={rgb, 255:red, 144; green, 19; blue, 254 }  ,fill opacity=1 ] (382.97,592.09) .. controls (385.31,592.2) and (387.14,594.15) .. (387.06,596.44) .. controls (386.97,598.72) and (385.01,600.49) .. (382.67,600.38) .. controls (380.33,600.26) and (378.5,598.32) .. (378.58,596.03) .. controls (378.67,593.74) and (380.63,591.98) .. (382.97,592.09) -- cycle ;
\draw  [fill={rgb, 255:red, 144; green, 19; blue, 254 }  ,fill opacity=1 ] (445.57,591.09) .. controls (447.91,591.2) and (449.74,593.15) .. (449.66,595.44) .. controls (449.57,597.72) and (447.61,599.49) .. (445.27,599.38) .. controls (442.93,599.26) and (441.1,597.32) .. (441.18,595.03) .. controls (441.27,592.74) and (443.23,590.98) .. (445.57,591.09) -- cycle ;

\draw (327.1,582.91) node [anchor=north west][inner sep=0.75pt]    {$H_{k}( X)$};
\draw (511.7,582.91) node [anchor=north west][inner sep=0.75pt]    {$T_{k}$};
\draw (392,582.91) node [anchor=north west][inner sep=0.75pt]    {$H_{0}( P)$};

\end{tikzpicture}

    \caption{Fixed points in homology}
    \label{fig:homology-loop}
\end{figure}

%% file: A-appendix.tex
\appendix

\bigskip

\section{The Chow ring and Chern classes}
\label{sec:intersection_theory}

We recall the definitions of the Chow groups and Chow ring of a projective variety, following~\cite{eisenbud_harris_2016,fulton_1998}.

\begin{definition}
Let $X$ be a smooth projective variety of dimension $n$.
\begin{enumerate}[label=(\roman*)]
    \item {\cite[Section 1.2.1]{eisenbud_harris_2016}} The group of $i$\textit{-cycles} of $X$ is the free abelian group $Z_i(X)$ generated by the irreducible $i$-dimensional subvarieties of $X$. An element of $Z_i(X)$, called an $i$\textit{-cycle}, is a finite, formal sum $\sum_i n_i V_i$ of $i$-dimensional subvarieties $V_i$ of $X$, where $n_i \in \mathbb{Z}$.
    \item {\cite[Proposition 1.6]{fulton_1998}} An $i$-cycle $Z \in Z_i(X)$ is \textit{rationally equivalent to zero} if there exist irreducible subvarieties $V_i \subseteq \mathbb{P}^1 \times X$ of dimension $i+1$ with dominant projection maps $V_i \rightarrow \mathbb{P}^1$ such that $Z = \sum_i V_i(0) - V_i(\infty)$, where $V_i(t) = V_i \cap (\{t\} \times X)$. The $i$-cycles rationally equivalent to zero form a subgroup $\mathrm{Rat}_i(X)$ of $Z_i(X)$.
    \item {\cite[Section 1.2.2-1.2.3]{fulton_1998}} The $i$\textit{-th Chow group} of $X$ is the quotient group $A_i(X) = Z_i(X)/\mathrm{Rat}_i(X)$. The class of an $i$-cycle $C \in Z_i(X)$ in $A_i(X)$ is denoted by $[C]$. The \textit{Chow group} of $X$ is the direct sum $A_{*}(X) = \oplus_{i = 0}^n A_i(X)$. The \textit{Chow ring} of $X$ is the direct sum $A^{*}(X) = \oplus_{i = 0}^n A^i(X)$, where $A^i(X) = A_{n-i}(X)$.
\end{enumerate}
\end{definition}

The Chow ring $A^{*}(X)$ has the structure of a commutative ring, with a product $A^i(X) \times A^j(X) \rightarrow A^{i+j}(X)$ called the \textit{intersection product}. 
We say that $C$ and $D$ \textit{intersect transversely} if on each component of $C \cap D$ at a generic point $p$, the sum of the tangent spaces of $C$ and $D$ is the tangent space of $X$: $T_p C + T_p D = T_p X$. 
The intersection product takes any codimension-$i$ and codimension-$j$ irreducible subvarieties $C,D \subseteq X$, replaces $C$ and $D$ by rationally equivalent subvarieties $C^\prime, D^\prime \subseteq X$ (if necessary) in order for $C^\prime$ and $D^\prime$ to intersect transversely, and defines $[C][D] = [C^\prime \cap D^\prime] \in A^{i+j}(X)$. 
The existence of a well-defined intersection product is due to Fulton~\cite{fulton_1998}; see~\cite[Appendix A]{eisenbud_harris_2016}.

\begin{definition}[{\label{pushforwardPullback}}{\cite[Section 1.3.6]{eisenbud_harris_2016}}]
Let $X$ and $Y$ be smooth projective varieties of dimensions $m$ and $n$, and $f: X \rightarrow Y$ a morphism. 
\begin{enumerate}[label=(\roman*)]
    \item Let $V \subseteq X$ be an irreducible subvariety of dimension $i$. Define a group homomorphism $f_*: A_i(X) \rightarrow A_i(Y)$ by 
    \[ [V] \mapsto \begin{cases} d \cdot [f(V)] & \dim f(V) = i \\ 
    0 & \dim f(V) < i ,
    \end{cases}  \]
    where $d:=[R(V):R(f(V))]$ is the degree of the field extension between the function fields $R(V)$ of $V$ and $R(f(V))$ of $f(V)$. The map $f_*$ extends to a group homomorphism $f_*: A_*(X) \rightarrow A_*(Y)$, called the \textit{pushforward} of $f$.
    \item There is a unique group homomorphism $f^*: A^i(Y) \rightarrow A^i(X)$ such that for all $W \subseteq Y$ a smooth subvariety with $i = \codim _Y W = \codim _X(f^{-1}(W))$, we have $f^*([W]) = [f^{-1}(W)]$. This extends to a ring homomorphism $f^*: A^*(Y) \rightarrow A^*(X)$ called the \textit{pullback} of $f$.
\end{enumerate}
\end{definition}

\begin{remark}{\label{rem:closed-immersion-pushforward}}
The degree of the field extension in the definition of $f_*$ is the degree of the covering of $f(V)$ by $V$. In particular, if $i: X \rightarrow Y$ is a closed immersion, then $i_*([X]) = [X]$.
\end{remark}

\begin{proposition}[Projection Formula, {\cite[Theorem 1.23(b)]{eisenbud_harris_2016}}]{\label{projectionFormula}}
If $X$ and $Y$ are smooth projective varieties,
\noindent $f:X\rightarrow Y$ is a morphism, and $[C] \in A_i(X)$ and $[D] \in A^j(Y)$ are cycle classes, then
$$[D] f_*([C]) = f_*(f^*([D]) [C]) \in A_{i-j}(Y).$$
\end{proposition}

\begin{definition}[{\cite[Theorem 5.3]{eisenbud_harris_2016}}{\label{def:chern-classes}}]
Let $X$ be a smooth projective variety of dimension $n$ and let $\mathscr{B}$ be a vector bundle over $X$. There exist unique classes $c_i(\mathscr{B}) \in A^i(X)$ for $i \in [n]$ called the \textit{Chern classes} of $\mathscr{B}$, depending only on the isomorphism class of $\mathscr{B}$, satisfying the following axioms:
\begin{enumerate}[label=(\roman*)]
    \item If $r$ is the rank of $\mathscr{B}$, then $c_i(\mathscr{B}) = 0$ for all $i > r$.
    \item If $\sigma_0,...,\sigma_{r-i} \in \Gamma(\mathscr{B})$ are global sections and their degeneracy locus $Z(\sigma_0,...,\sigma_{r-i}) \subseteq X$ has codimension $i$ in $X$, then $c_i(\mathscr{B}) = [Z(\sigma_0,...,\sigma_{r-i})]$.
    \item The \textit{Chern polynomial} of $\mathscr{B}$ is $C(t,\mathscr{B}) = 1 + \sum_{i = 1}^r c_i(\mathscr{B})t^i$. If $0 \rightarrow \mathscr{B} \rightarrow \mathscr{B}^\prime \rightarrow \mathscr{B}^{\prime\prime} \rightarrow 0$ is an exact sequence of vector bundles over $X$, then \[C(t,\mathscr{B}^\prime) = C(t,\mathscr{B})C(t,\mathscr{B}^{\prime\prime}).\]
    \item If $Y$ is a smooth projective variety and $f: Y \rightarrow X$ a morphism, 
    \noindent ${c_i(f^*\mathscr{B})=f^*(c_i(\mathscr{B}))}$.
    \item If $\mathscr{L}$ is a line bundle on $X$, then $C(t,\mathscr{L}) = 1 + c_1(\mathscr{L})t$, where $c_1(\mathscr{L}) \in A^1(X)$ is the class of the divisor of zeros minus the divisor of poles of any section of $\mathscr{L}$ defined on a Zariski-open set of $X$
\end{enumerate}
If $r=n$, then $c_n(\mathscr{B}) \in A^n(X)$ so $c_n(\mathscr{B}) = \nu(\mathscr{B})[p]$ for some integer $\nu(\mathscr{B})$ called the \textit{top Chern number} of $\mathscr{B}$, where $[p] \in A^n(X)$ is the class of a point $p \in X$.
\end{definition}

\begin{remark}{\label{rem:chern-class-axioms}} Compared with \cite[Theorem 5.3]{eisenbud_harris_2016}, we have added the redundant axiom (i) to our Definition \ref{def:chern-classes}(a) in order to help clarify the properties of Chern classes. The result \cite[Theorem 5.3]{eisenbud_harris_2016} is fully proven in \cite[Chapter 3]{fulton_1998}.
\end{remark}

\begin{definition}[{\cite[Remark 3.2.3, Example 3.2.3]{fulton_1998}}]{\label{def:chern-roots}}
The \textit{Chern roots} of $\mathscr{B}$ are the formal variables $\xi_i(\mathscr{B})$ in the formal factorization of the Chern polynomial:
$$C(t,\mathscr{B}) = \prod_{i = 1}^r (1 + \xi_i(\mathscr{B})t).$$
The \textit{Chern character} of $\mathscr{B}$ is 
$\ch(\mathscr{B}) = \sum_{i = 1}^r \exp (\xi_j(\mathscr{B}))$,
where $\exp (\alpha) = \sum_{k = 0}^\infty \frac{1}{k!}\alpha^k$ is a formal sum in the formal variable $\alpha$.
\end{definition}

From Definitions~\ref{def:chern-classes}~and~\ref{def:chern-roots}, one can obtain the following properties.

\begin{proposition}[{\cite[Remark 3.2.3, Example 3.2.3]{fulton_1998}}]{\label{chernProperties}}
Let $X$ be a smooth projective variety and $\mathscr{B}$ and $\mathscr{B}^\prime$ be vector bundles over $X$.
\begin{enumerate}[(a)]
    \item If $\mathscr{B}$ is the trivial bundle, then $C(t,\mathscr{B}) = 1$.
    \item The Chern polynomial of $\mathscr{B}$ and its dual are related by $C(t,\mathscr{B}^{\vee}) = C(-t,\mathscr{B})$.
    \item The Chern character satisfies $\ch(\mathscr{B} \otimes \mathscr{B}^\prime) = \ch(\mathscr{B})\ch(\mathscr{B}^\prime)$.
\end{enumerate}
\end{proposition}

%% file: main.bbl
\begin{thebibliography}{10}

\bibitem{ABO2016121}
Hirotachi Abo, David Eklund, Thomas Kahle, and Chris Peterson.
\newblock Eigenschemes and the {J}ordan canonical form.
\newblock {\em Linear Algebra and its Applications}, 496:121--151, 2016.

\bibitem{Abo_Portakal_Sodomaco_2023}
Hirotachi Abo, İrem Portakal, and Luca Sodomaco.
\newblock A vector bundle approach to nash equilibria.
\newblock {\em In preparation}, 2023.

\bibitem{eigenconfiguration}
Hirotachi {Abo}, Anna {Seigal}, and Bernd {Sturmfels}.
\newblock {Eigenconfigurations of Tensors}.
\newblock {\em arXiv e-prints}, page arXiv:1505.05729, May 2015.

\bibitem{anandkumar2014tensor}
Animashree Anandkumar, Rong Ge, Daniel Hsu, Sham~M Kakade, and Matus Telgarsky.
\newblock Tensor decompositions for learning latent variable models.
\newblock {\em Journal of machine learning research}, 15:2773--2832, 2014.

\bibitem{auddy2020perturbation}
Arnab Auddy and Ming Yuan.
\newblock Perturbation bounds for orthogonally decomposable tensors and their
  applications in high dimensional data analysis.
\newblock {\em arXiv preprint arXiv:2007.09024}, 2020.

\bibitem{benson2019three}
Austin~R Benson.
\newblock Three hypergraph eigenvector centralities.
\newblock {\em SIAM Journal on Mathematics of Data Science}, 1(2):293--312,
  2019.

\bibitem{eigenschemes}
Valentina Beorchia, Francesco Galuppi, and Lorenzo Venturello.
\newblock Eigenschemes of ternary tensors.
\newblock {\em SIAM Journal on Applied Algebra and Geometry}, 5(4):620--650,
  2021.

\bibitem{bernstein1973coxeter}
IN~Bernstein, Israil’M Gel'fand, and Vladimir~A Ponomarev.
\newblock Coxeter functors and {G}abriel's theorem.
\newblock {\em Russian mathematical surveys}, 28(2):17, 1973.

\bibitem{bodnar2021weisfeiler}
Cristian Bodnar, Fabrizio Frasca, Yuguang Wang, Nina Otter, Guido~F Montufar,
  Pietro Lio, and Michael Bronstein.
\newblock Weisfeiler and {L}ehman go topological: Message passing simplicial
  networks.
\newblock In {\em International Conference on Machine Learning}, pages
  1026--1037. PMLR, 2021.

\bibitem{breiding-compressed}
Paul Breiding, Fulvio Gesmundo, Mateusz Michałek, and Nick Vannieuwenhoven.
\newblock Algebraic compressed sensing.
\newblock {\em Applied and Computational Harmonic Analysis}, 65:374--406, 2023.

\bibitem{condition-book}
Peter Brgisser and Felipe Cucker.
\newblock {\em Condition: The Geometry of Numerical Algorithms}.
\newblock Springer Publishing Company, Incorporated, 2013.

\bibitem{cartwright2013number}
Dustin Cartwright and Bernd Sturmfels.
\newblock The number of eigenvalues of a tensor.
\newblock {\em Linear Algebra and Its Applications}, 438(2):942--952, 2013.

\bibitem{catanese2006maximum}
Fabrizio Catanese, Serkan Ho{\c{s}}ten, Amit Khetan, and Bernd Sturmfels.
\newblock The maximum likelihood degree.
\newblock {\em American Journal of Mathematics}, 128(3):671--697, 2006.

\bibitem{hypergraph2}
Jingya Chang, Yannan Chen, Liqun Qi, and Hong Yan.
\newblock Hypergraph clustering using a new laplacian tensor with applications
  in image processing.
\newblock {\em SIAM Journal on Imaging Sciences}, 13(3):1157--1178, 2020.

\bibitem{hypergraph1}
Yannan Chen, Liqun Qi, and Xiaoyan Zhang.
\newblock The fiedler vector of a laplacian tensor for hypergraph partitioning.
\newblock {\em SIAM Journal on Scientific Computing}, 39(6):A2508--A2537, 2017.

\bibitem{chen2017supervised}
Zhengdao Chen, Xiang Li, and Joan Bruna.
\newblock Supervised community detection with line graph neural networks.
\newblock {\em arXiv preprint arXiv:1705.08415}, 2017.

\bibitem{chu-inverse-1998}
Moody~T. Chu.
\newblock Inverse eigenvalue problems.
\newblock {\em SIAM Review}, 40(1):1--39, 1998.

\bibitem{Chu_Golub_2002}
Moody~T. Chu and Gene~H. Golub.
\newblock Structured inverse eigenvalue problems.
\newblock {\em Acta Numerica}, 11:1–71, 2002.

\bibitem{derksen_introduction_2017}
Harm Derksen and Jerzy Weyman.
\newblock {\em An introduction to quiver representations}.
\newblock Number 184 in Graduate studies in mathematics. American Mathematical
  Society, Providence, Rhode Island, 2017.

\bibitem{tensor-eigenvalue-problem}
Weiyang Ding and Yimin Wei.
\newblock Generalized tensor eigenvalue problems.
\newblock {\em SIAM Journal on Matrix Analysis and Applications},
  36:1073--1099, 01 2015.

\bibitem{draisma2016euclidean}
Jan Draisma, Emil Horobe{\c{t}}, Giorgio Ottaviani, Bernd Sturmfels, and
  Rekha~R Thomas.
\newblock The euclidean distance degree of an algebraic variety.
\newblock {\em Foundations of computational mathematics}, 16:99--149, 2016.

\bibitem{draisma2018best}
Jan Draisma, Giorgio Ottaviani, and Alicia Tocino.
\newblock Best rank-k approximations for tensors: generalizing
  {E}ckart--{Y}oung.
\newblock {\em Research in the Mathematical Sciences}, 5(2):27, 2018.

\bibitem{eisenbud_harris_2016}
David Eisenbud and Joe Harris.
\newblock {\em 3264 and All That: A Second Course in Algebraic Geometry}.
\newblock Cambridge University Press, 2016.

\bibitem{el2021implicit}
Laurent El~Ghaoui, Fangda Gu, Bertrand Travacca, Armin Askari, and Alicia Tsai.
\newblock Implicit deep learning.
\newblock {\em SIAM Journal on Mathematics of Data Science}, 3(3):930--958,
  2021.

\bibitem{complex-dynamics-fornaess-sibony}
John~Erik Fornaess and Nessim Sibony.
\newblock Complex dynamics in higher dimension. {I}.
\newblock In Camacho C., Lins~Neto A., Moussu R., and Sad P., editors, {\em
  Complex analytic methods in dynamical systems - IMPA, January 1992}, number
  222 in Ast\'erisque. Soci\'et\'e math\'ematique de France, 1994.

\bibitem{fornaess1995complex}
John~Erik Fornaess and Nessim Sibony.
\newblock Complex dynamics in higher dimension. ii.
\newblock {\em Modern methods in complex analysis (Princeton, NJ, 1992)},
  137:135--182, 1995.

\bibitem{friedland-ottaviani}
Shmuel Friedland and Giorgio Ottaviani.
\newblock The number of singular vector tuples and uniqueness of best rank-one
  approximation of tensors.
\newblock {\em Found. Comput. Math.}, 14(6):1209–1242, dec 2014.

\bibitem{fulton_1998}
William Fulton.
\newblock {\em Intersection theory}.
\newblock Springer, 1998.

\bibitem{gabriel}
Peter Gabriel.
\newblock Unzerlegbare {D}arstellungen {I}.
\newblock {\em Manuscripta mathematica}, 6:71--104, 1972.

\bibitem{GGTV23}
Francesco Galuppi, Fulvio Gesmundo, Ettore~Teixeira Turatti, and Lorenzo
  Venturello.
\newblock Characteristic polynomials and eigenvalues of tensors.
\newblock {\em arXiv.org}, 2023.

\bibitem{gilmer2017neural}
Justin Gilmer, Samuel~S Schoenholz, Patrick~F Riley, Oriol Vinyals, and
  George~E Dahl.
\newblock Neural message passing for quantum chemistry.
\newblock In {\em International conference on machine learning}, pages
  1263--1272. PMLR, 2017.

\bibitem{gu2020implicit}
Fangda Gu, Heng Chang, Wenwu Zhu, Somayeh Sojoudi, and Laurent El~Ghaoui.
\newblock Implicit graph neural networks.
\newblock {\em Advances in Neural Information Processing Systems},
  33:11984--11995, 2020.

\bibitem{hamilton2020graph}
William~L Hamilton.
\newblock Graph representation learning.
\newblock {\em Synthesis Lectures on Artifical Intelligence and Machine
  Learning}, 14(3):1--159, 2020.

\bibitem{Hartshorne}
Robin Hartshorne.
\newblock {\em Algebraic Geometry}.
\newblock Springer-Verlag, New York, 1977.
\newblock Graduate Texts in Mathematics, No. 52.

\bibitem{hua2022high}
Chenqing Hua, Guillaume Rabusseau, and Jian Tang.
\newblock High-order pooling for graph neural networks with tensor
  decomposition.
\newblock {\em arXiv preprint arXiv:2205.11691}, 2022.

\bibitem{shafarevich}
Shafarevich Igor~R.
\newblock {\em Basic Algebraic Geometry 1 : Varieties in Projective Space}.
\newblock Number 1, Varieties in projective space in Basic Algebraic Geometry.
  Springer, 2013.

\bibitem{hypergraph4}
Chuyang Ke and Jean Honorio.
\newblock Exact inference in high-order structured prediction.
\newblock In {\em Proceedings of the 40th International Conference on Machine
  Learning}, ICML'23. JMLR.org, 2023.

\bibitem{kempf_1993}
G.~Kempf.
\newblock {\em Algebraic Varieties}.
\newblock London Mathematical Society Lecture Note Series. Cambridge University
  Press, 1993.

\bibitem{Li_Markov_Chain_Inverse}
Chi-Kwong Li and Shixiao Zhang.
\newblock Stationary probability vectors of higher-order markov chains.
\newblock {\em Linear Algebra and its Applications}, 473:114--125, 2015.
\newblock Special issue on Statistics.

\bibitem{yokohama2}
Mao-lin Liang and Bing Zheng.
\newblock Gradient-based iterative algorithms for the tensor nearness problems
  associated with sylvester tensor equations.
\newblock {\em Communications in Mathematical Sciences}, 19:2275--2290, 01
  2021.

\bibitem{LDZ21}
Maolin Liang, Lifang Dai, and Ruijuan Zhao.
\newblock The inverse eigenvalue problems of tensors with an application in
  tensor nearness problems.
\newblock {\em Pacific Journal of Optimization}, 17(3):399--413, 2021.

\bibitem{lim2005singular}
Lek-Heng Lim.
\newblock Singular values and eigenvalues of tensors: a variational approach.
\newblock In {\em 1st IEEE International Workshop on Computational Advances in
  Multi-Sensor Adaptive Processing, 2005.}, pages 129--132. IEEE, 2005.

\bibitem{hypergraph3}
Tianhang Liu and Yimin Wei.
\newblock The abstract laplacian tensor of a hypergraph with applications in
  clustering.
\newblock {\em J. Sci. Comput.}, 93(1), October 2022.

\bibitem{OSV21}
Giorgio Ottaviani, Luca Sodomaco, and Emanuele Ventura.
\newblock Asymptotics of degrees and ed degrees of segre products.
\newblock {\em Advances in Applied Mathematics}, 130:102242--, 2021.

\bibitem{oudot_persistence_2015}
Steve~Y. Oudot.
\newblock {\em Persistence theory: from quiver representations to data
  analysis}.
\newblock Number volume 209 in Mathematical surveys and monographs. American
  Mathematical Society, Providence, Rhode Island, 2015.

\bibitem{Pan16}
Jay Pantone.
\newblock The asymptotic number of simple singular vector tuples of a cubical
  tensor.
\newblock {\em Online Journal of Analytic Combinatorics}, 05 2016.

\bibitem{qi2005eigenvalues}
Liqun Qi.
\newblock Eigenvalues of a real supersymmetric tensor.
\newblock {\em Journal of Symbolic Computation}, 40(6):1302--1324, 2005.

\bibitem{qi2017tensor}
Liqun Qi and Ziyan Luo.
\newblock {\em Tensor analysis: spectral theory and special tensors}.
\newblock SIAM, 2017.

\bibitem{robeva2016orthogonal}
Elina Robeva.
\newblock Orthogonal decomposition of symmetric tensors.
\newblock {\em SIAM Journal on Matrix Analysis and Applications},
  37(1):86--102, 2016.

\bibitem{robeva2017singular}
Elina Robeva and Anna Seigal.
\newblock Singular vectors of orthogonally decomposable tensors.
\newblock {\em Linear and Multilinear Algebra}, 65(12):2457--2471, 2017.

\bibitem{seigal2022principal}
Anna Seigal, Heather~A Harrington, and Vidit Nanda.
\newblock Principal components along quiver representations.
\newblock {\em Foundations of Computational Mathematics}, pages 1--37, 2022.

\bibitem{sodomaco2022span}
Luca Sodomaco and Ettore~Teixeira Turatti.
\newblock The span of singular tuples of a tensor beyond the boundary format.
\newblock {\em arXiv preprint arXiv:2206.08606}, 2022.

\bibitem{spanier}
E~H Spanier.
\newblock {\em Algebraic Topology}.
\newblock McGraw-Hill, 1966.

\bibitem{turatti2022tensors}
Ettore Turatti.
\newblock On tensors that are determined by their singular tuples.
\newblock {\em SIAM Journal on Applied Algebra and Geometry}, 6(2):319--338,
  2022.

\bibitem{YH17}
Ke~Ye and Shenglong Hu.
\newblock Inverse eigenvalue problem for tensors.
\newblock {\em Communications in Mathematical Sciences}, 15:1627--1649, 01
  2017.

\bibitem{zayed-inverse}
E.~M.~E. Zayed.
\newblock An inverse eigenvalue problem for an arbitrary, multiply connected,
  bounded domain in $r^3$ with impedance boundary conditions.
\newblock {\em SIAM Journal on Applied Mathematics}, 52(3):725--729, 1992.

\end{thebibliography}
